\documentclass[a4paper, 10pt]{amsart}

\usepackage{amsmath,amssymb,amscd}

\usepackage{amsfonts}
\usepackage[english]{babel}
\usepackage[leqno]{amsmath}
\usepackage{amssymb,amsthm}
\usepackage{amscd}
\usepackage{enumerate}
\usepackage{epsfig}

\usepackage{layout}
\usepackage{fullpage}

\usepackage[usenames,dvipsnames]{color}

\usepackage[backref=page]{hyperref}

\hypersetup{
 colorlinks,
 citecolor=Green,
 linkcolor=blue,
 urlcolor=Blue}

\usepackage[matrix,arrow,tips,curve]{xy}

\input{xy}
\xyoption{all}

\newenvironment{sis}{\left\{\begin{aligned}}{\end{aligned}\right.}

\newtheorem{thm}{Theorem}[section]
\newtheorem{lemma}[thm]{Lemma}
\newtheorem{lemdef}[thm]{Lemma-Definition}
\newtheorem{prop}[thm]{Proposition}
\newtheorem{cor}[thm]{Corollary}

\newtheorem{conj}[thm]{Conjecture}
\newtheorem{fact}[thm]{Fact}

\newtheorem*{thmA}{Theorem A}
\newtheorem*{thmC}{Theorem C}
\newtheorem*{thmD}{Theorem D}
\newtheorem*{thmE}{Theorem E}
\newtheorem*{corB}{Corollary B}

\numberwithin{equation}{section}
\theoremstyle{definition}
\newtheorem{defi}[thm]{Definition}
\newtheorem{convention}[thm]{}

\theoremstyle{remark}
\newtheorem{remark}[thm]{Remark}

\newcommand{\Z}{\mathbb{Z}}
\newcommand{\Q}{\mathbb{Q}}

\newcommand{\Pic}{\operatorname{Pic}}

\newcommand{\un}{\underline}
\newcommand{\ov}{\overline}
\newcommand{\wt}{\widetilde}
\newcommand{\wh}{\widehat}

\DeclareMathOperator{\Hom}{{Hom}}
\DeclareMathOperator{\Spec}{Spec \:}
\DeclareMathOperator{\Spf}{Spf \:}
\DeclareMathOperator{\Def}{Def}

\def \Im{{\rm Im}}
\DeclareMathOperator{\Supp}{Supp}
\DeclareMathOperator{\ann}{ann}
\DeclareMathOperator{\Ass}{Ass}
\DeclareMathOperator{\id}{id}

\def \PP{\mathbb{P}}

\def \Gm{{\mathbb G}_m}

\def \GL{{\rm GL}}

\def\J{\overline J}

\def \P{\mathcal P}
\def \F{\mathcal F}
\def \N {\mathcal N}
\def \X{\mathcal X}
\def \Y{\mathcal Y}
\def\I{\mathcal I}
\def \L{\mathcal L}

\def\O{\mathcal O}
\def\A{\mathcal A}

\def \D{\mathcal D}

\def\M0{\mathcal M^0}

\def \Tr{{\rm Tr}}

\def\bJ{{\mathbb J}}
\def\bJbar{\ov{{\mathbb J}}}

\def\calI{{\mathcal I}}
\def\calM{{\mathcal M}}

\def \M{\mathcal M}

\def\m {\mathfrak{m}}

\def \Xsm{X_{\rm sm}}

\def \Pun{\P^{\rm un}}

\def \beun{\beta^{\rm un}_{\un q}}

\def \HHom{\mathcal Hom}
\def \EExt{\mathcal Ext}

\def \H {\mathcal H}

\newcommand{\lotimes}{{\,\stackrel{\mathbf L}{\otimes}\,}}

\DeclareMathOperator{\codim}{{codim}}

\DeclareMathOperator{\supp}{{supp}}
\newcommand{\bbR}{{\mathbb R}}
\newcommand{\bbN}{{\mathbb N}}
\newcommand{\bbC}{{\mathbb C}}

\newcommand{\bbZ}{{\mathbb Z}}

\newcommand{\bbA}{{\mathbb A}}
\newcommand{\cC}{{\mathcal C}}

\newcommand{\cI}{{\mathcal I}}

\newcommand{\cU}{{\mathcal U}}
\newcommand{\cV}{{\mathcal V}}
\newcommand{\cY}{{\mathcal Y}}
\newcommand{\cZ}{{\mathcal Z}}
\newcommand{\cplx}[1]{{{\mathcal #1}^{\scriptscriptstyle\bullet}}}

\newcommand{\rk}{\operatorname{rk}}

\title{Fourier-Mukai and autoduality for compactified Jacobians. I}

\author[Melo]{Margarida Melo}
\address{Dipartimento di Matematica, Universit\`a Roma Tre, Largo S. Leonardo Murialdo 1, 00146 Roma (Italy)}
\email{melo@mat.uniroma3.it}

\author[Rapagnetta]{Antonio Rapagnetta}
\address{Dipartimento di Matematica, Universit\`a di Roma II - Tor Vergata, 00133 Roma (Italy)}
\email{rapagnet@mat.uniroma2.it}

\author[Viviani]{Filippo Viviani}
\address{Dipartimento di Matematica, Universit\`a Roma Tre, Largo S. Leonardo Murialdo 1, 00146 Roma (Italy)}
\email{viviani@mat.uniroma3.it}

\begin{document}


%\subjclass[2010]{14L24, 14H10, 14H40, 14H20, 14C05, 14D23.}

\begin{abstract}

To every singular reduced projective curve $X$ one can associate, following E. Esteves, many fine compactified Jacobians, depending on the choice of a polarization on $X$, each of which yields a modular compactification of a disjoint union of the generalized Jacobian of $X$. We prove that, for a reduced curve with locally planar singularities, the integral (or Fourier-Mukai) transform with kernel the Poincar\'e sheaf from the derived category of the generalized Jacobian of $X$ to the derived category of any fine compactified Jacobian of $X$ is fully faithful, generalizing a previous result of D. Arinkin in the case of integral curves. As a consequence, we prove that there is a canonical isomorphism (called autoduality) between the generalized Jacobian of $X$ and the connected component of the identity of the Picard scheme of any fine compactified Jacobian of $X$ and that  algebraic equivalence and numerical equivalence of line bundles coincide on any fine compactified Jacobian, generalizing previous results of Arinkin, Esteves, Gagn\'e, Kleiman, Rocha, Sawon.

The paper contains an Appendix in which we explain how our work can be interpreted in view of the Langlands duality for the Higgs bundles  as proposed by Donagi-Pantev.
 %the second one (due to Ana Cristina  L{\'o}pez Mart{\'{\i}}n) contains a fully faithful criterion for Fourier-Mukai transforms that we use in the proof of our main theorem.

\end{abstract}

\maketitle

\tableofcontents

\section{Introduction}

Let $C$ be a smooth irreducible projective curve over an algebraically closed field $k$
%(of arbitrary characteristic)
and let $J(C)$ be its Jacobian variety. Then $J(C)$ is an abelian variety that carries lots of information about the curve itself. Among abelian varieties, Jacobians  have the important property of being ``autodual'', i.e., they are canonically isomorphic to their dual abelian varieties. This is equivalent to the existence of a Poincar\'e line bundle $\P$ on $J(C)\times J(C)$ which is universal as a family of algebraically trivial line bundles on $J(C)$. In the breakthrough work \cite{mukai}, S. Mukai proved that the Fourier-Mukai transform with kernel $\P$ is an auto-equivalence of the bounded derived category of  $J(C)$ \footnote{More generally, for an arbitrary abelian variety $A$ with dual abelian variety $A^{\vee}$, Mukai proved that the Fourier-Mukai transform associated to the Poincar\'e line bundle on $A\times A^{\vee}$ gives an equivalence  between the bounded derived category of $A$ and that of $A^{\vee}$.}.

\vspace{0,1cm}

The aim of this paper and its sequel \cite{MRV2}, which are strongly based on our previous paper \cite{MRV1},   is to extend these results to fine compactified Jacobians (as defined by E. Esteves in \cite{est1}) of {\em reduced projective curves with locally planar singularities.} The case of integral (i.e. reduced and irreducible) projective curves with locally planar singularities was dealt with by D. Arinkin in \cite{arin1} and \cite{arin2}, generalizing previous partial results of Esteves-Gagn\'e-Kleiman \cite{egk}, Esteves-Kleiman   \cite{EK} and Sawon \cite{Saw} for integral projective curves with double point singularities. Moreover, the autoduality result has been recently extended by Esteves-Rocha \cite{ER} to tree-like curves, i.e. curves with locally planar singularities such that the unique singular points lying in more than one irreducible component are separating nodes (e.g. nodal curves of compact type). Finally, while this paper was under the referee process, two related papers have appeared on arXiv: D. Arinkin and R. Fedorov established in \cite{AF} a partial Fourier-Mukai transform for degenerate abelian schemes (in characteristic zero); J. L. Kass proved in \cite{Kas4} that  autoduality  holds true for (possibly coarse) compactified Jacobians of nodal curves and stable quasiabelian varieties (in characteristic zero).

%Indeed, the two papers of Arinkin (\cite{arin1} and \cite{arin2}) were our main source of inspiration although the case of reducible curves present several extra-difficulties that we have to overcome; see the discussion that follows.

\vspace{0,1cm}

The main motivation for this work
%(as well as for the work of Arinkin \cite{arin1} and \cite{arin2})
comes from the Langlands duality conjecture for Hitchin systems  proposed by Donagi-Pantev in \cite{DP} as a classical limit of the conjectural geometric Langlands correspondence (which we review in more details in the Appendix ).
In the special case of the Langlands self-dual linear group $\GL_r$, the  Langlands duality conjecture predicts an autoequivalence $\Phi:D^b(\M)\stackrel{\cong}{\longrightarrow} D^b(\M)$ of the bounded derived category of quasi-coherent sheaves of the moduli stack $\M$ of Higgs bundles of rank $r$ on a fixed smooth projective curve $C$, which should intertwine the classical limit tensorization functors with the classical limit  Hecke functors (see \cite[Sec. 2]{DP} for more details). The moduli stack $\M$ of Higgs bundles admits a morphism $\H:\M\to \A$, called the Hitchin morphism, to an affine space $\A$ parametrizing certain degree-$r$ singular covers of $C$, called spectral curves (see \eqref{E:fibr-Hitchin}). According to the so-called spectral correspondence (see Fact \ref{F:corr-spec}), the fiber of $\H^{-1}(\wt{C})$ over a given spectral curve $\wt{C}\to C$ is the stack of rank-$1$ torsion-free sheaves on $\wt{C}$, which, for $\wt{C}$ reduced, contains any fine compactified Jacobian of $\wt{C}$ as an open and projective subscheme.
The autoequivalence $\Phi$ is expected to be given by a Fourier-Mukai transform with kernel equal to a universal Poincar\'e sheaf $\P$ on $\M\times_{\A}\M$. Moreover, $\Phi$ is expected to preserve the Hitchin  morphism $\H$, i.e. for any spectral curve $\wt{C}\to C$ the restriction $\P_{\wt{C}}$ of $\P$ to $\H^{-1}(\wt{C})\times \H^{-1}(\wt{C})$ should induce a  Fourier-Mukai autoequivalence ${\Phi}^{\P_{\wt{C}}}:D^b(\H^{-1}(\wt{C}))\stackrel{\cong}{\longrightarrow} D^b(\H^{-1}(\wt{C}))$.
Theorem E below (which we will prove in a sequel \cite{MRV2} to this paper) can be regarded as a first step toward proving the Langlands duality conjecture over the open subset of reduced spectral curves, thus extending the work of Donagi-Pantev \cite{DP} for smooth spectral curves  and the work of Arinkin \cite{arin2} for integral spectral curves.

\vspace{0,1cm}

Before stating our results, we need to briefly recall how fine compactified Jacobians of singular curves are defined.

\subsection{Fine compactified Jacobians of singular curves}

Let $X$ be a {\em reduced} projective connected curve. The {\bf {\em generalized Jacobian}} $J(X)$ of $X$ is the connected component of the Picard scheme of $X$ containing the identity. It is not difficult to see that $J(X)$ is a smooth irreducible algebraic group of dimension equal to the arithmetic genus $p_a(X)$ of $X$, parametrizing line bundles on $X$ that have multidegree zero, i.e. degree zero on each irreducible component of $X$. However, for a singular curve $X$, the generalized Jacobian $J(X)$ is rarely complete. The problem of compactifying it is very natural and it has attracted the attention of many mathematicians, starting from the pioneering work of Mayer-Mumford and of  Igusa in the 50's, till the more recent works of Oda-Seshadri, Altmann-Kleiman, Caporaso, Pandharipande, Simpson, Jarvis, Esteves, etc.. (we refer to the introduction of \cite{est1} for an  account of the different approaches).

%Of course, when dealing with singular curves instead of dealing with the Jacobian of $X$ one should deal with the compactified Jacobian of $X$.
%when we consider singular curves, the situation is much more intricate and presently the research on the subject is very active. To start with, the generalized Jacobian of a singular curve is, in general, not compact. The problem of compactifying it is, of course, very natural, and it is considered to go back to the work of Mayer-Mumford and of Igusa in the 50's.
%Since then, several solutions appeared, starting with the work of Igusa himself and with the contribution of mathematicians as Oda-Seshadri, D'Souza, Altmann-Kleiman, Caporaso, Pandharipande, Simpson, Esteves, Schmitt, and so on. These constructions differ from one another in different aspects as the generality of the construction, the modular description of the boundary and the functorial properties.

Here we will consider  {\bf {\em fine compactified Jacobians}}, as constructed by E. Esteves in \cite{est1}.
 %which parametrize torsion-free rank-1 (i.e. rank-1 on each irreducible component of $X$) sheaves on $X$ that are semistable with respect to a general polarization on $X$.
Fine compactified Jacobians depend upon a \emph{polarization} on $X$, which for us will be simply a collection of rational numbers $\un q=\{\un q_{C_i}\}$,
one for each irreducible component $C_i$ of $X$, such that $|\un q|:=\sum_i \un q_{C_i}\in \Z$.
%We briefly review the definition of Esteves's fine compactified Jacobians of a reduced curve $X$; we refer the reader to \S\ref{S:fine-Jac} for more details.
%We will denote by $\{C_1, \ldots, C_{\gamma}\}$ the irreducible components of $X$.
%A \emph{polarization} on $X$ is a tuple of rational numbers $\un q=\{\un q_{C_i}\}$, one for each irreducible component $C_i$ of $X$, such that $|\un q|:=\sum_i \un q_{C_i}\in \Z$.
%We call $|\un q|$ the total degree of $\un q$.
%Given any (complete) subcurve $Y \subseteq X$, we set $\displaystyle \un{q}_Y:=\sum_{C_j\subseteq Y} \un{q}_{C_j}.$
A torsion-free rank-1  sheaf $I$ on $X$ of Euler characteristic  $\chi(I):=h^0(X,I)-h^1(X,I)$  equal to $|\un q|$ is called
\emph{$\un q$-semistable} (resp. \emph{$\un q$-stable}) if
for every proper subcurve $Y\subset X$, we have that
\begin{equation*}
\chi(I_Y)\geq  \sum_{C_i\subseteq Y} \un q_{C_i} \: \: (\text{resp. } > ),
\end{equation*}
where $I_Y$ is the biggest torsion-free quotient of the restriction $I_{|Y}$ of $I$ to the subcurve $Y$. A polarization $\un q$ is called {\em general} if there are no strictly $\un q$-semistable sheaves, i.e. if every $\un q$-semistable sheaf is also $\un q$-stable (see Definition \ref{def-int} for a numerical characterization of general polarizations). A  fine compactified Jacobian of $X$ is the fine moduli space $\ov{J}_X(\un q)$ of torsion-free rank-1 sheaves
 on $X$ that are $\un q$-semistable (or equivalently $\un q$-stable) with respect to a general polarization $\un q$ on $X$. Indeed, it is known that $\ov{J}_X(\un q)$ is a projective scheme over $k$ (see Fact \ref{F:Este-Jac}) on which the generalized Jacobian $J(X)$ of $X$ acts naturally by tensor product.

 If the curve $X$ has locally planar singularities, then we proved in \cite[Thm. A]{MRV1} that any fine compactified Jacobian $\ov{J}_X(\un q)$ of $X$ has the following remarkable properties:
\begin{itemize}
\item $\ov{J}_X(\un q)$ is a reduced scheme with locally complete intersection singularities.
\item The smooth locus of $\ov{J}_X(\un q)$ coincides with the open subset $J_X(\un q)\subseteq \ov{J}_X(\un q)$ parametrizing line bundles; in particular $J_X(\un q)$ is dense in $\ov{J}_X(\un q)$ and $\ov{J}_X(\un q)$ is of pure dimension equal to $p_a(X)$.
\item $\ov{J}_X(\un q)$ is connected.
\item $\ov{J}_X(\un q)$ has trivial dualizing sheaf.
\item $J_X(\un q)$ is the disjoint union of a number of copies of $J(X)$ equal to the complexity $c(X)$ of the curve $X$ (as defined in \cite[Def. 5.12]{MRV1}); in particular, $\ov{J}_X(\un q)$ has $c(X)$ irreducible components, independently of the chosen polarization $\un q$.
\end{itemize}
In the proof of all the above properties, we use in an essential way the fact that the curve has locally planar singularities and indeed we expect that many of the above properties are false without this assumptions (see \cite[Rmk. 2.7]{MRV1} and the references therein). The last property in the above list says that any two fine compactified Jacobians of a given curve $X$ are birational; however, the authors have found  in \cite{MRV1} examples of reducible curves (indeed even nodal curves) that admit non isomorphic (and even non homeomorphic if $k=\bbC$) fine compactified Jacobians.

After these preliminaries, we can now state our main results.

\subsection{Main results}

Since any fine compactified Jacobian $\ov{J}_X(\un q)$ of $X$ is a fine moduli space for certain sheaves (as the name suggests), there exists a universal sheaf $\I$ on $X\times \ov{J}_X(\un q)$. Using this universal sheaf and the formalism of the determinant of cohomology, it is possible to define a Poincar\'e line bundle $\P$ on $\ov{J}_X(\un q)\times J(X)$; we refer the reader to \S\ref{S:Poincare} for details.

Our first result concerns the Fourier-Mukai transform with kernel $\P$. This result can be seen as a first partial generalization of the above mentioned result of Mukai \cite{mukai} in the case of Jacobians. In Theorem E below (whose proof appears in \cite{MRV2}), we will give a second and more satisfactory generalization.

%Our first result is a generalization of Theorem A in \cite{arin1}.

%\begin{thmA}\label{T:MainA}
%Let $X$ be a reduced curve with locally planar singularities. Let $J$ be the generalized Jacobian of $X$ and
%$\ov{J}$ be any fine compactified Jacobian of $X$.
%Let $\P$ be a Poincar\'e bundle on $\J\times J$. Then we have that
%\begin{equation}\label{directim}
%R p_{2*} \P= \O_{\zeta}[-g]
%\end{equation}
%where $\O_{\zeta}$ is the structure sheaf of the neutral element $\zeta=[\O_X]\in J$, and $p_2:\J\times J\to J$  is the second projection.
%\end{thmA}

%From Theorem \ref{T:MainA} and either adapting the original argument of Mukai (\cite[Thm. 2.2]{mukai}) or using the %criterion of \textbf{Cristina}, we deduce the following

\begin{thmA}\label{T:MainA}
Let $X$ be a reduced projective connected curve with locally planar singularities over an algebraically closed field $k$.
Let $J(X)$ be the generalized Jacobian of $X$ and let $\ov{J}_X(\un q)$ be a fine compactified Jacobian of $X$.
Denote by  $D^b(J(X))$ and $D^b(\ov{J}_X(\un q))$ the bounded derived categories of quasi-coherent sheaves of  $J(X)$ and of $\ov{J}_X(\un q)$, respectively. Let $\P$ be a Poincar\'e line bundle on $\ov{J}_X(\un q)\times J(X)$.
Then the Fourier-Mukai transform (or integral transform) with kernel  $\P$
\begin{eqnarray*}
\Phi^{\P}:D^b(J(X))& \longrightarrow & D^b(\ov{J}_X(\un q))\\
\cplx{E} &\longmapsto & R  p_{1*}(p_2^*(\cplx{E})\otimes \P)
\end{eqnarray*}
is fully-faithful, where with $p_i$ we denote the projection of $\ov{J}_X(\un q)\times J(X)$ on the $i$-th factor.
\end{thmA}

As a corollary of Theorem A, we can compute the cohomology of the line bundles $\P_M:=\P_{|\J\times \{M\}}$ on $\ov{J}_X(\un q)$, as $M$ varies in $J(X)$, generalizing the classical result for abelian varieties (see \cite[Sec. 13]{Mum}).

\begin{corB}\label{C:CorB}
Same assumptions as in Theorem A.
For any $M\in J(X)$, let $\P_M:=\P_{|\ov{J}_X(\un q)\times \{M\}}\in \Pic(\ov{J}_X(\un q))$. Then we have that
$$H^i(\ov{J}_X(\un q), \P_M)=
\begin{cases}
0 & \text{ if } M\neq [\O_X], \\
\bigwedge^i H^1(X, \O_X) & \text{ if } M=[\O_X].
\end{cases}
$$
\end{corB}

As we mentioned in the introduction, Jacobians of smooth curves are autodual. In other words, given a smooth projective curve $C$, its Jacobian $J(C)$ is canonically isomorphic to the dual abelian variety which, by definition, is equal to $\Pic^o(J(C))$, i.e. the connected component of the Picard scheme of $J(C)$ containing the origin.
Our next result is a generalization of this autoduality result to fine compactified Jacobians.

\begin{thmC}\label{T:MainC}
Same assumptions as in Theorem A.
The morphism
$$\begin{aligned}
 \beta_{\un q}: J(X) & \longrightarrow \Pic^o(\ov{J}_X(\un q)) \\
M & \mapsto \P_M:=\P_{|\ov{J}_X(\un q) \times \{M\}}
\end{aligned}
$$
is an isomorphism of algebraic groups.
\end{thmC}

Finally, it is well known that a line bundle on an abelian variety $A$ is algebraically equivalent to zero if and only if it is numerically equivalent to zero (see \cite[Cor. 2, p. 178]{Mum}). In other words, the connected component $\Pic^o(A)$ of the Picard scheme $\Pic(A)$ of $A$ containing the identity (which also parametrizes line bundles algebraically equivalent to zero) coincides with the open and closed subset $\Pic^{\tau}(A)\subseteq \Pic(A)$ parametrizing line bundles numerically equivalent to zero. This is equivalent to say that the N\'eron-Severi group ${\rm NS}(A)=\Pic(A)/\Pic^o(A)$ of $A$ is torsion-free,
since the torsion subgroup of ${\rm NS}(A)$ is equal to $\Pic^{\tau}(A)/\Pic^o(A)$.
We prove that the same holds true for fine compactified Jacobians.

\begin{thmD}\label{T:MainD}
Same assumptions as in Theorem A.  Then we have that
$$\Pic^o(\ov{J}_X(\un q))=\Pic^{\tau}(\ov{J}_X(\un q)).$$
Equivalently, the N\'eron-Severi group ${\rm NS}(\ov{J}_X(\un q))$ is torsion-free.
\end{thmD}

Note that the above Theorem D is new even for irreducible curves in positive characteristic: the proof of Theorem D  for irreducible curves by D. Arinkin (see \cite[Prop. 12]{arin1}) uses in a crucial way that ${\rm char}(k)=0$.

\vspace{0,2cm}

In a sequel of this paper \cite{MRV2}, we will use the results of this article to prove the following
%which can be seen as a strengthening of Theorem A and a further generalization of Mukai's result in \cite{mukai} to the case of singular curves.

\begin{thmE}[\cite{MRV2}]\label{T:MainE}
%Same assumptions as in Theorem A.
Let $X$ be a reduced projective and connected curve with locally planar singularities and arithmetic genus $p_a(X)$ over an algebraically closed field $k$ of characteristic zero or greater than $p_a(X)$. Let  $\ov{J}_X(\un q)$ and $\ov{J}_X(\un q')$ be two (possibly equal) fine compactified Jacobians of $X$.
There exists a (naturally defined) Cohen-Macaulay  sheaf $\ov{\P}$ on $\ov{J}_X(\un q)\times \ov{J}_X(\un q')$ such that  the Fourier-Mukai transform (or integral transform) with kernel  $\ov{\P}$
\begin{eqnarray*}
\Phi^{\ov{\P}}:D^b(\ov{J}_X(\un q'))& \longrightarrow & D^b(\ov{J}_X(\un q))\\
\cplx{E} &\longmapsto & R  p_{1*}(p_2^*(\cplx{E})\lotimes \ov{\P})
\end{eqnarray*}
is an equivalence.
\end{thmE}

Note that, in the special case when $\ov{J}_X(\un q)\cong\ov{J}_X(\un q')$, Theorem E can be seen as a strengthening of Theorem A and a further generalization of Mukai's result
in \cite{mukai} to the case of singular reduced curves.
Moreover, this result provides a first step towards the proof of the Langlands duality for Higgs bundles (see \eqref{E:fibdual}) over the  open subset of reduced spectral curves (i.e. over the so-called regular locus of the Hitchin morphism); see the Appendix  for more details.

On the other hand, in the general case when $\ov{J}_X(\un q)\not\cong \ov{J}_X(\un q')$, Theorem E implies that any two fine compactified Jacobians of $X$ (which are birational, but possibly non isomorphic,  Calabi-Yau singular projective
varieties by what said above) are derived equivalent. This result seems to suggest an extension to (mildly) singular varieties of the conjecture of Kawamata \cite{Kaw}, which predicts that birational Calabi-Yau
smooth projective varieties should be derived equivalent.
Moreover,  a topological counterpart of the above result is obtained by the third author, together with L. Migliorini and V. Schende, in \cite{MSV}: any two fine compactified Jacobians of a complex curve $X$  (under the same assumptions on $X$) have  the same perverse Leray filtration on their cohomology. This result again seems to suggest an extension to (mildly) singular varieties of the result of Batyrev \cite{Bat} which  says that birational Calabi-Yau smooth projective complex varieties have the same Hodge numbers.

\subsection{Sketch of the proofs}

Let us now give a brief outline of the proofs of the main results, trying to highlight the main ingredients that we use.

{\bf Theorem A} follows easily from the formula
\begin{equation}\label{E:for-Mum}
R p_{2*} \P\cong  {\bf k}(0)[-g]
\end{equation}
where  ${\bf k}(0)$  denotes the skyscraper sheaf supported at the origin $0=[\O_X]\in J(X)$, $g=p_a(X)$ is the  arithmetic genus of $X$ and $p_2:\ov{J}_X(\un q)\times J(X)\to J(X)$  is the projection onto the second factor. Indeed, formula \eqref{E:for-Mum} is a generalization of a well-known result of Mumford (see \cite[Sec. III.13]{Mum}) for abelian varieties which was indeed the crucial step for the celebrated original result of Mukai \cite{mukai}.

In order to prove \eqref{E:for-Mum}, the key idea, which we learned from D. Arinkin in \cite{arin1} and \cite{arin2}\footnote{In loc. cit., D. Arinkin considers the stack of all integral curves with locally planar singularities, which is of finite type. Here, we replace this stack with the semiuniversal deformation space of $X$ since the stack of all reduced curves with locally planar singularities is not of finite type.}, is to prove a similar formula for the effective semiuniversal deformation family of $X$ (see \S\ref{S:DefX} for more details):
\begin{equation*}
\xymatrix{
X \ar[d]\ar@{^{(}->}[r]\ar@{}[dr]|{\square}&  \X\ar[d]^{\pi}\\
\Spec k\ar@{^{(}->}[r] &  \Spec R_X.
}
\end{equation*}

The generalized Jacobian $J(X)$ and the fine compactified Jacobian $\ov{J}_X(\un q)$  deform over $\Spec R_X$ to, respectively, the universal generalized Jacobian $v:J(\X)\to \Spec R_X$ (see Fact \ref{F:ungenJac}) and  the universal fine compactified Jacobian $u:\ov{J}_{\X}(\un q)\to \Spec R_X$ with respect to the polarization $\un q$ (see Theorem
\ref{T:univ-fine}). Therefore we get the following  diagram
\begin{equation}\label{E:diag-intr}
\xymatrix{
&  \ov{J}_{\X}(\un q)\times_{\Spec R_X}  J(\X) \ar[dl]^{\wt{u}} \ar[dr]_{\wt{v}} & \\
J(\X) \ar[dr]^{v}\ar@{}[rr]|{\square} & & \ov{J}_{\X}(\un q) \ar[dl]_{u}\\
& \Spec R_X \ar@(dl,dl)[ul]^{\zeta} &
}
\end{equation}
where the central square is Cartesian and $\zeta$ is the zero section of $v$. Moreover, the Poincar\'e line bundle $\P$ on $\ov{J}_X(\un q)\times J(X)$ deforms to the universal Poincar\'e line bundle $\Pun$ on the fiber product  $\ov{J}_{\X}(\un q)\times_{\Spec R_X}  J(\X)$ (see \S\ref{S:nodalC}).

Equation \eqref{E:for-Mum} will follow, by restricting to the central fiber of $v$, from the following universal version of it (which we prove in Theorem \ref{T:push-for}):
\begin{equation}\label{E:Mum-intr}
R\wt{u}_* (\Pun)\cong \zeta_*(\O_{\Spec R_X})[-g].
\end{equation}
A key intermediate step in proving \eqref{E:Mum-intr} consists in showing that
\begin{equation*}
R\wt{u}_*(\Pun)[g]\cong R^g \wt{u}_*( \Pun) \text{ is a Cohen-Macaulay sheaf such that } \supp (R^g \wt{u}_*( \Pun))=\Im (\zeta). \tag{*}
\end{equation*}
%\begin{equation*}
%\supp (R^g \wt{u}_*( \Pun))=\Im (\zeta). \tag{**}
%\end{equation*}
%\begin{enumerate}[(i)]
%\item $R\wt{u}_*(\Pun)=R^g \wt{u}_*( \Pun)[-g]$ is a Cohen-Macaulay sheaf of codimension $g$.
%\item $\supp(R^g \wt{u}_*( \Pun))=\Im (\zeta)$.
%\end{enumerate}
The proof of (*) has two main ingredients. The first ingredient  is the study of the cohomology of the line bundles $\P_M\in \Pic(\ov{J}_X(\un q))$, for $M\in J(X)$; see \S \ref{S:coho-Poin}. Here, we use in an essential way the Abel map $A_L:X \longrightarrow \bJbar_X $, for $L\in \Pic(X)$, with values in the scheme $\bJbar_X$ parametrizing all simple torsion-free rank-1 sheaves on $X$, which was studied by the authors in \cite[\S 6]{MRV1} (see \S \ref{S:Abel} for a review).
The second ingredient is the \emph{equigeneric stratification} of $\Spec R_X$, i.e. the stratification of $\Spec R_X$ according to the arithmetic genus of the normalization of the geometric  fibers of the universal family $\X\to \Spec R_X$.
If $X$ has locally planar singularities, then each equigeneric stratum has codimension at least equal to the total $\delta$-invariant and all its generic points correspond to nodal curves: a result that is certainly well known to the experts (and proved partially by Teissier \cite{Tei} and Diaz-Harris \cite{DH} over $k=\bbC$ and by \cite{MY} over an algebraically closed field $k$ of large characteristic), and of which we will give a detailed proof in \cite{RV}.
These properties allow us to prove (*) over the generic points of each equigeneric stratum, using in an essential way Theorem C for nodal curves; see \S \ref{S:nodalC}.

The proof of {\bf Theorem C} follows the same idea of using the semiuniversal deformation family $\X\to \Spec R_X$ of $X$.
Under the assumption  that
\begin{equation*}
h^1(\ov{J}_X(\un q),\O_{\ov{J}_X(\un q)})=p_a(X), \tag{**}
\end{equation*}
the map $\beta_{\un q}$ of Theorem C deforms over $\Spec R_X$ to a homomorphism
\begin{equation}\label{E:beta-intr}
\beun: J(\X)\to \Pic^o(\ov{J}_{\X}(\un q))
\end{equation}
between two group schemes which are smooth, separated and of finite type over $\Spec R_X$ (see Fact \ref{F:ungenJac}, Theorem \ref{T:Pic-univ}\eqref{T:Pic-univ4} and Proposition \ref{P:hom-betaun}). We note here that the representability of $\Pic^o(\ov{J}_{\X}(\un q))$ and its smoothness over $\Spec R_X$ (proved in Theorem \ref{T:Pic-univ}\eqref{T:Pic-univ4}) use in a crucial way the assumption (**).

In Theorem \ref{T:isobetauniv}, we prove that the map $\beun$ is an isomorphism (assuming that (**) holds true), which therefore implies Theorem C restricting to the closed point of $\Spec R_X$. The proof of Theorem \ref{T:isobetauniv} uses the fact (due to Esteves-Gagn\'e-Kleiman \cite{egk}) that  $\beun$ is an isomorphism over the open subset $U\subseteq \Spec R_X$ (whose complement has codimension at least two by Lemma \ref{L:morpi}\eqref{L:morpi3})  of curves having at most one node, which combined with Van der Waerden's theorem on the purity of the ramification locus and Zariski's main theorem, gives that $\beun$ is an open embedding, hence an isomorphism.

Formula (**) is proved for nodal curves in Proposition \ref{P:H1-nodal} using results of Oda-Seshadri \cite{OS} and Alexeev-Nakamura \cite{AN}. Note that Theorem C for nodal curves plays a key role in establishing \eqref{E:Mum-intr}, hence in the proof of Theorem A and Corollary B.
For an arbitrary curve $X$ with locally planar singularities, formula (**) follows from Corollary B, hence from the Fourier-Mukai type result of Theorem A. A direct proof of (**) would allow to give a Fourier-Mukai's free proof of Theorem C (and  also of Theorem D as we will see below).

Finally, let us sketch the proof of {\bf Theorem D}, which will be given in  \S  \ref{S:ThmD}.

In Theorem \ref{T:thmD-Abel}, we will first prove Theorem D in the special case where the curve $X$ does not admit separating nodes and the fine compactified Jacobian $\ov{J}_X(\un q)$ admits an Abel map, i.e. if there exists $L\in \Pic(X)$ such that $\Im\,{A_L}\subseteq \ov{J}_X(\un q)$. Note that this hypothesis is quite restrictive for a fine compactified Jacobian since in general only a few of them will admit an Abel map (see e.g. \cite[\S 7]{MRV1}).
Once again, the strategy will be to work on the semiuniversal deformation family $\X\to \Spec R_X$. Indeed, we can deform the line bundle $L\in \Pic(X)$ that gives the Abel map $A_L:X\to \ov{J}_X(\un q)$ to a line bundle $\L$ on $\X$ in order to obtain a universal Abel map $A_{\L}:\X\to \ov{J}_{\X}(\un q)$. By taking the pull-back via $A_{\L}$, we obtain the following commutative diagram of group schemes (all of which are smooth, separated and of finite type over $\Spec R_X$, by Fact \ref{F:ungenJac} and Theorem \ref{T:Pic-univ}):
\begin{equation}\label{E:Abel-intr}
\xymatrix{
\Pic^{\tau}(\ov{J}_{\X}(\un q)) \ar@{->>}[rd]^{A_{\L}^{*,\tau}} & \\
& J(\X)\\
\Pic^o(\ov{J}_{\X}(\un q)) \ar[ru]_{A_{\L}^{*,o}}^{\cong} \ar@{^{(}->}[uu]^i& \\
}
\end{equation}
where $i$ is the natural open embedding and $A_{\L}^{*,o}$ is an isomorphism since it is the right inverse of $\beun$ (by Proposition \ref{P:prop-AL}), which is an isomorphism by Theorem \ref{T:isobetauniv}. The morphism $A_{\L}^{*,\tau}$ is an isomorphism over the open subset $U\subseteq \Spec R_X$  of curves having at most one node (as it follows from \cite{egk}); using that $\Spec R_X\setminus U$ has codimension at least two, together with  Van der Waerden's theorem on the purity of the ramification locus and Zariski's main theorem, we conclude that $A_{\L}^{*,\tau}$ is an open embedding, hence an isomorphism. Therefore $i$ must be an equality and Theorem D in this special case follows by restricting to the closed point of $\Spec R_X$.

In order to prove Theorem D in the general case, i.e. if either $X$ does have separating nodes or $\ov{J}_X(\un q)$ does not admit an Abel map, we first reduce to curves without separating nodes using that any fine compactified Jacobians of a curve $X$ is the product of fine compactified Jacobians of subcurves of $X$ without separating nodes (see Theorem \ref{T:Abel}) and that the formation of $\Pic^o$ and $\Pic^{\tau}$ commutes with products (provided that they are smooth algebraic groups) by a result of A. Langer \cite[Cor. 4.7]{Lan}.
Then, if $X$ does not have separating nodes but the fine compactified Jacobian $\J_X(\un q)$ does not admit an Abel map,
we consider another fine compactified Jacobian $\ov{J}_X(\un q')$ of $X$ that does admit an Abel map (such a fine compactified Jacobian $\ov{J}_X(\un q')$ exists by Theorem \ref{T:Abel}\eqref{T:Abel3}) and we are able to deduce
Theorem D for $\ov{J}_X(\un q)$ knowing that it does hold true for $\ov{J}_X(\un q')$ (by Theorem \ref{T:thmD-Abel}). The key ingredient is to compare their universal fine compactified Jacobians $\ov{J}_{\X}(\un q)$ and $\ov{J}_{\X}(\un q')$
by showing that they are isomorphic over the open subset $U\subseteq \Spec R_X$  of curves having at most one node (see Lemma \ref{L:comp-unJac}). We refer to \S  \ref{S:ThmD} for more details.

\subsection{Outline of the paper}

The paper is organized as follows.

Section \ref{S:comp-Jac} is devoted to collecting several facts on fine compactified Jacobians of reduced curves. In \S\ref{S:sheaves}, we consider the scheme $\bJbar_X$ parametrizing all simple torsion-free rank-1 sheaves on a curve $X$ (see Fact \ref{F:huge-Jac}) and we recall its properties under the assumption that $X$ has locally planar singularities (see Theorem \ref{T:prop-J-planar}). In \S\ref{S:fine-Jac}, we introduce fine compactified Jacobians of $X$ (see Fact \ref{F:Este-Jac}) and we recall their properties under the assumption that $X$ has locally planar singularities (see Theorem \ref{T:compJac}). Finally, in \S \ref{S:Abel},  we recall the definition of the $L$-twisted Abel map of degree one and its main properties (see Theorem \ref{T:Abel}).

Section \ref{S:univ-Jac} is devoted to collecting several results on the universal fine compactified Jacobians. In \S\ref{S:DefX}, we  recall some basic facts about the semiuniversal deformation space of a curve $X$ and  the properties of its equigeneric stratification in the case where $X$ has locally planar singularities. In \S\ref{S:unJac}, we introduce  the universal fine compactified Jacobians relative to the semiuniversal deformation of a curve $X$ (see Fact \ref{F:univ-Jac} and Theorem \ref{T:univ-fine}) and study these Jacobians under the assumption that $X$ has locally planar singularities (see Theorem \ref{T:univ-Jac}).

Section \ref{S:Pic-univ} is devoted to study the representability of the relative Picard scheme of the universal fine compactified Jacobians and of its subfunctors parametrizing line bundles that are fiberwise algebraically or numerically equivalent to the trivial line bundle (see Theorem \ref{T:Pic-univ}).

In Section \ref{S:Poincare}, we define the Poincar\'e line bundle and study its behavior with respect to the Abel maps (see Proposition \ref{P:prop-AL}).

In Section \ref{S:coho-Poin}, we study the cohomology of the restricted Poincar\'e line bundles on a fine compactified Jacobian, obtaining some  special cases of Corollary B.

Section \ref{S:nodalC} contains a proof of Theorem C for nodal curves while Section \ref{S:proof1} contains the proof of Theorem A, Corollary B and the general case of Theorem C. Finally, Theorem D is proved in  Section \ref{S:ThmD}.

 In the Appendix, we first discuss the Hitchin fibration and the description of its fibers in terms of compactified Jacobians of spectral curves (see Fact \ref{F:corr-spec}). Then we state the conjectural Langlands duality for Higgs bundles  (see Conjecture \ref{C:Lang-duality}) and its fiberwise version for each spectral curve (see \eqref{E:fibdual}).

\vspace{0,2cm}

The following notations will be used throughout the paper.

\subsection*{Notations}
\begin{convention}
 	$k$ will denote an algebraically closed field (of arbitrary characteristic), unless otherwise stated. All \textbf{schemes} are $k$-schemes, and all morphisms are implicitly assumed to respect
	the $k$-structure.
\end{convention}

\begin{convention}\label{N:curves}
	A \textbf{curve}  is a \emph{reduced} projective scheme over $k$ of pure dimension $1$. Unless otherwise specified, a curve is meant to be  connected.

Given a curve $X$, we denote by $X_{\rm sm}$ the smooth locus of $X$, by $X_{\rm sing}$ its singular locus and by $\nu:X^{\nu}\to X$ the normalization morphism.
 We denote by $\gamma(X)$, or simply by $\gamma$ where there is no danger of confusion, the number of irreducible components of $X$.

We denote by $p_a(X)$  the \emph{arithmetic genus} of $X$, i.e.  $p_a(X):=1-\chi(\O_X)=1-h^0(X,\O_X)+h^1(X, \O_X) $.
We denote by $g^{\nu}(X)$ the \emph{geometric genus} of $X$, i.e. the sum of the genera of the connected components of the normalization $X^{\nu}$. Note that $g^{\nu}(X)=h^1(X^{\nu}, \O_{X^{\nu}})$.

\end{convention}

\begin{convention}
	A \textbf{subcurve} $Z$ of a curve $X$ is a closed $k$-subscheme $Z \subseteq X$ that is reduced  and of pure dimension $1$.  We say that a subcurve $Z\subseteq X$ is non-trivial if
	$Z\neq \emptyset, X$.
	
	Given two subcurves $Z$ and $W$ of $X$ without common irreducible components, we denote by $Z\cap W$ the $0$-dimensional subscheme of $X$ that is obtained as the
	scheme-theoretic intersection of $Z$ and $W$ and we denote by $|Z\cap W|$ its length.
	
	Given a subcurve $Z\subseteq X$, we denote by $Z^c:=\ov{X\setminus Z}$ the \textbf{complementary subcurve} of $Z$ and we set $\delta_Z=\delta_{Z^c}:=|Z\cap Z^c|$.
 \end{convention}

\begin{convention}
A curve $X$ is called \textbf{Gorenstein} if its dualizing sheaf $\omega_X$ is a line bundle.
\end{convention}

\begin{convention}
A curve $X$ has \textbf{locally complete intersection (l.c.i.) singularities at $p\in X$} if the completion $\wh{\O}_{X,p}$ of the local ring of $X$ at $p$ can be written as
$$\wh{\O}_{X,p}=k[[x_1,\ldots,x_r]]/(f_1,\ldots,f_{r-1}),$$
for some $r\geq 2$ and some $f_i\in k[[x_1,\ldots,x_r]]$. A curve $X$ has locally complete intersection (l.c.i.)
singularities if $X$ is l.c.i. at every $p\in X$.
It is well known  that  a curve with l.c.i. singularities is Gorenstein.
\end{convention}

\begin{convention}\label{N:locplan}
A curve $X$ has \textbf{locally planar singularities at $p\in X$} if  the completion
$\wh{\O}_{X,p}$ of the local ring of $X$ at $p$ has embedded dimension at most two, or equivalently if it can be written
as
$$\wh{\O}_{X,p}=k[[x,y]]/(f),$$
for a reduced series $f=f(x,y)\in k[[x,y]]$.
A curve $X$ has locally planar singularities if $X$ has locally planar singularities at every $p\in X$.
Clearly, a curve with locally planar singularities has l.c.i. singularities, hence it is Gorenstein. A (reduced) curve has locally planar singularities if and only if it can be embedded in a smooth surface (see \cite{AK0}).
\end{convention}

\begin{convention}
A curve $X$ has a \textbf{node at $p\in X$} if  the completion
$\wh{\O}_{X,p}$ of the local ring of $X$ at $p$ is isomorphic to
$$\wh{\O}_{X,p}=k[[x,y]]/(xy).$$
\end{convention}

\begin{convention}\label{N:sep-node}
A \textbf{separating point}
%of a connected curve $X$ is a node $p\in X$ such that the partial normalization of $X$ at $p$ is disconnected.
is a closed  point $n\in X$ for which there exists a subcurve $Z\subset X$ such that $\delta_Z=1$ and $Z\cap Z^c=\{n\}$. Often, we will deal with reduced curves satisfying the following
\begin{equation}\label{E:dagger}
\un{\text{Condition } (\dagger)}: \text{Every separating point is a node.}
\end{equation}
Every Gorenstein curve satisfies the condition $(\dagger)$ by  \cite[Prop. 1.10]{Cat}. However, the union of the three coordinate axes in ${\mathbb A}^3$ is a (non Gorenstein) reduced curve that does not satisfy condition $(\dagger)$
(see    \cite[Example 6.5]{MRV1}).
\end{convention}

\begin{convention}\label{N:Pic-field}
Given a scheme $S$ proper over a field $k$ (not necessarily algebraically closed), we denote by $\Pic(S)$ its \textbf{Picard scheme}, which exists by a result of Murre
(see \cite[Cor. 9.4.18.3]{FGA} and the references therein). The \textbf{connected component of the identity} of $\Pic(S)$, denoted by $\Pic^o(S)$, parametrizes
line bundles on $S$ which are algebraically equivalent to the trivial line bundle (see \cite[Sec. 9.5]{FGA} for details).
The \textbf{torsion component of the identity} of $\Pic(S)$, denoted by $\Pic^{\tau}(S)$, parametrizes
line bundles on $S$ which are numerically equivalent to the trivial line bundle or, equivalently, such that some powers of them lie in $\Pic^o(S)$
(see \cite[Sec. 9.6]{FGA} for details). The scheme $\Pic^{\tau}(S)$ is an open and closed group subscheme of $\Pic(S)$ which is of finite type over $k$
(see \cite[Prop. 9.6.12]{FGA}).

On the other hand, given an arbitrary scheme $S$, we denote by ${\mathcal Pic}(S)$ the \textbf{Picard group}  of $S$, i.e. the abstract group consisting of all isomorphism classes of line bundles
on $S$ with the operation of tensor product.
\end{convention}

\begin{convention}\label{N:Jac-gen}
Given a curve $X$ over an algebraically closed field, we call $\Pic^o(X)$ the \textbf{generalized Jacobian} of $X$.
It is easy that the $k$-valued points of $\Pic^o(X)$ coincide with the group of line bundles on $X$ of multidegree $\un 0$ (i.e. having
degree $0$ on each irreducible component of $X$) together with the multiplication given by the tensor product.
The generalized Jacobian of $X$ is a connected commutative smooth algebraic group of dimension equal to
$h^1(X,\O_X)$ and it coincides with $\Pic^\tau(X)$.
We also use the notation $J(X)$ and $\Pic^{\un 0}(X)$ for the generalized Jacobian of $X$.
\end{convention}

\begin{convention}
Given a scheme $X$, we will denote by $D(X)$ the \textbf{derived category} of complexes of $\mathcal{O}_X$-modules with quasi-coherent cohomology sheaves and by $D^b(X)\subset D(X)$ the \textbf{bounded derived category} consisting of complexes
with only finitely many non-zero cohomology sheaves.
\end{convention}

\begin{convention}
Given a scheme $X$ and a closed point $x\in X$, we will denote by ${\bf k}(x)$ the \textbf{skyscraper sheaf} supported at $x$.
\end{convention}

\section{Fine Compactified Jacobians}\label{S:comp-Jac}

The aim of this section is to collecting several facts about fine compactified Jacobians of reduced curves with locally planar singularities, following \cite[\S 2]{MRV1}.
%Throughout this section, we fix a connected reduced curve $X$ with locally planar singularities.

\subsection{Simple rank-1 torsion-free sheaves}\label{S:sheaves}

Fine compactified Jacobians on a connected reduced curve $X$ parametrize simple rank-$1$ torsion free sheaves on $X$.

\begin{defi}
A coherent sheaf $I$ on a connected reduced curve $X$ is said to be:
\begin{enumerate}[(i)]
\item \emph{rank-1} if $I$ has generic rank $1$ at every irreducible component of $X$;
\item \emph{torsion-free} if $\Supp(I)=X$ and every non-zero subsheaf $J\subseteq I$ is such that $\dim \Supp(J)=1$;
%\footnote{If $X$ is irreducible, then $I$ is pure if and only if $I$ is torsion-free. Sometimes, by abuse of notation, one uses the word torsion-free to denote pure sheaves even if $X$ is not
%irreducible.};
\item \emph{simple} if ${\rm End}_k(I)=k$.
\end{enumerate}
\end{defi}
\noindent Note that any line bundle on $X$ is a simple rank-1 torsion-free sheaf.

Consider the functor
\begin{equation}\label{E:func-Jbar}
\bJbar_X^* : \{{\rm Schemes}/k\}  \to \{{\rm Sets}\}
\end{equation}
which associates to a $k$-scheme $T$ the set of isomorphism classes of $T$-flat, coherent sheaves on $X\times _k T$
whose fibers over $T$ are simple rank-1 torsion-free sheaves.
% (this definition agrees with the one in \cite[Def. 5.1]{AK} by virtue of \cite[Cor. 5.3]{AK}).
%such that the canonical map $\O_T\to (\pi_2)_*\Hom_{X_T}(\F,\F)$ is an isomorphism of sheaves on $T$.
  The functor $\bJbar_X^*$ contains the open subfunctor
\begin{equation}\label{E:func-J}
\bJ_X^* : \{{\rm Schemes}/k\}  \to \{{\rm Sets}\}
\end{equation}
which associates to a $k$-scheme $T$ the set of isomorphism classes of line bundles on $X\times _k T$.

\begin{fact}[Murre-Oort, Altman-Kleiman \cite{AK}, Esteves \cite{est1}]\label{F:huge-Jac}
Let $X$ be a connected reduced curve. Then
\noindent
\begin{enumerate}[(i)]
\item \label{F:huge1} The \'etale sheafification of $\bJ_X^*$ is represented by a $k$-scheme $\Pic(X)=\bJ_X$, locally of finite type over $k$. Moreover, $\bJ_X$ is formally smooth over $k$.
\item \label{F:huge2} The \'etale sheafification of $\bJbar_X^*$ is represented by a $k$-scheme $\bJbar_X$, locally of finite type over $k$. Moreover, $\bJ_X$ is an open subset of $\bJbar_X$ and $\bJbar_X$ satisfies the existence part of the valuative criterion for properness\footnote{Note that $\bJbar_X$ is not universally closed because it is not quasi-compact, in general.}.
\item \label{F:huge3} There exists a sheaf $\I$ on $X\times \bJbar_X$  such for every $\F\in \bJbar_X^*(T)$ there exists a unique map $\alpha_{\F}:T\to \bJbar_X$ with the property that
$\F=(\id_X\times \alpha_{\F})^*(\I)\otimes \pi_2^*(N)$ for some $N\in \Pic(T)$, where $\pi_2:X\times T\to T$ is the projection onto the second factor.
The sheaf $\I$ is uniquely determined up  to tensor product with the pullback of an invertible sheaf on $\bJbar_X$ and it is called a \emph{universal sheaf}.
\end{enumerate}
\end{fact}
\begin{proof}
See \cite[Fact 2.2]{MRV1}  and the references therein.

\end{proof}

%Given a coherent sheaf $I$ on $X$, its degree $\deg(I)$ is defined by $\deg(I):=\chi(I)-\chi(\O_X)$, where $\chi(I)$ (resp. $\chi(\O_X)$) denotes the Euler-Poincar\'e characteristic of $I$
%(resp. of the trivial sheaf $\O_X$).
Since the Euler-Poincar\'e characteristic $\chi(I):=h^0(X,I)-h^1(X,I)$ of a sheaf $I$ on $X$ is constant under deformations, we get a decomposition
\begin{equation}\label{E:dec}
\begin{sis}
& \bJbar_X=\coprod_{\chi \in \Z} \bJbar_X^{\chi},\\
& \bJ_X=\coprod_{\chi\in \Z} \bJ_X^{\chi},\\
\end{sis}
\end{equation}
where $\bJbar_X^{\chi}$ (resp. $\bJ_X^{\chi}$) denotes the open and closed subscheme of $\bJbar_X$ (resp. $\bJ_X$) parametrizing
simple rank-1 torsion-free sheaves $I$ (resp. line bundles $L$) such that $\chi(I)=\chi$ (resp. $\chi(L)=\chi$).

If $X$ has locally planar singularities, then $\bJbar_X$ has the following properties.

\begin{thm}\label{T:prop-J-planar}
Let $X$ be a connected reduced curve with locally planar singularities. Then
\begin{enumerate}[(i)]
\item $\bJbar_X$ is a reduced scheme with locally complete intersection singularities.
\item $\bJ_X$ is dense in $\bJbar_X$.
\item $\bJ_X$ is the smooth locus of $\bJbar_X$.
\end{enumerate}
\end{thm}
\begin{proof}
See \cite[Thm. 2.3]{MRV1}.
\end{proof}

\subsection{Fine compactified Jacobians}\label{S:fine-Jac}

For an integer  $\chi \in\mathbb Z$, the scheme $\bJbar_X^{\chi}$ is not of finite type nor separated over $k$ (and similarly for $\bJ_X^{\chi}$) if $X$ is not irreducible.
However, they can be covered by open subsets that are proper (and even projective) over $k$: the fine compactified Jacobians of $X$. Fine compactified Jacobians  depend on the choice of a polarization, whose definition is as
follows.
%  (using the notations of \cite{MV}).

\begin{defi}\label{pola-def}
A \emph{polarization} on a connected curve $X$ is a tuple of rational numbers $\un q=\{\un q_{C_i}\}$, one for each irreducible component $C_i$ of $X$, such that $|\un q|:=\sum_i \un q_{C_i}\in \Z$.
We call $|\un q|$ the total degree of $\un q$.

%for every connected component $Y\subset X$ of $X$ it holds that
%$$\sum_{C_i\subset Y} \un q_{C_i}\in \Z.$$
%The degree of the polarization is defined to be $|\un q|:=\sum_i \un q_{C_i}\in \Z$.
\end{defi}

Given any subcurve $Y \subseteq X$, we set $\un{q}_Y:=\sum_j \un{q}_{C_j}$ where the sum runs
over all the irreducible components $C_j$ of $Y$.
Note that giving a polarization $\un q$ is the same as giving an
assignment $(Y\subseteq X)\mapsto \un q_Y$ such that $\un q_X\in \Z$ and which is additive on $Y$, i.e. such that if $Y_1,Y_2\subseteq X$ are two subcurves of $X$ without common irreducible components, then $\un q_{Y_1\cup Y_2}=\un q_{Y_1}+\un q_{Y_2}$.

\begin{defi}\label{def-int}
A polarization $\un q$ is called \emph{integral} at a subcurve $Y\subseteq X$ if
$\un q_Z \in \Z$ for any connected component $Z$ of $Y$ and of $Y^c$.

A polarization is called
%\begin{enumerate}[(i)]
%\item \label{def-pol1} A polarization $\un q$ is called
\emph{general} if it is not integral at any proper subcurve $Y\subset X$.
 %which is not a connected component of $X$.
%\item \label{def-pol2} A polarization $\un q$ is called \emph{non-degenerate} if it is not integral at any proper %subcurve $Y\subsetneq X$ which is not a spine of $X$.
\end{defi}

\begin{remark}\label{R:conn-pola}
It is easily seen that  $\un q$ is general if and only if $\un q_Y\not\in \Z$ for any proper subcurve $ Y\subset X$ such that $Y$ and $Y^c$ are connected.
\end{remark}

For each subcurve $Y$ of $X$ and each torsion-free sheaf $I$ on $X$, the restriction $I_{|Y}$ of $I$ to $Y$ is not necessarily a
torsion-free sheaf on $Y$. However, $I_{|Y}$ contains a biggest subsheaf, call it temporarily $J$, whose support has dimension zero, or in other words  such that $J$ is a torsion sheaf.
We denote by $I_{Y}$ the quotient of $I_{|Y}$
by $J$. It is easily seen that $I_Y$ is torsion-free on $Y$ and it is the biggest torsion-free quotient of $I_{|Y}$: it is actually the unique torsion-free quotient of $I$ whose support is equal to $Y$.
Moreover, if $I$ is torsion-free rank-1 then $I_Y$ is torsion-free rank-1.
%We let $\deg_Y (I)$ denote the degree of $I_Y$ on $Y$, that is, $\deg_Y(I) := \chi(I_Y )-\chi(\O_Y)$.

\begin{defi}\label{sheaf-ss-qs}
\noindent Let $\un q$ be a polarization on $X$.
Let $I$ be a torsion-free rank-1  sheaf on $X$ such that $\chi(I)=|\un q|$ (not necessarily simple).
\begin{enumerate}[(i)]
\item \label{sheaf-ss} We say that $I$ is \emph{semistable} with respect to $\un q$ (or $\un q$-semistable) if
for every proper subcurve $Y\subset X$, we have that
\begin{equation}\label{multdeg-sh1}
\chi(I_Y)\geq \un q_Y.
\end{equation}
\item \label{sheaf-s} We say that $I$ is \emph{stable} with respect to $\un q$ (or $\un q$-stable) if it is semistable with respect to $\un q$
and if the inequality (\ref{multdeg-sh1}) is always strict.
\end{enumerate}
\end{defi}

\begin{remark}\label{R:tanteoss}
\noindent
\begin{enumerate}[(i)]
\item \label{R:tanteoss1}
It is easily seen that a torsion-free rank-1 sheaf $I$ is  $\un q$-semistable (resp. $\un q$-stable) if and only if \eqref{multdeg-sh1} is satisfied (resp. is satisfied with strict inequality)
for any subcurve $Y\subset X$ such that  $Y$ and $Y^c$ are connected.

\item \label{R:tanteoss2}
Let $\un q$ be a polarization on $X$ and $I$ a torsion-free rank-1 sheaf on $X$ that is stable with respect to $\un q$.  Then, it is easy to see that,  by slightly perturbing $\un q$, we get a general polarization $\un q'$ on $X$ for which $I$ remains stable.

\item \label{R:tanteoss3}
If $X$ has locally planar singularities, we can write the inequality \eqref{multdeg-sh1} in terms of the degree of $I_Y$ as follows
\begin{equation}\label{E:ineqdeg}
\chi(I_Y )-\chi(\O_Y):=\deg_Y(I)\geq \un q_Y-\chi(\O_Y)=\un q_Y+\frac{\deg_Y(\omega_X)}{2}-\frac{\delta_Y}{2},
\end{equation}
where we used the adjunction formula (see \cite[Lemma 1.12]{Cat})
%\footnote{In loc. cit., the formula is only stated for Gorenstein subcurves; however, an easy inspection of the proof reveals that the formula does hold true for non-Gorenstein subcurves as well.})
$$\deg_Y(\omega_X)=2p_a(Y)-2+\delta_Y=-2\chi(\O_Y)+\delta_Y.$$
The inequality \eqref{E:ineqdeg} was used to define stable rank-$1$ torsion-free sheaves on nodal curves in \cite{MV} and in \cite{CMKV}.
\end{enumerate}
\end{remark}

The geometric meaning for a polarization being general is clarified by the following result.

\begin{lemma}\label{L:nondeg}
Let $X$ be a connected reduced curve and let $\un q$ be a general polarization on $X$. Then every  $\un q$-semistable sheaf $I$ is also $\un q$-stable and hence simple.
\end{lemma}
\begin{proof}
See \cite[Lemmas 2.18]{MRV1}.
\end{proof}

For a general polarization $\un q$ on a connected reduced curve $X$, we will denote by
$\ov{J}_X(\un q)$ the open subscheme of $\bJbar_X$ parametrizing simple rank-1 torsion-free sheaves $I$ on $X$ which are $\un q$-semistable (or equivalently $\un q$-stable by Lemma \ref{L:nondeg}).
The scheme $\ov{J}_X(\un q)$ is called the \emph{fine compactified Jacobian} with respect to the polarization $\un q$.

%By \cite[Prop. 34]{est1}, the inclusions
%$\ov{J}^{s}_X(\un q)\subseteq \ov{J}^{ss}_X(\un q)\subset \bJbar_X$$
%are open.

\begin{fact}[Esteves \cite{est1}]\label{F:Este-Jac}
Let $X$ be a connected reduced curve.
\noindent
\begin{enumerate}[(i)]
\item \label{F:Este-Jac1} If $\un q$ is general polarization on $X$ then $\ov{J}_X(\un q)$ is a projective scheme over $k$ (not necessarily reduced).
\item \label{F:Este-Jac2} $\displaystyle \bJbar_X=\bigcup_{{\un q} \text{ general}} \ov{J}_X(\un q).$
\end{enumerate}
\end{fact}
\begin{proof}
Part \eqref{F:Este-Jac1} follows from \cite[Thm. A(1) and Thm. C(4)]{est1}.
Part \eqref{F:Este-Jac2} follows from \cite[Cor. 15]{est1}, which asserts that a simple torsion-free rank-1 sheaf is stable with respect to a certain polarization, together with Remark \ref{R:tanteoss}\eqref{R:tanteoss2}, which asserts that it is enough to consider general polarizations.
\end{proof}

We collect the properties of fine compactified Jacobians in the following Theorem.

\begin{thm}\label{T:compJac}
Let $X$ be a connected reduced curve with locally planar singularities. Then every fine compactified Jacobian $\ov{J}_X(\un q)$ satisfies the following properties:
\begin{enumerate}[(i)]
\item \label{T:compJac1} $\ov{J}_X(\un q)$ is a reduced scheme with locally complete intersection singularities;
\item \label{T:compJac2} The smooth locus of $\ov{J}_X(\un q)$ coincides with the open subset $J_X(\un q)\subseteq \ov{J}_X(\un q)$ parametrizing line bundles; in particular $J_X(\un q)$ is dense in $\ov{J}_X(\un q)$ and
$\ov{J}_X(\un q)$ is of pure dimension equal to $p_a(X)$;
\item \label{T:compJac3} $\ov{J}_X(\un q)$ is connected;
\item \label{T:compJac4} $\ov{J}_X(\un q)$ has trivial dualizing sheaf;
\item \label{T:compJac5} $J_X(\un q)$ is the disjoint union of a number of copies of the generalized Jacobian $J(X)$ of $X$ equal to the complexity $c(X)$ of the curve $X$;
in particular, $\ov{J}_X(\un q)$ has $c(X)$ irreducible components, independently of the chosen polarization $\un q$.
\end{enumerate}
\end{thm}
\begin{proof}
See \cite[Thm. A]{MRV1}.
\end{proof}

The complexity $c(X)$ of a reduced curve $X$ with planar singularities is an invariant of $X$ that depends on the pairwise intersection numbers of the irreducible components of $X$; see  \cite[Def. 5.10]{MRV1} for a definition.
Raynaud showed in \cite{Ray} that, for any one-parameter regular smoothing of $X$,  $c(X)$ is the number of connected components of the special fiber of the N\'eron model of the Jacobian of the generic fiber. Part \eqref{T:compJac5} of the above Theorem \ref{T:compJac} follows then from a result of J. Kass \cite{Kas2}, which says that, for a one-parameter regular smoothing of $X$, any relative fine compactified Jacobian is a compactification of the N\'eron model of its generic fiber.

%The name degree class group was first introduced by L. Caporaso in \cite[Sec. 4.1]{Cap}. The name complexity comes from the fact that if $X$ is a nodal curve then $c(X)$ is the complexity of the dual graph $\Gamma_X$ of $X$, i.e. the number of spanning trees of $\Gamma_X$ (see e.g. \cite[Sec. 2.2]{MV}).

%Fine compactified Jacobians of $X$ depend on the choice of a polarization and there are clearly infinitely many polarizations. However, we are now going to show that there are finitely many  isomorphism classes of fine compactified Jacobians.  The simplest way to show that two fine compactified Jacobians are isomorphic is to show that there is translation that sends one into the other.

%\begin{defi}\label{D:transla}
%Let $X$ be a connected curve. We say that two compactified Jacobians $\ov{J}_X(\un q)$ and $\ov{J}_X(\un q')$ are \emph{equivalent by translation } if there exists  a line bundle $L$ on $X$ inducing an isomorphism
%$$\begin{aligned}
%\ov{J}_X(\un q) & \stackrel{\cong}{\longrightarrow} \ov{J}_X(\un q'), \\
%I & \mapsto I\otimes L.
%\end{aligned}$$
%\end{defi}

%Note however that, in general,  there could be  fine compactified Jacobians that are isomorphic without being equivalent by translation, see \S  \ref{S:genus1} for some explicit examples.

%\begin{lemma}\label{L:finite-eq}
%Let $X$ be a connected curve. There is a finite number of fine compactified Jacobians up to equivalence by translation.
%\end{lemma}

\subsection{Abel maps}\label{S:Abel}

In this subsection, we review, for later use, the construction and main properties of (twisted) Abel maps of degree one into fine compactified Jacobians, following \cite[\S 6]{MRV1}.

To this aim, we restrict ourselves to a connected  reduced curve $X$ satisfying condition $(\dagger)$, as in  \S\ref{N:sep-node}.
Let $\{n_1,\ldots$ $ ,n_{r-1}\}$ be the separating points of $X$, which are nodes by assumption.   Denote by $\wt{X}$ the partial normalization of $X$ at the set $\{n_1,\ldots,n_{r-1}\}$.
Since each $n_i$ is a node, the curve $\wt{X}$ is a disjoint union of $r$ connected reduced curves $\{Y_1,\ldots,Y_r\}$ such that each $Y_i$ does not have separating points.
We have a natural morphism
\begin{equation*}
\tau:\wt{X}=\coprod_i Y_i\to X.
\end{equation*}
We can naturally identify each $Y_i$ with a subcurve of $X$ in such a way that their union  is $X$ and that they do not have common
 irreducible components. We call the components $Y_i$ (or their image in $X$) the \emph{separating blocks}  of $X$.

\begin{thm}\label{T:Abel}
Let $X$ be a connected reduced curve satisfying condition $(\dagger)$.
\begin{enumerate}[(i)]
\item \label{T:Abel1} The pull-back map
$$\begin{aligned}
\tau^*: \bJbar_X & \longrightarrow \prod_{i=1}^r \bJbar_{Y_i}\\
I & \mapsto (I_{|Y_1},\ldots, I_{|Y_r}),
\end{aligned}
$$
is an isomorphism. Moreover, given any fine compactified Jacobians $\ov{J}_{Y_i}(\un q^i)$ on $Y_i$, $i=1,\dots,r$, there exists a (uniquely determined) fine compactified Jacobian $\J_X(\un q)$ on $X$ such that
\begin{equation*}
\tau^*: \ov{J}_X(\un q)\xrightarrow{\cong} \prod_i \ov{J}_{Y_i}(\un q^i),
\end{equation*}
and every fine compactified Jacobian on $X$ is obtained in this way.

\item \label{T:Abel2}
For every $L\in \Pic(X)$, there exists a unique morphism $A_L:X\to \bJbar_X^{\chi(L)-1}$  such that for every $1\leq i\leq r$ and every $p\in Y_i$ it holds
\begin{equation*}
\tau^*(A_L(p))=(M_1^i,\ldots,M_{i-1}^i,\m_p \otimes L_{|Y_i},M_{i+1}^i,\ldots,M_r^i)
\end{equation*}
for some (uniquely determined) elements $M_j^i\in \bJbar_{Y_j}$ for $j\neq i$, where $\m_p$ is the ideal of the point $p$ in $Y_i$.

\item \label{T:Abel3} If, moreover, $X$ is Gorenstein, then for every $L\in \Pic(X)$ there exists a general polarization $\un q$ with $|\un q|=\chi(L)-1$ such that
$\Im A_L\subseteq \J_X(\un q)$.

\item \label{T:Abel4} For every $L\in \Pic(X)$, the morphism $A_L$ is an embedding away from the separating blocks of arithmetic genus zero (which are isomorphic to $\PP^1$) while it contracts each rational separating block $Y_i\cong \PP^1$
 into a seminormal point of $A_L(X)$, i.e. an ordinary singularity with linearly independent tangent directions.

\end{enumerate}
\end{thm}
\begin{proof}
See \cite[Thm. D]{MRV1}.
\end{proof}

The map $A_L$ in Theorem \ref{T:Abel}\eqref{T:Abel2} is called the \emph{(L-twisted) Abel map} of $X$. Fine compactified Jacobians $\J_X(\un q)$ for which there exists $L\in \Pic(X)$ with the property that $\Im A_L\subseteq \J_X(\un q)$  are said to   \emph{admit an Abel map}.  Theorem \ref{T:Abel}\eqref{T:Abel3} says that any connected reduced Gorenstein curve has  some fine compactified Jacobians which admit an Abel map. However, not every fine compactified Jacobian of $X$ (even for a nodal curve) admits an Abel map, see \cite[Sec. 7]{MRV1} for some examples.

Note that if $\ov{J}_X(\un q)$ is a fine compactified Jacobian of $X$  and $L\in \Pic(X)$ is such that $\Im A_L\subseteq \ov{J}_X(\un q)$, then
the $L$-twisted Abel map $A_L:X\to \ov{J}_X(\un q)\subseteq \bJbar_X$ induces via pull-back a homomorphism
\begin{equation}\label{E:pullback-Abel}
A_L^*:\Pic(\ov{J}_X(\un q))\to \Pic(X)=\bJ_X,
\end{equation}
which clearly sends $\Pic^o(\ov{J}_X(\un q))$ into $\Pic^o(X)=J(X)$.

\section{Universal fine compactified Jacobians}\label{S:univ-Jac}

The aim of this subsection is to review the definition and main properties of the universal fine compactified Jacobians, following \cite[\S 4-5]{MRV1}.

\subsection{Deformation theory of $X$}\label{S:DefX}

We start by recalling in this subsection some well-known facts about the deformation theory of a (reduced) curve $X$. For basic facts  on deformation theory, we refer to the book of Sernesi \cite{Ser}.

Let $\Def_X$ be the deformation functor of $X$.
%and, for any $p\in X_{\rm sing}$, we denote by $\Def_{X,p}$ the deformation functor of the complete local $k$-algebra $\wh{\O}_{X,p}$.
%Note that there is a natural morphism of functors (see \cite[\S 4.1(d)]{Rim})
%\begin{equation}\label{E:mor-func}
%\Def_X\to \Def_X^{\rm loc}:=\prod_{p\in X_{\rm sing}} \Def_{X,p}.
%\end{equation}
%The following result is well known.
%\begin{fact}\label{F:for-smooth}
%If $X$ is a reduced curve with l.c.i. singularities, then the functors  $\Def_X$ and $\Def_{X}^{\rm loc}$ are smooth and the morphism $\displaystyle \Def_X\to \Def_{X}^{\rm loc}$ is  smooth.
%\end{fact}
%\begin{proof}
%See \cite[Cor. 4.13]{Rim}.
%\end{proof}
According to \cite[Cor. 2.4.2]{Ser}, the functor $\Def_X$ admits a semiuniversal \footnote{Some authors use the word miniversal instead of semiuniversal. We prefer to use  the word semiuniversal in order to be coherent with the terminology of the book of Sernesi \cite{Ser}.} formal couple
$(R_X,\ov{\X})$, where $R_X$ is a Noetherian complete local $k$-algebra with maximal ideal $\m_{X}$ and residue field
$k$ and
$$\ov{\X}\in \wh{\Def_X}(R_X):=\varprojlim \Def_X\left(\frac{R_X}{\m_{X}^n}\right)$$
is a formal deformation of $X$ over $R_X$. Recall that this means that the morphism of functors
\begin{equation}\label{E:map-func1}
h_{R_X}:=\Hom(R_X,-)\longrightarrow \Def_X
\end{equation}
determined by $\ov{\X}$ is smooth and induces an isomorphism
of tangent spaces $T R_X:=(\m_X/\m_X^2)^{\vee}\stackrel{\cong}{\to} T \Def_X$ (see \cite[Sec. 2.2]{Ser}).
The formal couple $(R_X,\ov{\X})$ can be also viewed as a flat morphism of formal schemes
\begin{equation}\label{E:form-fam}
\ov{\pi}:{\ov \X}\to \Spf\: R_X,
\end{equation}
where $\Spf$ denotes the formal spectrum, such that the fiber over $o:=[\m_X]\in \Spf  R_X$ is isomorphic to $X$
(see \cite[p. 77]{Ser}).
Note that the semiuniversal formal couple $(R_X,\ov{\X})$ is unique by  \cite[Prop. 2.2.7]{Ser}.

Since $X$ is projective and $H^2(X,\O_X)=0$, Grothendieck's existence theorem (see \cite[Thm. 2.5.13]{Ser}) gives that
the formal deformation \eqref{E:form-fam} is \emph{effective}, i.e. there exists a deformation
$\pi:\X\to \Spec R_X$ of $X$ over $\Spec R_X$  whose completion along $X=\pi^{-1}(o)$
is isomorphic to \eqref{E:form-fam}.  In other words, we have a Cartesian diagram
\begin{equation}\label{E:eff-fam}
\xymatrix{
X \ar[d]\ar@{^{(}->}[r]\ar@{}[dr]|{\square}& \ov{\X}\ar[r]\ar[d]^{\ov{\pi}}\ar@{}[dr]|{\square} & \X\ar[d]^{\pi}\\
\Spec k\cong o \ar@{^{(}->}[r] & \Spf R_X\ar[r] &  \Spec R_X.
}
\end{equation}
Note also that the deformation $\pi$ is unique by \cite[Thm. 2.5.11]{Ser}. For later reference, we collect the properties of the effective semiuniversal deformation morphism $\pi:\X\to \Spec R_X$ into  the following:

\begin{lemma}\label{L:morpi}
Let $X$ be a (reduced and connected) curve.
\begin{enumerate}[(i)]
\item \label{L:morpi1} The effective semiuniversal deformation  $\pi:\X\to \Spec R_X$ is a flat and projective morphism with geometrically reduced and geometrically connected fibers.
\item \label{L:morpi2}  If $X$ has l.c.i. singularities then $R_X$ is a power series ring (hence $\Spec R_X$ is irreducible) and the generic fiber of $\pi$ is smooth.
\item Assume that $X$ has locally planar singularities. Then the following hold true:
\begin{enumerate}[(a)]
\item \label{L:morpinew} All the  fibers of $\pi$ have locally planar singularities.
\item \label{L:morpi3} Let $U$ be the open subset of $\Spec R_X$ consisting of all the (schematic) points $s\in \Spec R_X$ such that
the geometric fiber $\X_{\ov s}$ of the universal family $\pi:\X\to \Spec R_X$ is smooth or has a unique singular point that  is a
node. Then the codimension of the complement of $U$ inside $\Spec R_X$ is at least two.
\end{enumerate}
\end{enumerate}
\end{lemma}
\begin{proof}
Part \eqref{L:morpi1}: the fact that $\pi$ is flat is part of the definition of a deformation and the fact that $\pi$ is projective follows directly from the proof of Grothendieck's existence theorem (see \cite[Thm. 2.5.13]{Ser}) using that the central fiber $X$ is projective. Since the central fiber $X$ is (geometrically) reduced  and the property of having geometrically reduced fibers is open for a flat, proper morphism of finite presentation by \cite[Thm. 12.2.4(v)]{EGAIV3}, it follows that all the  fibers of $\pi$ are geometrically reduced.
Moreover, the fibers of $\pi$ are geometrically connected because $X$ is (geometrically) connected and the number of geometric connected components of the fibers is locally constant for a universally open (e.g. flat) and proper
morphism by \cite[Prop. 15.5.7]{EGAIV3}.

Part \eqref{L:morpi2}: by the definition of semiuniversal deformation ring and \cite[Thm. C.4]{Ser}, the ring  $R_X$ is a power series ring if and only if $\Def_X$ is  smooth; this last property does hold true if $X$ has
locally planar singularities (see e.g. \cite[Cor. 4.13]{Rim}). The  generic fiber of $\pi$ is smooth because a reduced curve is smoothable if and only if it has locally formally smoothable singularities
(see \cite[Cor. 29.10]{Har-ter}) and l.c.i. singularities are locally formally smoothable (see \cite[Ex. 29.0.1]{Har-ter}). For another proof of the last statement, see \cite[Prop. 4.1.1]{Lau}.

Part \eqref{L:morpinew}: this follows from the well-known fact that the property of having locally planar singularities is open in a projective family of curves, see e.g. \cite[Proof of Prop. 3.5]{MY}.

Part \eqref{L:morpi3}: see \cite[Lemma 4.3]{MRV1}.
\end{proof}

The space $\Spec R_X$ admits two stratifications into closed subsets according to either the arithmetic genus or the geometric genus of the
normalization of the fibers of the family $\pi$. More precisely, using the notation introduced in \S\ref{N:curves}, we have two functions
\begin{equation}
\begin{aligned}
p_a^{\nu}: \Spec R_X & \longrightarrow \bbN,\\
 s & \mapsto p_a^{\nu}(\X_{\ov s}):=p_a(\X_{\ov s}^{\nu}),\\
\end{aligned}
\hspace{2cm}
\begin{aligned}
g^{\nu}: \Spec R_X & \longrightarrow \bbN,\\
s & \mapsto g^{\nu}(\X_{\ov s})=g^{\nu}(\X_{\ov s}^{\nu}),
\end{aligned}
\end{equation}
where $\X_s:=\pi^{-1}(s)$ is the fiber of $\pi$ over the (schematic) point $s\in \Spec R_X$ and  $\X_{\ov s}:=\X_s\times_{k(s)}\ov{k(s)}$ is the geometric fiber over $s$.
%of $\pi$ over the point $s\in \Spec R_X$.
Since  the number of connected components of $\X_{\ov s}^{\nu}$ is the number $\gamma(\X_{\ov s})$
of irreducible components of $\X_{\ov s}$, we have the relation
\begin{equation}\label{E:geo-ari}
p_a^{\nu}(\X_{\ov s})=g^{\nu}(\X_{\ov s})-\gamma(\X_{\ov s})+1\leq g^{\nu}(\X_{\ov s}).
\end{equation}

\begin{lemma}\label{L:lower-semcont}
The functions $p_a^{\nu}$ and $g^{\nu}$ are lower semi-continuous.
\end{lemma}
\begin{proof}
This is known to the experts: a proof over the complex numbers can be
found in \cite[I. Thm. 1.3.2]{Tei}, \cite[Prop. 2.4]{DH} or \cite[Chap. II, Thm. 2.54]{GLS}; a proof over an arbitrary field for integral curves
(in which case $p_a^{\nu}=g^{\nu}$) can be found in \cite[Prop. A.2.1]{Lau}.
See \cite{RV} for a complete proof in our more general setting.

\end{proof}

Using the above Lemma, formula \eqref{E:geo-ari} and the fact that the arithmetic genus $p_a$ stays constant in the family $\pi$ because of flatness, we get
that
$$
p_a(X^{\nu})=p_a^{\nu}(X)\leq p_a^{\nu}(\X_{\ov s})\leq g^{\nu}(\X_{\ov s})\leq p_a(\X_{\ov s})=p_a(X).
$$
Therefore for any $p_a(X^{\nu})\leq l\leq p_a(X)$ we have two closed subsets of $\Spec R_X$:
\begin{equation}\label{E:strata-RX}
(\Spec R_X)^{g^{\nu}\leq l}:=\{s\in \Spec R_X \: : \: g^{\nu}(\X_{\ov s})\leq l \}\subseteq (\Spec R_X)^{p_a^{\nu}\leq l}:=\{s\in \Spec R_X \: : \: p_a^{\nu}(\X_{\ov s})\leq l \}.
\end{equation}
If $X$ has locally planar singularities, then the stratification by the arithmetic genus of the normalization (which is sometimes called the \emph{equigeneric stratification}) is particularly well-behaved. 

% In particular, we have the following result which is due to Teissier and Diaz-Harris in characteristic zero (see \cite[Prop. (4.17), Thm. (4.15)]{DH} or \cite[Chap. II]{GLS}); we present here a new proof that is valid in any characteristic.

\begin{thm}\label{T:strati}
Assume that $X$ has locally planar singularities. Then, for any $p_a(X^{\nu})\leq l\leq p_a(X)$, we have that:
\begin{enumerate}[(i)]
\item \label{T:strati1} The closed subset
$(\Spec R_X)^{p_a^{\nu}\leq l}\subset \Spec R_X$ has codimension at least $p_a(X)-l$.
\item \label{T:strati2} Each generic point $\eta$ of $(\Spec R_X)^{p_a^{\nu}\leq l}$ is such that $\X_{\ov \eta}$ is a nodal curve.
    %with at least $p_a(X)-l$ nodes.
\end{enumerate}
\end{thm}

Part \eqref{T:strati1} of the above Theorem follows over $k=\bbC$ from \cite[Thm. 4.15, Prop. 4.17]{DH}  and over an algebraically closed field $k$ of characteristic $0$ or bigger than the maximum of the
multiplicities of the points of $X$ from \cite[Prop. 3.5]{MY}. 
% A proof over an algebraically closed field of arbitrary characteristics seems to be new.
Part \eqref{T:strati2}  \emph{est bien connue mais ne semble \^etre d\'emontr\'ee nulle part} (not even for $k=\bbC$!) as Laumon points out in the sentence preceding \cite[Thm. A.4.2]{Lau}. The result is certainly well known to the experts and it has been used many times in the literature (see e.g. \cite{MS}, \cite{MSV}). We will give a complete proof of the above Theorem in \cite{RV}.

\vspace{0.1cm}

From the above Theorem \ref{T:strati} together with the inclusion in \eqref{E:strata-RX}, we get the following
\begin{cor}\label{C:Diaz-Har}
Assume that $X$ has locally planar singularities. Then, for any $p_a(X^{\nu})\leq k\leq p_a(X)$,
the codimension of the closed subset
$(\Spec R_X)^{g^{\nu}\leq k}$ inside $\Spec R_X$ is at least $p_a(X)-k$.
\end{cor}

%VECCHIO LEMMA (ora incorporato nel nuovo Lemma 3.2)

%We end this subsection by recalling the  following result, which will be used later on.

%\begin{lemma}\label{L:codim1}
%Assume that $X$ has locally planar singularities. Let $U$ be the open subset of $\Spec R_X$ consisting of all the (schematic) points $s\in \Spec R_X$ such that the geometric fiber $\X_{\ov s}$ of the universal family $\pi:\X\to \Spec R_X$ is smooth or has a unique singular point that  is a node. Then the codimension of the complement of $U$ inside $\Spec R_X$ is at least two.
%\end{lemma}
%\begin{proof}
%See \cite[Lemma 4.3]{MRV1}.
%\end{proof}

\subsection{Universal fine compactified Jacobians}\label{S:unJac}

In this subsection, we introduce the universal fine compactified Jacobians relative to the effective semiuniversal deformation $\pi:\X\to \Spec R_X$ introduced in  \S \ref{S:DefX}.
To this aim, consider the functor
$$\bJbar_{\X}^*:\{\Spec R_X-\text{schemes} \} \longrightarrow \{ \text{Sets} \}$$
which sends a scheme $T\to \Spec R_X$ to the set of isomorphism classes of $T$-flat, coherent sheaves on
$\X_T:=T\times_{\Spec R_X} \X$ whose fibers over $T$ are simple rank-1  torsion-free sheaves. The functor
$\bJbar_{\X}^*$ contains the open subfunctor
$$\bJ_{\X}^*:\{\Spec R_X-\text{schemes} \} \longrightarrow \{ \text{Sets} \}$$
which sends a scheme $T\to \Spec R_X$ to the set of isomorphism classes of line bundles on
$\X_T$.

Analogously to Fact \ref{F:huge-Jac}, we have the following

\begin{fact}[Altman-Kleiman, Esteves]\label{F:univ-Jac}
\noindent
\begin{enumerate}[(i)]
\item \label{F:univ-Jac1}The \'etale sheafification of $\bJbar_{\X}^*$ is represented by
a scheme $\bJbar_{\X}$ endowed with a morphism
$u:\bJbar_{\X}\to \Spec R_X$, which is locally of finite type and satisfies the existence part of the valuative criterion for properness.
The scheme $\bJbar_{\X}$ contains an open subset $\bJ_{\X}$ which represents the \'etale sheafification of
$\bJ_{\X}^*$ and the restriction $u:\bJ_{\X}\to \Spec R_X$ is formally smooth.\\
Moreover, the geometric fiber of $\bJbar_{\X}$ (resp. of $\bJ_{\X}$) over any point $s\in \Spec R_X$ is isomorphic to $\bJbar_{\X_s}$ (resp. $\bJ_{\X_s}$).

\item \label{F:univ-Jac2} There exists a sheaf $\wh{\I}$ on $\X\times_{\Spec R_X} \bJbar_{\X}$  such that for every $\F\in \bJbar_{\X}^*(T)$ there exists a unique $\Spec R_X$-map $\alpha_{\F}:T\to \bJbar_{\X}$ with the property that $\F=(\id_{\X}\times \alpha_{\F})^*(\wh{\I})\otimes \pi_2^*(N)$ for some $N\in \Pic(T)$, where
$\pi_2:\X\times_{\Spec R_X} T\to T$ is the projection onto the second factor.
The sheaf $\wh{\I}$ is uniquely determined up  to tensor product with the pullback of an invertible sheaf on $\bJbar_{\X}$ and it is called a \emph{universal sheaf} on $\bJbar_{\X}$.\\
Moreover, the restriction of $\wh{\I}$ to $X\times \bJbar_X$ is equal to a universal sheaf as in Fact \ref{F:huge-Jac}\eqref{F:huge3}.

\end{enumerate}
\end{fact}
\begin{proof}
See \cite[Fact 4.4]{MRV1} and the references therein.
\end{proof}

In \cite[Thm. 4.5]{MRV1}, the authors proved that  the completed local ring of $\bJbar_{\X}$ at a point $I$ of the central fiber $u^{-1}([\m_X])=\bJbar_X$ is a semiuniversal deformation ring
for the deformation functor $\Def_{(X,I)}$ of the pair $(X,I)$. Applying a result of   Fantechi-G\"ottsche-van Straten \cite{FGvS} which says that $\Def_{(X,I)}$ is unobstructed if $X$ has locally planar singularities, we get the following

\begin{thm}\label{T:reg-univ}
Let $X$ be a connected reduced curve with locally planar singularities. Then  the scheme $\bJbar_{\X}$ is regular.
\end{thm}
\begin{proof}
See \cite[Thm. 4.5(iii)]{MRV1}.
\end{proof}

The universal fine compactified Jacobians will be certain open subschemes of $\bJbar_X$, proper over $\Spec R_X$, whose definition will depend on
 a general polarization $\un q$ on $X$, see  Definition \ref{def-int}. Indeed, the polarization $\un q$ induces
a polarization on each fiber of the effective semiuniversal deformation family
$\pi:\X\to \Spec R_X$, in the following way. Recall that for any (schematic) point $s\in \Spec R_X$, we denote by $\X_s:=\pi^{-1}(s)$ the fiber of $\pi$ over $s$
and by $\X_{\ov s}:=\X_s\times_{k(s)} \ov{k(s)}$ the geometric fiber over $s$.
%where $k(s)$ is the residue field of $s$ and $\ov{k(s)}$ is its algebraic closure.

There is a natural specialization map
\begin{equation}\label{E:map-subcurve}
\begin{aligned}
\Sigma_s:\{\text{Subcurves of } \X_{\ov s}\} & \longrightarrow \{\text{Subcurves of } X\}\\
\X_{\ov s}\supseteq Z & \mapsto \ov{Z}\cap X\subseteq X,
\end{aligned}
\end{equation}
where $\ov{Z}$ denotes the Zariski closure  inside $\X$ of the image of $Z$ under the natural morphism $\X_{\ov{s}}\to \X_s\hookrightarrow \X$ and the intersection $\ov{Z}\cap X$ is endowed with the reduced scheme structure.
The properties of this map are studied in \cite[Sec. 5]{MRV1}. Here we notice that the function $Z\mapsto \delta_Z=|Z\cap Z^c|$ is lower semicontinuous with respect to $\Sigma_s$\footnote{Indeed, we strongly believe that the function
$Z\mapsto \delta_Z$ is invariant under the map $\Sigma_s$, but we do not know how to prove this and also we do not need this stronger result.}. More precisely, we have

\begin{lemma}\label{L:seminf-del}
Assume that $X$ is a  (reduced and connected) curve with locally planar singularities. For any  subcurve $Z$ of $\X_{\ov s}$ we have that $\delta_{\Sigma_s(Z)}\leq \delta_Z$.
\end{lemma}
\begin{proof}
Consider the relative dualizing sheaf $\omega_{\pi}$ of the family $\pi:\X\to \Spec R_X$.
 %which can be characterized as the unique coherent sheaf on $\X$ such that its restriction to every geometric fiber $\X_{\ov s}$ is the dualizing sheaf of $\X_{\ov s}$ (since there are no non-trivial line bundles on $\Spec R_X$).
Since $X$ has locally planar singularities by assumption, then all the geometric fibers of $\pi$ have locally planar singularities  by Lemma \ref{L:morpi}\eqref{L:morpinew}. In particular, all the geometric fibers of $\pi$ are Gorenstein, which implies that $\omega_{\pi}$ is a line bundle and its restriction to every geometric fiber $\X_{\ov s}$ is the dualizing sheaf of $\X_{\ov s}$.

The Claim in the proof of \cite[Thm. 5.4]{MRV1} applied to the line bundle $\omega_{\pi}$ implies that
\begin{equation}\label{E:eq-omega}
\deg_Z(\omega_{\X_{\ov s}})=\deg_{\Sigma_s(Z)}(\omega_X).
\end{equation}
On the other hand, the subcurve $Z$ has also locally planar singularities (hence it is Gorenstein), and the adjunction formula  \cite[Lemma 1.12]{Cat}
%\footnote{which holds for any Gorenstein subcurve of a Gorenstein curve.})
%The adjunction formula, in the form presented here, is true even without the hypothesis that the subcurve in question is Gorenstein, as a quick inspection of the proof of loc. cit. reveals.})
gives that
\begin{equation}\label{E:adj-for}
\begin{sis}
& \deg_Z(\omega_{\X_{\ov s}})=2 p_a(Z)-2+\delta_Z, \\
& \deg_{\Sigma_s(Z)}(\omega_X)=2p_a(\Sigma_s(Z))-2+\delta_{\Sigma_s(Z)}.
\end{sis}
\end{equation}
Consider now the image $Y$ of $Z\subseteq \X_{\ov s}$ in the (usual) fiber $\X_s$ over $s$. Since the irreducible components of  $\X_{ s}$ are geometrically integral by Lemma \ref{L:morpi}\eqref{L:morpi1} and \cite[Lemma 5.1]{MRV1}, we get that
\begin{equation}\label{E:gen-basech}
p_a(Y)=p_a(Z).
\end{equation}

Now, arguing as in Step II of the proof of Lemma \ref{L:lower-semcont}, we can find a discrete valuation ring $R$ with residue field $k$, together with a morphism  $f: \Spec R\to  \ov{\{s\}}$ that maps the generic point $\eta$ of $\Spec R$ to $s\in \Spec R_X$ and the special point $0$ of $\Spec R$ to $[\m_X]\in \Spec R_X$.
Denote by $\tau: \Y\to \Spec R$ the pull-back of  $\pi:\ov{Y}\to \ov{\{s\}}$ via the morphism $f$
%Note that, by the definition of the map $\Sigma_s$, the special fiber $\Y_0$ coincides set-theoretically with the reduced curve $\Sigma_s(Z)$.
and consider the closure $\cZ:=\ov{\Y_{\eta}}$ of the generic fiber $\Y_{\eta}$  inside $\Y$, i.e. the unique closed subscheme $\cZ\subseteq \Y$ which is flat over $\Spec R$ and such that its generic fiber $\cZ_{\eta}$ is equal to $\Y_{\eta}$ (see \cite[Prop. 2.8.5]{EGAIV2}). Since the arithmetic genus is constant for a flat and proper family of curves, we deduce that the arithmetic genus of the special fiber $\cZ_0$ of $\cZ$ satisfies
\begin{equation}\label{E:flat-g}
p_a(\cZ_0)=p_a(\cZ_{\eta})=p_a(\Y_{\eta})=p_a(Y).
\end{equation}
As proved in \cite[Proof of the Claim in Theorem 5.4]{MRV1},  the 1-cycle associated to $\cZ_0$ coincides with the 1-cycle associated to $\Sigma_s(Z)$,
%the special fiber $\cZ_0$ is a generically reduced (possibly non-reduced) curve whose underlying reduced curve is $\Sigma_s(Z)$,
or in other words $\cZ_0$ coincides with the reduced curve $\Sigma_s(Z)$ except for the possible presence of embedded points.
Since the presence of embedding points decreases the arithmetic genus, we get that
\begin{equation}\label{E:go-down}
p_a(\cZ_0)\leq p_a(\Sigma_s(Z)).
\end{equation}
Now we conclude that $\delta_{\Sigma_s(Z)}\leq \delta_Z$ by putting together \eqref{E:eq-omega}, \eqref{E:adj-for}, \eqref{E:gen-basech}, \eqref{E:flat-g} and \eqref{E:go-down}.
\end{proof}

From the above Lemma, we deduce a Corollary that will be used in what follows (see Theorem \ref{T:thmD-Abel}).

\begin{cor}\label{C:sep-point}
 Assume that $X$ is a  (reduced and connected) curve with locally planar singularities. If $X$ does not have separating nodes then every geometric fiber $\X_{\ov s}$ of the effective semiuniversal deformation $\pi:\X\to \Spec R_X$ does not have
separating nodes.
\end{cor} 
\begin{proof}
Assume that $\X_s$ has a separating node $p$ and we are going to show that $X$ has also a separating node.
By \S\ref{N:sep-node}, there exists a subcurve $Z$ of $\X_{\ov s}$ whose scheme-theoretic intersection with the complementary subcurve $Z^c$ is equal to $\{p\}$. In particular, $\delta_Z=1$. Consider now the subcurve $\Sigma_s(Z)$ of $X$. Lemma \ref{L:seminf-del} implies $\delta_{\Sigma_s(Z)}\leq \delta_Z=1$. However, since $X$ is connected and $\Sigma_s(Z)$ is a non-trivial subcurve (because $\Sigma_s(Z)^c=\Sigma_s(Z^c)\neq \emptyset$ by \cite[Lemma 5.2]{MRV1}), we should have $\delta_{\Sigma_s(Z)}\neq 0$, which forces then $\delta_{\Sigma_s(Z)}=1$.
Since $X$ is Gorenstein, condition $(\dagger)$ of \S\ref{N:sep-node} implies that $\Sigma_s(Z)$ intersects scheme-theoretically its complementary subcurve $\Sigma_s(Z)^c$ into a  separating node $q$ of $X$, q.e.d.
\end{proof}

Using the specialization map $\Sigma_s$, we can show that a polarization on $X$ induces a polarization on each geometric fiber $\X_{\ov s}$.

\begin{lemdef}\label{D:def-pola}
Let $s\in \Spec R_X$ and let $\un q$ be a polarization on $X$. The polarization $\un q^s$ induced by $\un q$ on the geometric fiber
$\X_{\ov s}$ is defined by
$$\un q^s_Z:=\un q_{\Sigma_s(Z)}\in \Q$$
for every subcurve $Z\subseteq \X_{\ov s}$. If $\un q$ is general then $\un q^s$ is general.
\end{lemdef}
\begin{proof}
See \cite[Lemma-Definition 5.1]{MRV1}.
\end{proof}

Given a general polarization $\un q$ on $X$, by the next theorem we get an open subset of $\bJbar_X$ which is proper
over $\Spec R_X$.

\begin{thm}\label{T:univ-fine}
Let $\un q$ be a general polarization on $X$. Then there exists an open subscheme $\J_{\X}(\un q)\subseteq \bJbar_{\X}$
which is projective over $\Spec R_X$ and such that the geometric fiber of  $u:\J_{\X}(\un q)\to \Spec R_X$ over a point
$s\in \Spec R_X$ is isomorphic to $\J_{\X_{\ov s}}(\un q^s)$. In particular, the fiber of $\J_{\X}(\un q)\to \Spec R_X$
over the closed point $[\m_X]\in \Spec R_X$ is isomorphic to $\J_X(\un q)$.
\end{thm}
We call the scheme $\J_{\X}(\un q)$ the \emph{universal fine compactified Jacobian} of $X$ with respect to the polarization $\un q$. We denote by $J_{\X}(\un q)$ the open subset of $\J_{\X}(\un q)$ parametrizing line bundles,
i.e. $J_{\X}(\un q)=\J_{\X}(\un q)\cap \bJ_{\X}\subseteq \bJbar_{\X}$.

\begin{proof}
See \cite[Thm. 5.2]{MRV1}.
\end{proof}

If the curve $X$ has locally planar singularities, then the universal fine compactified Jacobians of $X$ have several nice properties that we collect in the following statement.

\begin{thm}\label{T:univ-Jac}
Assume that $X$ has locally planar singularities and let $\un q$ be a general polarization on $X$.
Then we have:
\begin{enumerate}[(i)]
\item \label{T:univ-Jac1} The scheme $\J_{\X}(\un q)$ is regular and irreducible.
\item \label{T:univ-Jac2} The surjective map $u:\J_{\X}(\un q)\to \Spec R_X$ is projective and flat of relative dimension $p_a(X)$.
\item \label{T:univ-Jac4} The relative dualizing sheaf of $u$ is  trivial.
\item \label{T:univ-Jac3} The smooth locus of $u$ is $J_{\X}(\un q)$.
\end{enumerate}
\end{thm}
\begin{proof}
See \cite[Thm. C]{MRV1}.
\end{proof}

%Note that a statement similar to Corollary \ref{C:triv-can} was proved  by Arinkin in \cite[Cor. 9]{arin1} for the universal compactified Jacobian over the moduli stack of \emph{integral} curves with locally planar singularities.

Finally, note that the universal fine compactified Jacobians are acted upon by the universal generalized Jacobian, whose properties are collected into the following

\begin{fact}[Bosch-L\"utkebohmert-Raynaud]\label{F:ungenJac}
There is an open subset of $\bJ_{\X}$, called the \emph{universal generalized Jacobian} of $\pi:\X\to \Spec R_X$
and denoted by $v: J(\X) \to \Spec R_X$, whose geometric fiber over any point $s\in \Spec R_X$ is the generalized Jacobian
$J(\X_{\ov s})$ of the geometric fiber $\X_{\ov s}$ of $\pi$ over $s$.

The morphism $v$ makes $J(\X)$ into a smooth and separated group scheme of finite type over $\Spec R_X$.
\end{fact}
\begin{proof}
The existence of a group scheme $v:J(\X)\to \Spec R_X$ whose fibers are the generalized Jacobians of the fibers of $\pi:\X\to \Spec R_X$
follows by \cite[Sec. 9.3, Thm. 7]{BLR}, which can be applied since $\Spec R_X$
is a strictly henselian local scheme (because $R_X$ is a complete local ring) and the geometric fibers of $\pi:\X\to \Spec R_X$ are reduced and connected since $X$ is assumed
to be so. The result of loc. cit. gives also that the map $v$ is smooth, separated and of finite type.
% while the fact that it is of finite type follows from a theorem of Kleiman (see e.g. \cite[Sec. 8.4, Thm. 4]{BLR}).
\end{proof}

We denote by $\zeta: \Spec R_X\to J(\X)$ the zero section of the group scheme $v: J(\X)\to \Spec R_X$; in other words, $\zeta$ is the morphism which sends a geometric point $\ov{s}$ lying over
a point $s\in \Spec R_X$ into the trivial line bundle on the geometric fiber  $\X_{\ov s}$ of $\pi:\X\to \Spec R_X$ over $s$.

\section{The Picard scheme of the universal fine compactified Jacobians}\label{S:Pic-univ}

The aim of this section is to discuss the properties of the Picard scheme of the universal fine compactified Jacobians $u:\ov{J}_{\X}(\un q)\to \Spec R_X$,
introduced in \S \ref{S:univ-Jac}.  Following \cite[Sec. 9.2]{FGA} and \cite[Sec. 8.1]{BLR}, we define the \emph{relative Picard functor} $\Pic_u$
as the fppf-sheaf  associated  to the contravariant Picard functor
$$\begin{aligned}
\Pic_u: {\rm Sch}/R_X & \to {\rm Grps},\\
T & \mapsto {\mathcal Pic}(\ov{J}_{\X}(\un q)\times_{\Spec R_X} T),
\end{aligned}
$$
where ${\rm Sch}/R_X$ is the category of schemes over $\Spec R_X$, ${\rm Grps}$ is the category of abelian groups and $\mathcal Pic$ denotes the Picard group as defined in \S\ref{N:Pic-field}.

Following \cite[Sec. 9.5 and 9.6]{FGA} and \cite[Sec. 8.4]{BLR}, we consider the two subfunctors of the relative Picard functor
\begin{equation}\label{E:Pic-func}
\Pic^o_u\subseteq \Pic_u^{\tau}\subseteq \Pic_u,
\end{equation}
such that $\Pic^o_u$ (resp. $\Pic_u^{\tau}$) consists of the elements of $\Pic_u$ whose restriction to every fiber $u^{-1}(s)$ for $s\in \Spec R_X$ belongs
to $\Pic^o(u^{-1}(s))$ (resp. $\Pic^{\tau}(u^{-1}(s))$), see \S\ref{N:Pic-field}.

We summarize the properties of the above functors in the following

\begin{thm}\label{T:Pic-univ}
Let $X$ be a curve with locally planar singularities of arithmetic genus $p_a(X)$ and let $\un q$ be a general polarization on $X$.
\begin{enumerate}[(i)]
\item \label{T:Pic-univ1} $\Pic_u$ is represented by a group scheme $\Pic(\ov{J}_{\X}(\un q))$ locally of finite type over $\Spec R_X$.
\item \label{T:Pic-univ2} $\Pic_u^{\tau}$ is represented by an open subgroup scheme $\Pic^{\tau}(\ov{J}_{\X}(\un q))\subseteq \Pic(\ov{J}_{\X}(\un q))$
which is of finite type and separated over $\Spec R_X$.
%\item \label{T:Pic-univ3} If $h^1(\ov{J}_{X}(\un q),\O_{\ov{J}_{X}(\un q)})=p_a(X)$, then $\Pic(\ov{J}_{\X}(\un q))$ is formally smooth over $\Spec R_X$.
\item \label{T:Pic-univ4} Assume that $h^1(\ov{J}_{X}(\un q),\O_{\ov{J}_{X}(\un q)})=p_a(X)$. Then $\Pic_u^o$ is represented by an open subgroup scheme
$\Pic^o(\ov{J}_{\X}(\un q))\subseteq \Pic^{\tau}(\ov{J}_{\X}(\un q))$, and both of them are of finite type, separated and smooth over $\Spec R_X$.
\end{enumerate}
\end{thm}
\begin{proof}
Observe that  the morphism $u$ is projective and flat by Theorem \ref{T:univ-Jac}\eqref{T:univ-Jac2} and
the fibers of $u$ are geometrically connected by Theorem \ref{T:compJac}\eqref{T:compJac3}.  Therefore, part \eqref{T:Pic-univ1} will follow from
Mumford's representability criterion for the Picard scheme (see \cite[Sec. 8.2, Thm. 2]{BLR} or \cite[Thm. 9.4.18.1]{FGA}) once we have proved the following

\un{Claim 1:} The irreducible components of the fibers of $u$ are geometrically irreducible.

Let $V\subseteq \ov{J}_{\X}(\un q)$ be the biggest open subset where the restriction of the morphism $u:\ov{J}_{\X}(\un q)\to \Spec R_X$ is smooth.
Since $u$ is flat, the  fiber $V_{s}$ of $V$ over a point $s\in \Spec R_X$ is the smooth locus of the fiber $u^{-1}(s)$, which is geometrically reduced because
$u^{-1}(s)\otimes_{k(s)} \ov{k(s)}\cong \J_{\X_{\ov s}}(\un q^s) $ is reduced by Theorem \ref{T:compJac}\eqref{T:compJac1}.
  In particular, $V_s:=\pi^{-1}(s)\subseteq u^{-1}(s)$ and  $V_{\ov s}:=V_s\times_{k(s)} \ov{k(s)}\subseteq u^{-1}(s)\otimes_{k(s)} \ov{k(s)}\cong \J_{\X_{\ov s}}(\un q^s) $ are dense open subsets.

Therefore, the irreducible components of $\X_s$ (resp. of $\X_{\ov s}$) are equal to the irreducible components of $u^{-1}(s)$ (resp. of $ \J_{\X_{\ov s}}(\un q^s)$). However, since $V_s$ is smooth over $k(s)$ by construction, the irreducible components of $V_s$ coincide with the connected components of $V_s$ and similarly for $V_{\ov s}$.
In conclusion, we have to show that the connected components of $V_s$ are geometrically connected for any point $s\in \Spec R_X$.

Let $C$ be a connected component of $V_s$, for some point $s\in \Spec R_X$.  The closure $\wt{C}$  of $C$ inside $ \ov{J}_{\X}(\un q)$  will contain some irreducible component of the central fiber $\ov{J}_{X}(\un q)$
by the upper semicontinuity of the dimension of the fibers (see \cite[Lemma 13.1.1]{EGAIV3}) applied to the projective surjective morphism $\wt{C}\to \ov{\{s\}}$.
Hence, the closure $\ov C$ of $C$ inside $V$ will contain some (not necessarily unique) connected component $C_o$ of the central fiber $V_o=V_{[\m_X]}$. Now, since $R_X$ is a strictly henselian ring and $V\to \Spec R_X$ is smooth,  given any point $p\in C_o\subseteq V_o$, we can find a section $\sigma$ of $V\to \Spec R_X$ passing through $p$ (see \cite[Sec. 2.2, Prop. 14]{BLR}). Clearly, $\sigma(s)$ is a $k(s)$-rational point of $C$. Therefore we conclude that $C$ is geometrically connected by  \cite[Cor. 4.5.14]{EGAIV2}, q.e.d.

\vspace{0,1cm}

% Since $R_X$ is strictly henselian (being a complete local ring with algebraically closed residue field), the fibers of $u$ are such that their irreducible
%components are geometrically irreducible, see \cite[Lemma 18]{est1}.

%Moreover, the morphism $u$ is projective and flat by Theorem \ref{T:univ-Jac}\eqref{T:univ-Jac2} and
%the fibers of $u$ are geometrically connected by Theorem \ref{T:compJac}\eqref{T:compJac3}.  Therefore, we can apply  Mumford's
%representability criterion for the Picard scheme (see \cite[Sec. 8.2, Thm. 2]{BLR} or \cite[Thm. 9.4.18.1]{FGA}) in order to deduce that the relative Picard functor $\Pic_u$ is represented
%by a group scheme $f:\Pic(\ov{J}_{\X}(\un q))\to \Spec R_X$ locally of finite type, which proves part \eqref{T:Pic-univ1}.

Let us now prove part  \eqref{T:Pic-univ2}. Since $u$ is proper and $\Pic_u$ is represented by a scheme, a result of Kleiman  \cite[Thm. 4.7]{Klei} gives that $\Pic_u^{\tau}$ is represented by an open subgroup scheme $\Pic^{\tau}(\ov{J}_{\X}(\un q))\subseteq \Pic(\ov{J}_{\X}(\un q))$ which is moreover of finite type over $\Spec R_X$.
In order to prove that $f^{\tau}: \Pic^{\tau}(\ov{J}_{\X}(\un q)) \to \Spec R_X$ is separated, it is enough to prove, using the valuative criterion of separatedness
(see \cite[Prop. 7.2.3]{EGAII}),
that for any map $\Spec R\to \Spec R_X$, where $R$ is a discrete valuation ring,
the base change  map
$$f_R^{\tau}: \Pic^{\tau}(\ov{J}_{\X}(\un q))\times_{\Spec R_X} \Spec R\cong \Pic^{\tau}(\ov{J}_{\X}(\un q)\times_{\Spec R_X} \Spec R) \to \Spec R$$
is separated. Since the fibers of $\ov{J}_{\X}(\un q)\times_{\Spec R_X} \Spec R \to \Spec R$ are geometrically reduced by Theorem \ref{T:compJac}\eqref{T:compJac1} and $R$ is a discrete valuation ring, a result of Raynaud
\cite[Cor. 6.4.5]{Ray} guarantees that the map $f_R^{\tau}$ is separated. Part \eqref{T:Pic-univ2} is now proved.

\vspace{0,1cm}

Before proving the remaining assertions, we prove the following

\un{Claim 2:} If $h^1(\ov{J}_{X}(\un q),\O_{\ov{J}_{X}(\un q)})=p_a(X)$, then the geometric fibers of $f:\Pic(\ov{J}_{\X}(\un q))\to \Spec R_X$ are smooth of dimension $p_a(X)$.

Indeed, if $s$ is the  generic point of $\Spec R_X$, then $\X_{\ov s}$ is a smooth projective curve of genus equal to $p_a(X)$ by Lemma  \ref{L:morpi}\eqref{L:morpi2} and therefore
$\ov{J}_{\X_{\ov s}}(\un q)$ is an abelian variety of dimension equal to $p_a(X)$. This implies that $\Pic(\ov{J}_{\X_{\ov s}}(\un q))$ is a smooth group scheme of dimension $p_a(X)$
over  $\ov{k(s)}$. Consider now the function
%$\un{dim}: \Pic(\ov{J}_{\X}(\un q))\to \bbN$ sending a point
$$\Pic(\ov{J}_{\X}(\un q))\ni x \mapsto \dim_x f^{-1}(f(x))\in \bbN,$$
which is   upper semi-continuous by Chevalley's theorem (see \cite[(13.1.3)]{EGAIV3}). Since the fibers of $f$ are group schemes (because $f$ is such), the local dimension stays constant on each fiber which implies that
$\dim_x f^{-1}(s)=\dim f^{-1}(s)$ for any $s\in \Spec R_X$ and  any  $x\in f^{-1}(s)$. Moreover, since the dimension of a scheme locally of finite type over a field is invariant under field extensions (see \cite[(4.1.4)]{EGAIV2}),
we also have that $\dim f^{-1}(s)=\dim \Pic(\ov{J}_{\X_{\ov s}}(\un q))$ for any $s\in \Spec R_X$. Putting everything together  we deduce that
\begin{equation}\label{E:upper1}
\dim \Pic(\ov{J}_{\X_{\ov s}}(\un q)) \geq p_a(X) \text{ for any } s\in S.
\end{equation}
On the other hand, if $h^1(\ov{J}_{X}(\un q),\O_{\ov{J}_{X}(\un q)})=p_a(X)$ then by upper semicontinuity of the cohomology groups, we get that
\begin{equation}\label{E:upper2}
h^1(\ov{J}_{\X_{\ov s}}(\un q),\O_{\ov{J}_{\X_{\ov s}}(\un q)})\leq p_a(X) \text{ for any } s\in S.
\end{equation}
Combining \eqref{E:upper1} and \eqref{E:upper2}, we infer that $\dim \Pic(\ov{J}_{\X_{\ov s}}(\un q))\geq h^1(\ov{J}_{\X_{\ov s}}(\un q),\O_{\ov{J}_{\X_{\ov s}}(\un q)})$ for any $s\in S$.
This implies that, for every $s\in S$, $\Pic(\ov{J}_{\X_{\ov s}}(\un q))$ is smooth of dimension equal to $p_a(X)$ (see \cite[Cor. 9.5.13]{FGA}) which proves the claim.

\vspace{0,1cm}

Let us now prove part \eqref{T:Pic-univ4}. Since the geometric fibers of $f$ are smooth of the same dimension by Claim 2, $\Pic^o_u$ is represented by an open subgroup scheme
$\Pic^o(\ov{J}_{\X}(\un q))\subseteq \Pic^{\tau}(\ov{J}_{\X}(\un q))$,  smooth and of finite type over $\Spec R_X$, by \cite[Prop. 9.5.20]{FGA}. Moreover, $\Pic^o(\ov{J}_{\X}(\un q))$ is separated over $\Spec R_X$ because $\Pic^{\tau}(\ov{J}_{\X}(\un q))$ is separated over $\Spec R_X$ by \eqref{T:Pic-univ2}. Therefore it remains to prove that $f^{\tau}:\Pic^{\tau}(\ov{J}_{\X}(\un q))\to \Spec R_X$ is smooth.

For any $n\in \bbN$, denote by $\phi_n: \Pic(\ov{J}_{\X}(\un q))\to \Pic(\ov{J}_{\X}(\un q))$ the group scheme homomorphism sending an element to its $n$-th power.
Following \cite[\S 1]{FGAVI}, consider the following open subgroup schemes of $\Pic^{\tau}(\ov{J}_{\X}(\un q))$:
\begin{equation}\label{E:tantiPic}
\begin{aligned}
f^{\sigma}: \Pic^{\sigma}(\ov{J}_{\X}(\un q)):=\bigcup_{(n,p)=1}\phi_n^{-1}\left(\Pic^o(\ov{J}_{\X}(\un q)) \right)\to \Spec R_X, \\
f^{\rho}: \Pic^{\rho}(\ov{J}_{\X}(\un q)):=\bigcup_{n=p^r}\phi_n^{-1}\left(\Pic^o(\ov{J}_{\X}(\un q)) \right)\to \Spec R_X, \\
\end{aligned}
\end{equation}
where $p$ denotes the characteristic of the base field $k$ and $(n,p)$ denotes the greatest common divisor of $n$ and $p$.
Clearly, the multiplication map induces a surjective homomorphism of $\Spec R_X$-group schemes (see also \cite[\S 1]{FGAVI})
\begin{equation}\label{E:molti}
m: \Pic^{\sigma}(\ov{J}_{\X}(\un q))\times_{\Spec R_X} \Pic^{\rho}(\ov{J}_{\X}(\un q)) \twoheadrightarrow
\Pic^{\tau}(\ov{J}_{\X}(\un q)).
\end{equation}
According to \cite[Thm. 2.5]{FGAVI}, the $n$-th power morphism $\phi_n$ is \'etale, hence in particular universally open, if $(n,p)=1$. This implies that:
\begin{enumerate}[(a)]
\item \label{E:a} $f^{\sigma}: \Pic^{\sigma}(\ov{J}_{\X}(\un q))\to \Spec R_X$ is smooth (hence universally open), using that $\Pic^o(\ov{J}_{\X}(\un q))\to \Spec R_X$ is smooth and \cite[Prop. 2.10(i)]{FGAVI};
\item \label{E:b} $f^{\rho}: \Pic^{\rho}(\ov{J}_{\X}(\un q))\to \Spec R_X$ is universally open by \cite[Thm. 1.1(iv)]{FGAVI}.
\end{enumerate}
Therefore, $\Pic^{\sigma}(\ov{J}_{\X}(\un q))\times_{\Spec R_X} \Pic^{\rho}(\ov{J}_{\X}(\un q))\to \Spec R_X$ is universally open, since the property of being universally open is stable by base change  and composition (see \cite[(14.3.4)]{EGAIV3}). This indeed implies that $f^{\tau}:\Pic^{\tau}(\ov{J}_{\X}(\un q))\to \Spec R_X$ is universally open, using that the  multiplication map $m$ of \eqref{E:molti} is surjective and \cite[(14.3.4)(i)]{EGAIV3}.

%Since $f$ is locally of finite type by part \eqref{T:Pic-univ1}, equidimensional and the codomain $\Spec R_X$ of $f$ is geometrically unibranch (being regular), then $f$ is universally open by Chevalley's criterion (\cite[(14.4.3)]{EGAIV3}).
Now, since $f^{\tau}$ is universally open and  of finite type, the fibers of $f^{\tau}$ are geometrically reduced (being smooth)
and the codomain $\Spec R_X$ of $f^{\tau}$ is locally Noetherian and reduced (being Noetherian and regular), then we conclude that $f^{\tau}$ is flat by \cite[(15.2.3)]{EGAIV3}.
Finally, since $f^{\tau}$ is flat and of finite presentation (being of finite type over a Noetherian codomain) and the geometric fibers of $f^{\tau}$ are smooth, we conclude that $f^{\tau}$ is smooth by
\cite[Sec. 2.4, Prop. 8]{BLR}, q.e.d.

\end{proof}

\section{The Poincar\'e bundle}\label{S:Poincare}

The aim of this section is to introduce the Poincar\'e line bundle for fine compactified Jacobians and to study its properties. Throughout this section, we fix a reduced connected curve $X$ (not necessarily with locally planar singularities).

With this in mind, consider the triple product $X\times \bJbar_X\times \bJ_X$ and, for any $1\leq i<j\leq 3$, denote by $p_{ij}$ the projection onto the product of the $i$-th and $j$-th factors. Choose a universal sheaf $\I$ on $X\times \bJbar_X$
as in Fact \ref{F:huge-Jac}\eqref{F:huge3} and denote by $\calI^0$ its restriction to
$X\times \bJ_X\subseteq X\times \bJbar_X$.
Consider the trivial family of curves
$$p_{23}:X\times\bJbar_X\times \bJ_X\to \bJbar_X\times \bJ_X$$
and form the line bundle on $\bJbar_X\times \bJ_X$, called the \emph{Poincar\'e bundle}:
\begin{equation}\label{E:Poin-sheaf}
\P:={\mathcal D}_{p_{23}}(p_{12}^*\I\otimes p_{13}^*\calI^0)^{-1}\otimes \mathcal D_{p_{23}}(p_{13}^*\calI^0)
\otimes {\mathcal D}_{p_{23}}(p_{12}^*\mathcal I)
\end{equation}
where ${\mathcal D}_{p_{23}}$ denotes the determinant of cohomology with respect to the morphism $p_{23}$.
For the basic properties of the determinant of cohomology, we refer to \cite{knudsen2} (see also
\cite[Sec. 6.1]{est1} for a summary).

In the sequel we will be often interested in the line bundles
$$\P_M:=\P_{|\bJbar_X\times \{M\}}\in {\mathcal Pic}(\bJbar_X),$$
where $M\in \bJ_X$ is a line bundle on $X$. Although the Poincar\'e line bundle \eqref{E:Poin-sheaf}  depends on the chosen universal sheaf $\I$, the restriction $\P_M$ does not if
$M$ has degree $0$.

\begin{lemma}\label{L:PM}
If $M\in \bJ_X^{1-p_a(X)}$, i.e. if $\chi(M)=1-p_a(X)$ (or, equivalently, if $M$ has degree $0$), then the line bundle $\P_M\in  {\mathcal Pic}(\bJbar_X)$ is  given by
\begin{equation}\label{E:PM}
\P_M=\D_{\pi_2}(\I\otimes \pi_1^*M)^{-1}\otimes \D_{\pi_2}(\pi_1^*M)\otimes \D_{\pi_2}(\I),
\end{equation}
where as usual $\pi_i$ denotes the projection of $X\times \bJbar_X$ onto the $i$-th factor (for $i=1,2$)
and $\I$ is any universal sheaf on $X\times \bJbar_X$ as in Fact \ref{F:huge-Jac}\eqref{F:huge3}. In particular, $\P_M$ is independent of the chosen universal sheaf $\I$.
\end{lemma}
\begin{proof}
Formula \eqref{E:PM} (for any $M\in \bJ_X$) with respect to the universal sheaf $\I$ used in \eqref{E:Poin-sheaf}   follows from the fact that the determinant of cohomology commutes with base change.

The fact that \eqref{E:PM} is independent from the chosen $\I$ if $\chi(M)=1-p_a(X)$ follows from the projection formula for the determinant of cohomology
using that any universal sheaf $\wt{\I}$ on $X\times \bJbar_X$ is related to $\I$ via
$\wt{\I}=\I\otimes \pi_2^*(N),$
for some $N\in \Pic(\bJbar_X)$, where $\pi_2:X\times \bJbar_X\to \bJbar_X$ denotes the projection onto the second factor. The computation is similar to the one in \cite[Proof of Prop. 2.2]{egk}
and left to the reader.
\end{proof}

For any general polarization $\un q$ on $X$, the restriction of the Poincar\'e bundle $\P$ to $\ov{J}_X(\un q)\times \bJ_X^{1-p_a(X)}$  defines, via the universal property of $\Pic(\ov{J}_X(\un q))$,
an algebraic morphism
\begin{equation}\label{D:mor-beta}
\begin{aligned}
\wt{\beta}_{\un q}: \bJ_X^{1-p_a(X)} & \longrightarrow \Pic(\ov{J}_X(\un q)),\\
M & \mapsto (\P_M)_{| \ov{J}_X(\un q)}.
\end{aligned}
\end{equation}
 Lemma \ref{L:PM}, together with the fact that $\bJ_X^{1-p_a(X)}$ is reduced by Fact \ref{F:huge-Jac}\eqref{F:huge1}, implies that the morphism $\beta_{\un q}$ is independent of the chosen Poincar\'e bundle $\P$.

Note that, from \eqref{E:PM}, it follows that
$\wt{\beta}_{\un q}(\O_X)=(\P_{\O_X})_{| \ov{J}_X(\un q)}=\O_{\ov{J}_X(\un q)}.$
%The morphism $\beta$ sends the trivial sheaf on $X$ into the trivial sheaf on $\bJbar_X$.
Therefore the morphism $\wt{\beta}_{\un q}$ restricts to a morphism
\begin{equation}\label{E:beta-pola}
\begin{aligned}
\beta_{\un q}: J(X)=\Pic^o(X) & \to \Pic^o(\ov{J}_X(\un q)), \\
M & \mapsto (\P_M)_{| \ov{J}_X(\un q)}.
\end{aligned}
\end{equation}

\begin{prop}\label{P:beta-hom}
For any general polarization $\un q$ on $X$,
the maps $\wt{\beta}_{\un q}$ and $\beta_{\un q}$ are homomorphisms of group schemes.
\end{prop}
\begin{proof}
Since we have already observed that the maps in question are algebraic morphisms, it remains to prove that
\begin{equation}\label{E:mult-Poin}
\P_{M_1\otimes M_2}\cong\P_{M_1}\otimes \P_{M_2}.
\end{equation}
for any $M_1, M_2\in \bJ_X^{1-p_a(X)}$.

In order to prove this, observe that we can write (for $i=1,2$) $M_i=\O_X(-\gamma_i+\delta_i)$, where $\gamma_i$ and $\delta_i$ are effective divisors contained in the smooth locus of $X$.
Moreover, we can clearly assume that $\delta_1$ and $\delta_2$ (resp. $\gamma_1$ and $\gamma_2$) have disjoint support.

%Moreover, the effective divisors $\delta_i$ and $\gamma_i$ can be written as
%$$ \delta_i:=\sum_{a^i_j>0}a^i_j Q^i_j \hspace{0,5cm} \text{ and } \hspace{0,5cm}
%-\gamma_i:=\sum_{a^i_j<0} a^i_j Q^i_j.$$
Consider the following two exact sequences that are derived from the two exact sequences defining $\O_{\gamma_i}$ and
$\O_{\delta_i}$:
\begin{equation}\label{E:2ex-seq}
\begin{aligned}
& 0 \to \O_X(-\gamma_i) \to \O_X \to \O_{\gamma_i}\to 0,\\
& 0 \to \O_X(-\gamma_i)\to  M_i \to  {M_i}_{|\delta_i}\cong\O_{\delta_i}\to 0.
\end{aligned}
\end{equation}
Pulling back \eqref{E:2ex-seq} via $\pi_1$ and using the additivity of the determinant of cohomology, we get
\begin{equation}\label{E:form1}
 \D_{\pi_2}(\pi_1^*M_i)\cong \D_{\pi_2}(\pi_1^*M_i)\otimes \D_{\pi_2}(\pi_1^*\O_X)^{-1}\cong
 \D_{\pi_2}(\pi_1^*\O_{\delta_i})\otimes \D_{\pi_2}(\pi_1^*\O_{\gamma_i})^{-1}.
\end{equation}
Similarly, by tensoring the pull-back via $\pi_1$ of two exact sequences \eqref{E:2ex-seq} with $\I$ and using the additivity of
determinant of cohomology, we get
\begin{equation}\label{E:form2}
\D_{\pi_2}(\I\otimes \pi_1^* M_i)^{-1}\otimes \D_{\pi_2}(\I)\cong
\D_{\pi_2}(\I\otimes \pi_1^*\O_{\delta_i})^{-1}\otimes \D_{\pi_2}(\I\otimes \pi_1^* \O_{\gamma_i}).
\end{equation}
Note that the above exact sequences make sense since $\I$ is locally free along $\pi_1^{-1}(\delta_i)$ and $\pi_1^{-1}(\gamma_i)$ (because
$\gamma_i$ and $\delta_i$ are  contained in the smooth locus of $X$), hence $\I\otimes \pi_1^*\O_{\delta_i}$ and  $\I\otimes \pi_1^* \O_{\gamma_i}$
are flat over $\bJbar_X$ and we can consider their determinant of cohomology with respect to $\pi_2$.
By plugging \eqref{E:form1} and \eqref{E:form2} into the \eqref{E:PM}, we get
\begin{equation}\label{E:for-PMi}
\P_{M_i}\cong\D_{\pi_2}(\I\otimes \pi_1^*\O_{\delta_i})^{-1}\otimes \D_{\pi_2}(\I\otimes \pi_1^* \O_{\gamma_i})\otimes
\D_{\pi_2}(\pi_1^*\O_{\delta_i})\otimes \D_{\pi_2}(\pi_1^*\O_{\gamma_i})^{-1}.
\end{equation}
Since $M_1\otimes M_2=\O_X(-\gamma_1-\gamma_2+\delta_1+\delta_2)$,  we get in a similar way that
\begin{equation}\label{E:for-PM12}
\P_{M_1\otimes M_2}\cong\D_{\pi_2}(\I\otimes \pi_1^*\O_{\delta_1\cup\delta_2})^{-1}\otimes \D_{\pi_2}(\I\otimes \pi_1^* \O_{\gamma_1\cup \gamma_2})\otimes
\D_{\pi_2}(\pi_1^*\O_{\delta_1\cup \delta_2})\otimes \D_{\pi_2}(\pi_1^*\O_{\gamma_1\cup \gamma_2})^{-1}.
\end{equation}
Since $\delta_1$ and $\delta_2$ (reps.  $\gamma_1$ and $ \gamma_2$) are zero-dimensional subschemes of $X$ with disjoint  support, for any
coherent sheaf $\F$ on $X\times \bJbar_X$ which is locally free along $\pi_1^{-1}(\delta_i)$ and $\pi_1^{-1}(\gamma_i)$, we have that
\begin{equation}\label{E:disj-supp}
\begin{aligned}
& \D_{\pi_2}(\F\otimes \pi_1^*\O_{\delta_1\cup \delta_2})=\D_{\pi_2}(\F\otimes \pi_1^*\O_{\delta_1})\otimes \D_{\pi_2}(\F\otimes \pi_1^*\O_{ \delta_2}), \\
& \D_{\pi_2}(\F\otimes \pi_1^*\O_{\gamma_1\cup \gamma_2})=\D_{\pi_2}(\F\otimes \pi_1^*\O_{\gamma_1})\otimes \D_{\pi_2}(\F\otimes \pi_1^*\O_{ \gamma_2}). \\
\end{aligned}
\end{equation}
Comparing formulas \eqref{E:for-PMi} and \eqref{E:for-PM12} and using \eqref{E:disj-supp}, we get  the required formula \eqref{E:mult-Poin}.
\end{proof}

\begin{remark}
With an argument similar to the one in the proof of the above Proposition \ref{P:beta-hom}, it is possible to get a description of  the line bundle $\P_M$ on $\bJbar_X$.
More precisely,  given $M\in  \bJ_X^{1-p_a(X)}$ with $M\cong \O_X(\sum a_iQ_i)$ for a divisor $\sum a_i Q_i$ supported on the smooth locus of $X$,
then we get that
%the fiber of $\P$ over $(I,M)$ is canonically isomorphic to
\begin{equation}\label{E:fiber-Poin}
\P_{M}\cong \bigotimes (\cI_{|Q_i\times \bJbar_X})^{-a_i}.
\end{equation}
\end{remark}

\vspace{0,2cm}

An important property of the line bundles $\P_M$ on $\bJbar_X$ is the fact that they are invariant under pull-back for the multiplication map by an element
$N\in \Pic(X)$:
\begin{equation}\label{E:multN}
\begin{aligned}
-\otimes N : \bJbar_X & \to \bJbar_X \\
I & \mapsto I\otimes N.
\end{aligned}
\end{equation}

\begin{lemma}\label{L:inv-mult}
For any $N\in {\mathcal Pic}(X)$ and $M\in \bJ_X^{1-p_a(X)}$, we have that
$$(-\otimes N)^*\P_M\cong\P_M.$$
\end{lemma}
\begin{proof}
Consider the following commutative diagram
\begin{equation}\label{poincmult}
\xymatrix{
X\times \bJbar_X  \ar[r]^{(\id,-\otimes N)}  \ar[d]_{\pi_2} &  X\times \bJbar_X \ar[d]^{\pi_2}\\
\bJbar_X \ar[r]_{-\otimes N}  & \bJbar_X
}
\end{equation}
By definition of the multiplication map $-\otimes N$, it follows that
\begin{equation}\label{E:pull-mult}
(\id, -\otimes N)^*\I\cong\wt \I\otimes \pi_1^*N,
\end{equation}
for some universal sheaf $\wt \I$ on $X\times \bJbar_X$, possibly different from $\I$.
Using \eqref{E:PM} and \eqref{E:pull-mult}, together with the fact that the determinant of cohomology
commutes with pull-back, we get
\begin{equation}\label{E:pull-PM}
(-\otimes N)^*\P_M\cong\D_{\pi_2}(\wt \I\otimes \pi_1^*N\otimes \pi_1^*M)^{-1}\otimes \D_{\pi_2}(\pi_1^*M)\otimes
\D_{\pi_2}(\wt \I\otimes  \pi_1^*N).
\end{equation}
By comparing \eqref{E:pull-PM} and  the formula \eqref{E:PM} which remains true if we substitute $\I$ with $\wt \I$ as observed before, we deduce that the statement of the Lemma is equivalent to
\begin{equation}\label{E:comp-lb}
\D_{\pi_2}(\wt \I\otimes \pi_1^*M)^{-1}\otimes \D_{\pi_2}(\wt \I)\cong
\D_{\pi_2}(\wt \I\otimes \pi_1^*N\otimes \pi_1^*M)^{-1}\otimes \D_{\pi_2}(\wt \I\otimes  \pi_1^*N).
\end{equation}
In order to prove \eqref{E:comp-lb}, we proceed similarly to the proof of Proposition \ref{P:beta-hom}.  
We write $M\cong\O_X(-\gamma+\delta)$ where $\gamma$ and $\delta$ are two effective divisors on $X_{\rm sm}$ and we apply formula \eqref{E:form2} (with $M_i$ replaced by $M$) to the sheaves  $\I$ and $\wt \I\otimes \pi_1^*N$ in order to get
\begin{equation}\label{E:isom1}
\D_{\pi_2}(\wt \I\otimes \pi_1^*M)^{-1}\otimes \D_{\pi_2}(\wt \I)\cong \D_{\pi_2}( \wt \I\otimes \pi_1^{*}\O_{\delta})^{-1}\otimes \D_{\pi_2}(\wt \I \otimes \pi_1^{*}\O_{\gamma}),
\end{equation}
\begin{equation}\label{E:isom2}
\D_{\pi_2}(\wt \I\otimes \pi_1^*N\otimes \pi_1^*M)^{-1}\otimes \D_{\pi_2}(\wt \I\otimes \pi_1^*N)\cong
\D_{\pi_2}( \wt \I\otimes \pi_1^*N \otimes \pi_1^{*}\O_{\delta})^{-1}\otimes \D_{\pi_2}(\wt \I \otimes \pi_1^*N \otimes \pi_1^{*}\O_{\gamma}).
\end{equation}
Comparing \eqref{E:isom1} and \eqref{E:isom2} and using that $\wt \I\otimes \pi_1^*N \otimes \pi_1^{*}\O_{\delta}=\wt \I\otimes \pi_1^*(N \otimes \O_{\delta})\cong \wt \I\otimes \pi_1^{*}\O_{\delta}$ and   similarly with $\delta$ replaced by $\gamma$, we get the isomorphism in \eqref{E:comp-lb}, q.e.d.

%By pulling back via $\pi_1$ the two exact sequences associated to the two effective divisors $\gamma,\delta\subset X$ and tensoring with a suitable line bundle on $X\times \bJbar_X$, we get the two exact sequences on $X\times \bJbar_X$:
%\begin{equation}\label{E:exseq-div}
%\begin{sis}
%& 0\rightarrow \wt \I \otimes \pi_1^*\mathcal{O}(-\gamma)\rightarrow \wt \I \rightarrow \wt \I_{|\pi_1^{-1}(\gamma)}\rightarrow 0,\\
%& 0\rightarrow \wt \I \otimes \pi_1^*\mathcal{O}(-\gamma)\rightarrow \wt \I \otimes \pi_1^*M\rightarrow \wt \I_{|\pi_1^{-1}(\delta)}\rightarrow 0,\\
%\end{sis}
%\end{equation}
%From \eqref{E:exseq-div} and the additivity of the determinant of cohomology, we get that
%\begin{equation}\label{E:isom1}
%\D_{\pi_2}(\wt \I\otimes \pi_1^*M)^{-1}\otimes \D_{\pi_2}(\wt \I)\cong \D_{\pi_2}( \wt \I_{|\pi_1^{-1}(\delta)})^{-1}\otimes \D_{\pi_2}(\wt \I_{|\pi_1^{-1}(\gamma)}).
%\end{equation}
%Similarly, by tensoring the two exact sequences \eqref{E:exseq-div} with $\pi_1^*(N)$ and using the additivity of the determinant of cohomology together with the fact that $(\wt \I\otimes \pi_1^*N)_{|\pi_1^{-1}(\delta)}\cong \wt \I_{|\pi_1^{-1}(\delta)}$ and similarly with $\delta$ replaced by $\gamma$,  we get that
%\begin{equation}\label{E:isom2}
%\D_{\pi_2}(\wt \I\otimes \pi_1^*M\otimes \pi_1^*N)^{-1}\otimes \D_{\pi_2}(\wt \I\otimes \pi_1^*N)\cong
% \D_{\pi_2}( \wt \I_{|\pi_1^{-1}(\delta)})^{-1}\otimes \D_{\pi_2}(\wt \I_{|\pi_1^{-1}(\gamma)}).
%\end{equation}
%By putting together \eqref{E:isom1} and \eqref{E:isom2}, we get the isomorphism in \eqref{E:comp-lb}, q.e.d.
\end{proof}

The Poincar\'e bundle behaves well with respect to the decomposition of a curve into its separating blocks in the sense of \S\ref{S:Abel}.

\begin{lemma}\label{L:PM-blocks}
Let $X$ be a connected reduced curve satisfying condition $(\dagger)$ and denote by $Y_1,\ldots, Y_r$ its separating blocks as in \S\ref{S:Abel}.
Let $M\in J(X)$ and set $M_j:=M_{|Y_j}\in J(Y_j)$. Denote by $\P_M\in \Pic^o(\bJbar_X)$ and $\P_{M_j}\in \Pic^o(\bJbar_{Y_j})$
the corresponding  fibers of the Poincar\'e bundles for the curves $X$ and $Y_j$, respectively.
Then the push-forward of $\P_M$ via the isomorphism  $\tau^*:\bJbar_X\xrightarrow[]{\cong} \prod_j \bJbar_{Y_j}$ of Theorem \ref{T:Abel}\eqref{T:Abel1}
is equal to
$$\P_{M_1}\boxtimes \ldots \boxtimes \P_{M_r}:=p_1^*(\P_{M_1})\otimes \ldots \otimes p_r^*(\P_{M_r}),$$
where $p_j:\prod_i \bJbar_{Y_i}\to \bJbar_{Y_j}$ is the projection onto the $j$-th factor.
\end{lemma}

\begin{proof}
Consider the following commutative diagram
\begin{equation}\label{E:diagr2}
\xymatrix{
X\times \bJbar_X   \ar[d]^{\pi_2}
& \wt{X} \times \prod_i \bJbar_{Y_i} \ar[l]_{\tau \times (\tau^*)^{-1}} \ar[d]^{\wt{\pi_2}}
& Y_j\times \prod_i \bJbar_{Y_i} \ar[d] \ar@{_{(}->}[l]_{\eta_j\times \id} \ar[r]^{\id \times p_j}
\ar@{}[dr]|{\square} \ar[d]^{\wt{\pi_2}^j} & Y_j\times \bJbar_{Y_j} \ar[d]^{\pi_2^j}\\
\bJbar_X   & \prod_i \bJbar_{Y_i} \ar[l]_{\cong}^{(\tau^*)^{-1}} \ar@{=}[r] & \prod_i \bJbar_{Y_i} \ar[r]_{p_j}&\bJbar_{Y_j}
}
\end{equation}
where $\tau: \wt{X}=\coprod_i Y_i \to X$ is the normalization of $X$ at the separating nodes
of $X$ and $\eta_j:Y_j\hookrightarrow \wt{X}=\coprod_i Y_i$ is the natural inclusion.
Denote by $\pi_1$, $\wt{\pi_1}$, $\wt{\pi_1}^j$, $\pi_1^j$ the projections onto the first factors of the products  appearing in the middle row of diagram \eqref{E:diagr2}. Choose a universal sheaf $\I$ on $X\times \bJbar_X$  as in Fact \ref{F:huge-Jac}\eqref{F:huge3} and set $\wt{\I}:=(\tau \times (\tau^*)^{-1})^*(\I)$.

Since a torsion-free rank-1 sheaf on $X$ is completely determined by its pull-back to $\wt{X}$
by Theorem \ref{T:Abel}\eqref{T:Abel1}, we have that
the pull-back of $\P_M$ to $\prod_i \bJbar_{Y_i}$ via the isomorphism $(\tau^*)^{-1}$ is equal to
\begin{equation*}\tag{*}
((\tau^*)^{-1})^*(\P_M)=((\tau^*)^{-1})^*\left(\D_{\pi_2}(\I\otimes \pi_1^*M)^{-1}\otimes \D_{\pi_2}(\pi_1^*M)\otimes \D_{\pi_2}(\I)\right)\cong
\end{equation*}
$$\cong\D_{\wt{\pi_2}}(\wt{\I}\otimes \wt{\pi_1}^*\wt M)^{-1}\otimes \D_{\wt{\pi_2}}(\wt{\pi_1}^*\wt M)\otimes \D_{\wt{\pi_2}}(\wt{\I}),
$$
where $\wt M:=\tau^*M\in J(\wt X)$.
Since $\wt{X}$ is the disjoint union of the subcurves $Y_i$ and $(\eta_j\times \id)^*(\wt{\I})=
(\id \times p_j)^*(\I_j):=\wt{\I_j}$ for some universal sheaf $\I_j$ on $Y_j\times \bJbar_{Y_j}$,
we have that
\begin{equation*}\tag{**}
\D_{\wt{\pi_2}}(\wt{\I}\otimes \wt{\pi_1}^*\wt M)^{-1}\otimes \D_{\wt{\pi_2}}(\wt{\pi_1}^*\wt M)\otimes \D_{\wt{\pi_2}}(\wt{\I})\cong
\end{equation*}
$$
\bigotimes_{j=1}^r \left[\D_{\wt{\pi_2}^j}(\wt{\I}_j\otimes (\wt{\pi_1}^j)^*M_j)^{-1}\otimes \D_{\wt{\pi_2}^j}((\wt{\pi_1}^j)^*M_ j)\otimes \D_{\wt{\pi_2}^j}(\wt{\I}_j)\right].
$$
Finally, since the square in the right of diagram \eqref{E:diagr2} is cartesian, applying the base change
properties of the determinant of cohomology, we get
\begin{equation*}\tag{***}
p_j^*(\P_{M_j})=p_j^*\left(\D_{{\pi_2}^j}({\I}_j\otimes ({\pi_1}^j)^*M_j)^{-1}\otimes
\D_{{\pi_2}^j}(({\pi_1}^j)^*M_ j)\otimes \D_{{\pi_2}^j}({\I}_j)\right)\cong
\end{equation*}
$$\cong
\D_{\wt{\pi_2}^j}(\wt{\I}_j\otimes (\wt{\pi_1}^j)^*M_j)^{-1}\otimes \D_{\wt{\pi_2}^j}((\wt{\pi_1}^j)^*M_ j)\otimes \D_{\wt{\pi_2}^j}(\wt{\I}_j).
$$
By combining (*), (**) and (***) we get the equality
$$((\tau^*)^{-1})^*(\P_M)\cong\bigotimes_{j=1}^r p_j^*(\P_{M_j}),$$
which concludes the proof.
\end{proof}

Given any Abel map $A_L$ and choosing a fine compactified Jacobian $\ov{J}_X(\un q)$ such that $\Im A_L\subseteq \ov{J}_X(\un q)$ (which is always possible if $X$ is Gorenstein by  Theorem \ref{T:Abel}\eqref{T:Abel3}),
the morphism $\beta_{\un q}$ of \eqref{E:beta-pola} provides a right inverse for the morphism $A_L^*$ of \eqref{E:pullback-Abel}.
 %if $X$ does not have separating points.
This is originally due to Esteves-Gagn\'e-Kleiman in the case where $X$ is integral (see \cite[Prop. 2.2]{egk}).

\begin{prop}\label{P:prop-AL}
Let $X$ be a  connected reduced curve satisfying condition $(\dagger)$, as in \S\ref{N:sep-node}. Then, for every $L\in \Pic(X)$ and any general polarization $\un q$ such that $\Im A_L\subseteq \ov{J}_X(\un q)$,
we have that
$$A_L^*\circ \beta_{\un q} = {\rm id}_{J(X)}.$$
In other words, for every $M\in J(X)$ we have that $A_L^*((\P_M)_{|\ov{J}_X(\un q)})\cong M$.
\end{prop}
\begin{proof}
We will first prove the Proposition in the case where $X$ does not have separating points and then in the general case.

\un{Case I:} $X$ does not have separating points.

The proof in this case is an easy adaptation of \cite[Prop. 2.2]{egk} and it is therefore left to the reader. The crucial property that holds in this case (while failing in general) and that makes the proof of loc. cit. work is the fact that the Abel map $A_L:X\to \J_X(\un q)\subseteq \bJbar_X$ is defined by the sheaf  $\I_{\Delta}\otimes p_1^*L$ on $X\times X$, where $\Delta\subset X\times X$ is the diagonal, as it follows from Theorem \ref{T:Abel}\eqref{T:Abel2}.

\un{Case II}: $X$ satisfies condition $(\dagger)$.

Let $Y_i$ for $1\leq i\leq r$ be the separating blocks of $X$ as in \S\ref{S:Abel} and set $L_i:=L_{|Y_i}\in \Pic(Y_i)$.
According to Theorem \ref{T:Abel}\eqref{T:Abel1}, we can choose  general polarizations $\un q^i$ on $Y_i$, for $1\leq i\leq r$, such that $\tau^*$ induces an isomorphism between $\J_X(\un q)$ and $\prod_{i=1}^r \J_{Y_i}(\un q^i)$.
Since $\Im A_L\subseteq \J_X(\un q)$ by assumption, we have that
$\Im A_{L_i}\subseteq \ov{J}_{Y_i}(\un q^i)\subseteq \bJbar_{Y_i}$, for every $1\leq i\leq r$, by \cite[Prop. 6.7(ii)]{MRV1}.
We get the following  diagram
\begin{equation}\label{E:diagr1}
\xymatrix{
& \prod_i \Pic^o(\ov{J}_{Y_i}(\un q^i)) \ar@/_/[dl]_{\prod_i {A_{L_i}^*}}\ar[dr]^{\otimes_i p_i^*(-)} & \\
 \prod_i J(Y_i) \ar@/_/[ur]_{\prod_i \beta_{\un q_i}}& & \Pic^o(\prod_i \ov{J}_{Y_i}(\un q^i)) \ar[d]^{\wh{\tau^*}}_{\cong}  \\
J(X) \ar[u]^{\tau^*}_{\cong} \ar@/_/[rr]_{\beta_{\un q}} & & \Pic^o(\ov{J}_X(\un q)),\ar@/_/[ll]_{A_L^*}
}
\end{equation}
where $\wh{\tau^*}$ is the isomorphism induced on $\Pic^o$ by  $\tau^*:\J_X(\un q) \stackrel{\cong}{\longrightarrow} \prod_{i=1}^r\J_{Y_i}(\un q^i)$
and $\beta_{\un q_i}$ is the map \eqref{E:beta-pola} with respect to the general polarization $\un q^i$ on the curve $Y_i$. From the definition of the Abel map $A_L$ (see Theorem \ref{T:Abel}\eqref{T:Abel2}), it follows that
 the two maps from $\prod_i \Pic^o(\J_{Y_i}(\un q^i))$ to $\prod_i J(Y_i)$ that arise from diagram \eqref{E:diagr1} are equal.
Lemma \ref{L:PM-blocks} can be re-interpreted as saying that the two maps from $J(X)$ to $\Pic^o(\J_X(\un q))$
that arise from diagram \eqref{E:diagr1} are equal.

Since each $Y_i$ does not have separating points, then we have that
$A_{L_i}^*\circ \beta_{\un q_i}=\id_{J(Y_i)}$ by Case I. This implies that $A_{L}^*\circ \beta_{\un q}=\id_{J(X)}$ by an easy diagram chase in  \eqref{E:diagr1}.
\end{proof}

An immediate consequence of the above result is the following:

\begin{cor}\label{C:inj-beta}
Let $X$ be a  connected reduced curve satisfying condition $(\dagger)$ and let $\un q$ be a general polarization on $X$.
If $\ov{J}_X(\un q)$ admits an Abel map (in the sense of \S\ref{S:Abel}), then
the homomorphism $\beta_{\un q}:J(X)\to \Pic^o(\ov{J}_X(\un q))$ is injective.
\end{cor}

%\begin{remark}
%The conjecture  \ref{C:beta-inj?}Ê would follow, using the injectivity of $\beta$,  from the following
%statement: $$ (\P_M)_{|\ov{J}_X(\un q)}\Rightarrow \P_M=\O_{\bJbar_X}.$$
%\end{remark}

\section{Cohomology of restricted Poincar\'e bundles}\label{S:coho-Poin}

The aim of this section is to prove some results about the cohomology of restricted Poincar\'e bundles
$\P_M:=\P_{|\bJbar_X\times \{M\}}$, for $M\in J(X)$, to  the fine compactified Jacobians of a connected reduced curve $X$ (not necessarily with locally planar singularities).
%Throughout this section, we assume that $X$ is connected.

The first result is a generalization \cite[Prop. 1]{arin1}, which deals with $X$ integral.

\begin{prop}\label{P:vanishcohom}
Let $M\in J(X)$ and let $\un q$ be a general polarization on $X$.
If  there exists $i\in{\mathbb N}$ such that $H^i(\ov{J}_X(\un q),\P_M)\neq 0$, then ${\P_M}_{|\bJ_X}\cong \O_{\bJ_X}$.
%\begin{enumerate}[(i)]
%\item \label{P:vanishcohom1} If $H^i(\bJbar_X,\P_M)\neq 0$ for some $i$ then ${\P_M}_{|\bJ_X}\cong \O_{\bJ_X}$.
%\item \label{P:vanishcohom2} If $H^i(\ov{J}_X(\un q),\P_M)\neq 0$ for some $i$ then
%${\P_M}_{|\bJ_X}\cong \O_{\bJ_X}$.
%\end{enumerate}
\end{prop}
\begin{proof}
The proof is an adaptation of the proof of \cite[Prop. 1]{arin1}. However, for the benefit of the reader, we chose to give some more details
than in loc. cit.
%Set $\ov{J}$ be equal to either $\bJbar_X$ or $\ov{J}_X(\un q)$.

Note that the generalized Jacobian $J(X)$ acts on $\ov{J}_X(\un q)$. Denote by $T\to J(X)$ the $\Gm$-torsor corresponding to the line bundle $(\P_M)_{|J(X)}$. Let us first prove two claims.

\un{Claim 1}: $T$ has the structure of a commutative algebraic group that is an extension of $J(X)$ by $\Gm$, i.e. there is a sequence of commutative algebraic groups
\begin{equation}\label{E:seq-T}
0\to \Gm\to T\to J(X)\to 0.
\end{equation}

Let $p\in X$ and let $\mathcal{I}^{\un 0}$ be a universal sheaf on $X \times J(X)$ such that
its restriction at $p\times J(X)$ is trivial. Let $m: J(X) \times J(X)\rightarrow  J(X)$ be
the multiplication map and $\pi_{i,j}: X \times J(X) \times J(X)\rightarrow X \times J(X) $ be the projection maps.
By the see-saw principle, the line bundles
$\pi_{1,2}^{*}\mathcal{I}^{\un 0}\otimes \pi_{1,3}^{*}\mathcal{I}^{\un 0}$ and
$({\id}_{X}\times m)^{*}\mathcal{I}^{\un 0}$ on $X \times J(X) \times J(X)$ are isomorphic.
Let $\sigma$ be a nowhere vanishing section of $(\mathcal{I}^{\un 0})_{|p\times J(X)}$. The section $\sigma$  induces nowhere vanishing sections $\overline{\sigma}$ of
$(\pi_{1,2}^{*}\mathcal{I}^{\un 0}\otimes \pi_{1,3}^{*}\mathcal{I}^{\un 0})_{|p\times J(X)\times J(X)}$
and $\hat{\sigma}$ of $(({\id}_{X}\times m)^{*}\mathcal{I}^{\un 0})_{|p\times J(X)\times J(X)}$.
Let $\phi:\pi_{1,2}^{*}\mathcal{I}^{\un 0}\otimes \pi_{1,3}^{*}\mathcal{I}^{\un 0}\rightarrow({\id}_{X}\times m)^{*}\mathcal{I}^{\un 0}$ on $X \times J(X) \times J(X)$
be an isomorphism sending $\overline{\sigma}$ to $\hat{\sigma}$.
A straightforward computation shows that $\phi$ makes the complement of the zero section in $\mathcal{I}^{\un 0}$ into
a group scheme over $X$. As a consequence, for any $s\in X$ the isomorphism $\phi$ induces
a group structure on the complement $T_{s}$ of the zero section in $\mathcal{I}^{\un 0}_{|s\times J(X)}$.

Let $p_{i}$ be smooth points of $X$ such that $M=\mathcal{O}(\sum a_i p_i)$.
By equation  \eqref{E:fiber-Poin}, we get an isomorphism
$${\P_M}_{|J(X)}
%=\D_{\pi_2}(\I^{\un 0}\otimes \pi_1^*\O_{\delta})^{-1}\otimes \D_{\pi_2}(\I^{\un 0}\otimes \pi_1^* \O_{\gamma})\otimes
%\D_{\pi_2}(\pi_1^*\O_{\delta})\otimes \D_{\pi_2}(\pi_1^*\O_{\gamma})^{-1}=
\cong \bigotimes_i (\mathcal{I}^{\un 0}_{|p_i\times J(X)})^{-a_{i}}.$$
Hence $T$ is the complement of the zero section in $\bigotimes_i (\mathcal{I}^{\un 0}_{|p_i\times J(X)})^{-a_{i}}$ and it carries a group structure which is induced by the group
structures on the $T_{p_{i}}$.
This group structure makes $T$ an abelian group  and its natural group morphism onto
$J(X)$ produces the exact sequence (\ref{E:seq-T}), q.e.d.

\un{Claim 2}: The action of $J(X)$ on $\ov{J}_X(\un q)$ lifts to an action of $T$ on $(\P_M)_{|\ov{J}_X(\un q)}$.
Moreover, $\Gm\subset T$ acts  on $(\P_M)_{|\ov{J}_X(\un q)}$ fiberwise in the standard way by multiplication.

As for the previous claim, let $p\in X$ and let $\mathcal{I}^{\un q}$ be a universal sheaf on $X \times \ov{J}_X(\un q)$ such that
its restriction at $p\times \ov{J}_X(\un q)$ is trivial.
Denote by $p_{1,2}: X\times J(X)\times \ov{J}_X(\un q)\rightarrow  X\times J(X)$ and
by $p_{1,3}: X\times J(X)\times \ov{J}_X(\un q)\rightarrow  X\times \ov{J}_X(\un q)$ the projection maps and let
$a:J(X)\times \ov{J}_X(\un q)\rightarrow \ov{J}_X(\un q)$ be the action of $J(X)$ on $\ov{J}_X(\un q)$.
In this case the see-saw principle gives an isomorphism $\psi: p_{1,2}^{*}\mathcal{I}^{\un 0}\otimes p_{1,3}^{*}\mathcal{I}^{\un q}
\rightarrow
({\id}_{X}\times a)^{*}\mathcal{I}^{\un q}$. Moreover a suitable choice of $\psi$ (analogous to the choice of $\phi$ in the previous claim)
gives an action over $X$ of  the complement of the zero section in $\mathcal{I}^{\un 0}$ on $\mathcal{I}^{\un q}$. Hence,
for every $s\in X$, the isomorphism $\psi$ induces an action of
$T_{s}$ on   $\mathcal{I}^{\un q}_{|s\times \ov{J}_ X(\un q)}$.
Since equation \eqref{E:fiber-Poin} gives the equality $${\P_M}_{|\ov{J}_X(\un q)}\cong
\bigotimes_i (\mathcal{I}^{\un q}_{|p_i\times \ov{J}_X(\un q)})^{-a_{i}},$$
we finally get that $T$ acts on $(\P_M)_{|\ov{J}_X(\un q)}$ lifting the action of $J(X)$ on $\ov{J}_X(\un q)$. The second part of the claim follows from our description of the action, q.e.d.

We can now finish the proof of the Proposition.
According to Claim 2, the algebraic group $T$ acts on any cohomology group $H^i(\ov{J}_X(\un q), (\P_M)_{|\ov{J}_X(\un q)})$.
Suppose that $H^i(\ov{J}_X(\un q),$ $ (\P_M)_{|\ov{J}_X(\un q)})\neq 0$ for some index $i$. Consider a $T$-irreducible non-trivial
submodule $0\neq V\subseteq H^i(\ov{J}_X(\un q),(\P_M)_{|\ov{J}_X(\un q)})$. Since $T$ is commutative, $V$ is a one-dimensional representation of $T$.
Therefore the action of $T$ on $V$ is given by a character $\chi:T\to \Gm$ and,
since $\Gm\subset T$ acts  on $(\P_M)_{|\ov{J}_X(\un q)}$ fiberwise in the standard way by multiplication, it follows that $\chi_{|\Gm}={\id}$. As a consequence, the character $\chi$ gives a splitting of the exact sequence \eqref{E:seq-T}, from which we deduce that $T\cong J(X)\times \Gm$. This is indeed equivalent to the fact that
\begin{equation}\label{E:trivP}
(\P_M)_{|J(X)}\cong \O_{J(X)}.
\end{equation}

We conclude now by using the Lemma \ref{L:inv-mult}. Indeed, if $\{X_1, \ldots, X_{\gamma}\}$ are the irreducible components of $X$, then we have the decomposition
\begin{equation}\label{E:dec-con}
\bJ_X=\Pic(X)=\coprod_{\un d\in \Z^{\gamma}} \Pic^{\un d}(X)
\end{equation}
into connected components, where  $\Pic^{\un d}(X)$ is the connected component of $\Pic(X)$ parametrizing those line bundles $L$ on $X$ having multidegree $\un \deg(L)$ equal to $\un d=(\un d_1,\ldots,\un d_{\gamma})$, i.e. $\deg(L_{|X_i})=\un d_i$ for any $1\leq i\leq \gamma$.
 From the above decomposition \eqref{E:dec-con}, it is enough to show that $(\P_M)_{|\Pic^{\un d}(X)}\cong  \O_{\Pic^{\un d}(X)}$ for each multidegree $\un d$.
Fix such  multidegree $\un d$ and take a line bundle $N$ of multidegree $\un d$. The multiplication by $N^{-1}$ induces an isomorphism
$-\otimes N^{-1}: \Pic^{\un d}(X)\stackrel{\cong}{\to} \Pic^{\un 0}(X)=J(X).$
Using Lemma \ref{L:inv-mult} and \eqref{E:trivP}, we now get
$$(\P_M)_{|\Pic^{\un d}(X)}\cong  (-\otimes N^{-1})^*(\P_M)_{|\Pic^{\un d}(X)}\cong  (-\otimes N^{-1})^*((\P_M)_{|J(X)})\cong  (-\otimes N^{-1})^*(\O_{J(X)})\cong \O_{\Pic^{\un d}(X)}.$$
\end{proof}

The previous Proposition implies the following two Corollaries, that generalize \cite[Cor. 2]{arin1} and \cite[Cor. 3]{arin1} to the case where $X$ is not integral.

\begin{cor}\label{C:smoothcohom}
Assume that $X$ is Gorenstein.  Let $M\in J(X)$ and
let $\un q$ be a general polarization on $X$. If  $H^i(\ov{J}_X(\un q),\P_M)\neq 0$ for some $i$ then $M_{|\Xsm} \cong \O_{\Xsm}$.

%\begin{enumerate}[(i)]
%\item \label{C:smoothcohom1} If $H^i(\bJbar_X,\P_M)\neq 0$ for some $i$ then $M_{|\Xsm} \cong \O_{\Xsm}$.
%\item \label{C:smoothcohom2} Assume Conjecture \ref{C:Abel}\eqref{C:Abel1}. If $H^i(\ov{J}_X(\un q),\P_M)\neq 0$ for
%some $i$ then $M_{|\Xsm}\cong \O_{\Xsm}$.
%\end{enumerate}
\end{cor}

\begin{proof}
Consider the Abel map $A_L:X\to \bJbar_X$ for some $L\in \Pic(X)$ and choose a fine compactified Jacobian $\ov{J}_X(\un q')$ such that $\Im A_L\subseteq \ov{J}_X(\un q')$ (which is always possible if $X$ is Gorenstein by  Theorem \ref{T:Abel}\eqref{T:Abel3}). Clearly we have that $A_L(X_{\rm sm})\subseteq J_X(\un q')\subseteq  \bJ_X$.  Using Proposition \ref{P:vanishcohom} and Proposition \ref{P:prop-AL} (applied to the fine compactified Jacobian $\ov J_X(\un q')$), we get
$$M_{|\Xsm}\cong A_L^*((\P_M)_{|\ov J_X(\un q')})_{|\Xsm}
\cong({A_{L}}|_{X_{sm}})^*((\P_M)_{|J_X(\un q')})
\cong({A_{L}}|_{X_{sm}})^*(\O_{J_X(\un q')})\cong\O_{\Xsm}.$$
\end{proof}

For any general polarization $\un q$ on $X$, consider the locus
\begin{equation*}
\begin{aligned}
%& \N:=\{M \in J(X) : H^i(\bJbar_X,\P_M)\neq 0 \mbox{ for some }i\}\subseteq J(X), \\
& \N(\un q):=\{M \in J(X) : H^i(\ov{J}_X(\un q) ,\P_M)\neq 0 \mbox{ for some }i\}\subseteq J(X).\\
\end{aligned}
\end{equation*}
Notice that, by semicontinuity, $\N(\un q)$ is a closed subset of $J(X)$ and that
$$
\begin{aligned}
%& \N=\supp (R{p_2}_*\P), \\
& \N(\un q)=\supp (R{p_2}_*(\P_{|J(X)\times \ov{J}_X(\un q)})), \\
\end{aligned}
$$
where $p_2:\ov{J}_X(\un q)\times J(X)\to J(X)$ is the second projection.

\begin{cor}\label{codim}
Assume that $X$ satisfies condition $(\dagger)$  and let $\un q$ be a general polarization on $X$. Then
$$\dim \N(\un q)\leq p_a(X)-g^{\nu}(X).$$
%\begin{enumerate}[(i)]
%\item $\dim \N \leq p_a(X)-g^{\nu}(X).$
%\item $\dim \N(\un q)\leq p_a(X)-g^{\nu}(X),$
%for any general polarization $\un q$ on $X$.
%\end{enumerate}
 \end{cor}
\begin{proof}
Observe that it is enough to prove the Corollary after a base change  to an uncountable algebraically closed field; therefore, with a slight abuse of notation, we can assume that our algebraically closed base field $k$ is uncountable.

The normalization morphism $\nu:X^{\nu}\to X$ induces by pull-back a smooth and surjective morphism $\nu^*:J(X)\to J(X^{\nu})$ with fibers of dimension equal to $p_a(X)-g^{\nu}(X)$.
Denote by $\wt{N}\subseteq J(X^{\nu})$ the locus of line bundles on $X^{\nu}$ that are trivial on
$\nu^{-1}(X_{\rm sm})\subseteq X^{\nu}$.

\un{Claim:} $\wt{N}$ is a countable set.

Indeed, set $F:=X^{\nu}\setminus \nu^{-1}(\Xsm)$. We have an exact sequence
$$\Z^{F} \stackrel{\alpha}{\to} \Pic(X^{\nu}) \to \Pic(\nu^{-1}(\Xsm)),$$
where the last map is the restriction map and $\alpha$ sends
$\{m_P\}_{P\in F} \in \Z^{F}$ into $\O_X(\sum_{P\in F} m_P\cdot P)$. The claim follows since
$\wt{N}$ is equal to $\Im(\alpha)\cap J(X^{\nu})$, q.e.d.

Now, Corollary \ref{C:smoothcohom} implies that  the subset $\nu^*(\N(\un q))\subseteq J(X^{\nu})$
(which is constructible by Chevalley's theorem, see \cite[(1.8.4)]{EGAIV1}) is contained in the countable subset $\wt{N}\subset J(X^{\nu})$. Since $k$ is uncountable by assumption, this can only happen if $\nu^*(\N(\un q))$ is a finite union of points. Therefore, the dimension of $\N(\un q)$ can be at most equal to the dimension of the fibers of $\nu^*$, i.e. to $p_a(X)-g^{\nu}(X)$, q.e.d.

\end{proof}

Proposition \ref{P:vanishcohom} can be strengthened for $i=0$ if the curve $X$ has  locally planar singularities.

% for fine compactified Jacobians that admit an Abel map in the sense of Definition \ref{D:ex-Abel-sep}.
%if one has the non vanishing of the space of global sections of $\P_M$.

\begin{prop}\label{P:nonvaH0}
Assume that $X$ has locally planar singularities and let $\un q$ be a general polarization on $X$.
%Assume that $\ov{J}_X(\un q)$ admits an Abel map.
If $M\in J(X)$ is such that  $H^0(\ov{J}_X(\un q),\P_M)\neq 0$,  then ${\P_M}_{|\ov{J}_X(\un q)}\cong  \O_{\ov{J}_X(\un q)}$.
\end{prop}
\begin{proof}
We know that $\ov{J}_X(\un q)$ is a connected reduced projective scheme over $k$ by Fact \ref{F:Este-Jac} and Theorem \ref{T:compJac}.
As already observed in  \eqref{E:beta-pola}, we have that
${\P_M}_{|\ov{J}_X(\un q)}\in \Pic^o(\ov{J}_X(\un q))$.
We can apply the (quite standard) Lemma \ref{L:sec-Pico}  below in order to conclude that
${\P_M}_{|\ov{J}_X(\un q)}\cong  \O_{\ov{J}_X(\un q)}$.
%From Corollary \ref{C:inj-beta}, it follows now that $M\cong \O_X$.
\end{proof}

%The following Lemma, which was used in the proof of the above Proposition, is quite standard; however, we include a proof since we were not able to find a suitable reference.

\begin{lemma}\label{L:sec-Pico}
Let $V$ be a connected reduced projective scheme over an algebraically closed field $k$.
Let $\L$ be a line bundle belonging to $ \Pic^o(V)$, i.e. the connected component of $\Pic(V)$ containing the identity.
If $H^0(V, \L)\neq 0$ then $\L\cong  \O_V$.
\end{lemma}
\begin{proof}
Assume first that $V$ is irreducible. Then any non-zero section $s$  of $\L$ induces a generically injective map $\wt{s}:\O_V\to \L$ which is therefore injective since $\O_V$ does not contain torsion sheaves. Moreover, since $\L$ and $\O_V$ have the same Hilbert polynomial with respect to any ample line bundle on $V$ (being algebraically equivalent), $\wt{s}$ has trivial cokernel, hence it is an isomorphism.

%Let $s$ be a non-zero section of $\L$ and consider it as a map $\wt{s}:\O_V\to \L$. Since $V$ is irreducible and $s\not \equiv 0$, the map $\wt{s}$ is generically  injective. This implies that the kernel of $\wt{s}$ is a torsion sheaf of $\O_V$, hence it is zero since $\O_V$ does not contain torsion sheaves. In other words, $\wt{s}$ realizes $\O_V$ as a subsheaf of $\L$. However, since $\L\in \Pic^o(V)$ by assumption, $\L$ and $\O_V$ are algebraically equivalent and therefore they have the same Hilbert polynomial with respect to any ample line bundle on $V$. This can only happen if  $\wt{s}$ is an isomorphism, q.e.d.

 In the general case, let $V_1, \ldots, V_r$ be the irreducible components of $V$.  Take a non-zero section $s\in H^0(V, L)$ and consider its zero locus $Z(s)\subsetneq V$.
%$$Z(s):=\{P\in V: s(P)=0\}\subsetneq V.$$
For each irreducible component $V_i$, there are two possibilities: either $s_i:=s_{|V_i}\equiv 0$ in which case $V_i\subseteq Z(s)$, or $s_i\not\equiv 0$ in which case $L_{|V_i}\cong \O_{V_i}$ by what was proved above. In the second case,  $s_i\in H^0(V_i,L_{|V_i})=H^0(V_i, \O_{V_i})$ is given by a constant non-zero section (because $V_i$ is projective and integral), which implies that  $V_i\cap Z(s)=\emptyset$.
Since $V$ is connected and $Z(s)\neq V$, we deduce that $Z(s)=\emptyset$. In other words, $s$ is a nowhere vanishing section of $\L$, hence it defines an isomorphism $\L\cong   \O_V$.

\end{proof}

\section{Proof of Theorem C for nodal curves}\label{S:nodalC}

The aim of this section is to prove Theorem C from the introduction for nodal curves. The key fact about fine compactified Jacobians of nodal curves that we are going to
use is the following result.
% of Alexeev-Nakamura \cite{AN} and Alexeev \cite{Ale2}.

\begin{prop}\label{P:H1-nodal}
Let $X$ be a nodal curve and let $\un q$ be a general polarization on $X$. Then  we have that
\begin{equation}\label{E:Hi}
h^i(\ov{J}_X(\un q),\O_{\ov{J}_X(\un q)})=\binom{p_a(X)}{i} \: \text{ for any } \: 0\leq i \leq p_a(X).
\end{equation}
In particular, it holds that
\begin{equation}\label{E:H1-g}
h^1(\ov{J}_X(\un q),\O_{\ov{J}_X(\un q)})=p_a(X).
\end{equation}
\end{prop}
\begin{proof}
%If ${\rm char}(k)=0$ then there is a very simple proof of the statement. Indeed, consider the universal fine compactified Jacobian $u:\J_{\X}(\un q)\to \Spec R_X$ of Theorem \ref{T:univ-fine}. Since $X$ is nodal by assumption, the fine compactified Jacobian $\J_X(\un q)$ has semi-log-canonical singularities by \cite[Thm. B(i)]{CMKV}. The same is true for the geometric fibers $\J_{\X_{\ov s}}(\un q^s)$ of $u$, since $\X_{\ov s}$ is nodal for every $s\in \Spec R_X$.  It follows from \cite[Cor. 1.2]{KK} that $h^i(\J_{\X_{\ov s}}(\un q^s), \O_{\J_{\X_{\ov s}}(\un q^s)})$ stays constant (for any fixed $i$) as $s$ varies in $\Spec R_X$.  In particular,  $h^i(\ov{J}_X(\un q),\O_{\ov{J}_X(\un q)})$ is equal to the $i$-th cohomology group of the structure sheaf of the Jacobian of the geometric generic fiber $\X_{\ov s}$, which is equal to $\binom{p_a(X)}{i}$ since $\X_{\ov s}$ is a smooth curve of genus $p_a(X)$.

%In arbitrary characteristics,
We will adapt  the proof of \cite[Thm. 4.3]{AN}, where the analogous  result is proved for
stable quasiabelian varieties, i.e. special fibers of certain one parameter degenerations of abelian varieties constructed from Delaunay decompositions. However, only the canonical polarized compactified Jacobian of degree $g-1$ (see \cite[Sec. 3]{Ale2}, \cite{CV}) is a stable quasiabelian variety and this special compactified Jacobian is far away from being a fine compactified Jacobian (indeed, in some sense, it is the most degenerate compactified Jacobian). Therefore, we will indicate why the proof of loc. cit. can be extended to the case of fine compactified Jacobians of nodal curves\footnote{Indeed, the same result is true, with the same proof, for any (non necessarily fine) compactified Jacobians in the sense of Oda-Seshadri \cite{OS}.}.

With this aim, we have to recall some results of Oda-Seshadri \cite{OS}  on the structure of compactified  Jacobians of nodal curves (see also \cite[Sec. 2]{Ale2} and \cite[Sec. 3.1]{MRV1}).
%however, we will follow the notation of \cite[Sec. 3.1]{MRV1} which is coherent with the one of the present paper).
First of all, any fine compactified Jacobian of $X$ is equivalent by translation in the sense of \cite[Def. 3.1]{MRV1} (hence isomorphic) to
a fine compactified Jacobian $\J_X(\un q)$ with total degree equal to $|\un q|=1-p_a(X)$; therefore, from now on
we will restrict to general polarizations $\un q$ such that $|\un q|=1-p_a(X)$. For any such polarization $\un q$, we can consider the new polarization $\phi(\un q)$ defined by
\begin{equation}\label{E:pol-phi}
\phi(\un q)_{C_i}:=\un q_{C_i}+\frac{\deg_{C_i}(\omega_X)}{2},
\end{equation}
for any irreducible component $C_i$ of $X$. Observe that $|\phi(\un q)|=|\un q|+p_a(X)-1=0$. From Remark \ref{R:tanteoss}\eqref{R:tanteoss3} and \cite[Formula (2) and \S 2.1]{Ale2} (see also the discussion in \cite[\S 2.5, \S 2.6]{MV} and \cite[\S 2.2]{CMKV}), it follows that $\J_X(\un q)$ is isomorphic to the Oda-Seshadri
compactified Jacobian ${\rm Jac}_{\phi(\un q)}(X)$.

Consider now the dual graph $\Gamma=\Gamma_X$ of the nodal curve $X$ and let $H_1(\Gamma,A)$ be the first homology group of the graph $\Gamma$ with coefficients in the commutative ring $A$ (in the sequel, we will consider $A=\bbZ$ or $\bbR$). It is well known that $H_1(\Gamma,A)\cong A^r$ for some integer $r$ which is called the rank of $\Gamma$.
The generalized Jacobian $J(X)$ of $X$ is a semiabelian variety and it fits into the extension (see \cite[Prop. 10.2]{OS})
$$0\to T\to J(X)\to J(X^{\nu})\to 0,$$
where $T\cong \Gm^r$ is an $r$-dimensional torus whose character group is canonically isomorphic to $H_1(\Gamma, \bbZ)\cong \bbZ^r$ and $J(X^{\nu})$ is the Jacobian of the normalization $X^{\nu}$ of $X$.

To any polarization $\un q$ on $X$ as above, there is associated a locally finite  arrangement $\cV_{\un q}$ of affine rational hyperplanes of the real vector space $H_1(\Gamma,\bbR)\cong \bbR^r$, which cuts
$H_1(\Gamma,\bbR)$ into infinitely many rational polytopes giving rise to a (face-to-face) complex $\cC_{\un q}$ of polytopes. An explicit definition of the arrangement $\cV_{\un q}$ (which we do not include here since it is not needed for what follows) can be found in \cite[Sec. 3.1]{MRV1}. From the definition of loc. cit., it is clear that  $\cC_{\un q}$ coincides with the Voronoi complex of polytopes ${\rm Vor}_{\phi(\un q)}$ defined in \cite[Sec. 2.6]{Ale2}, which is dual to the Namikawa complex of polytopes ${\rm Del}_{\phi(\un q)}$ defined in \cite[\S I.5, I.6]{OS}.
%\footnote{The arrangement of hyperplanes $\cV_{\un q}$ is equal to the arrangement of hyperplanes ${\rm Vor}_{\phi}$ of \cite[Sec. 2.6]{Ale2} for a certain parameter $\phi$ which depends on $\un q$.}.
The lattice $\bbZ^r\cong H_1(\Gamma,\Z)\subset H_1(\Gamma,\bbR)$ acts by translations on $H_1(\Gamma,\bbR)$ and preserves both the arrangement of hyperplanes $\cV_{\un q}$ and the complex of polytopes $\cC_{\un q}$.

For any rational polytope $\sigma\in \cC_{\un q}$, let $T_{\sigma}$ be the corresponding projective $T$-toric variety and consider the variety $Z_{\sigma}:=T_{\sigma} \times^T J(X)=(T_{\sigma}\times J(X))/T$ which maps to the $g^{\nu}(X)$-dimensional abelian variety $J(X^{\nu})$ with fibers isomorphic to $T_{\sigma}$. In \cite[Thm. 13.2]{OS} (see also \cite[Thm. 2.9]{Ale2}), it is shown that $\J_X(\un q)\cong {\rm Jac}_{\phi(\un q)}(X)$ is obtained by choosing representatives
$\{\sigma_1,\ldots, \sigma_n\}$ for the maximal polytopes in $\cC_{\un q}$ (which correspond to the vertices of  ${\rm Del}_{\phi(\un q)}$) modulo $H_1(\Gamma,\bbZ)$ and gluing
 the disjoint union $\coprod_{i} Z_{\sigma_i}$ according to the identification of the faces of the $\sigma_i$'s in  the quotient complex  $\cC_{\un q}/H_1(\Gamma,\bbZ)$.

An equivalent way to rephrase the above result of Oda-Seshadri is the following (see \cite[\S 6]{And} for more details).
The varieties $\{Z_{\sigma}\}_{\sigma\in \cC_{\un q}}$ glue together, according to the way the polytopes fit together in the face-to-face complex $\cC_{\un q}$, and give rise to a locally finite $k$-scheme $\wt{P}_{\un q}$.
The action of  the lattice $H_1(\Gamma, \bbZ)$ on $\cC_{\un q}$ gives rise to an action of $H_1(\Gamma, \bbZ)$
 on the scheme $\wt{P}_{\un q}$ for which there exists a quotient $\wt{P}_{\un q}/H_1(\Gamma,\bbZ)$, which is indeed isomorphic to the fine compactified Jacobian $\J_X(\un q)\cong {\rm Jac}_{\phi(\un q)}(X)$.

This realization of $\J_X(\un q)$ as a quotient $\wt{P}_{\un q}/H_1(\Gamma,\bbZ)$ has the same properties of the realization of a stable quasiabelian variety $P_0$ as a quotient $\wt{P}_0/Y$, which is described in \cite[Thm. 3.17]{AN}. Therefore, by a direct inspection, the same proof of \cite[Thm. 4.3]{AN} for the computation of $h^i(P_0,\O_{P_0})$ applies to the computation of $h^i(\J_X(\un q), \O_{\J_X(\un q)})$: the crucial property of $\cC_{\un q}$, which makes the proof of loc. cit. work also in our case, is that the geometric realization $\lvert \cC_{\un q}/H_1(\Gamma,\bbZ)\rvert$ of the quotient complex $\cC_{\un q}/H_1(\Gamma,\bbZ)$ is homeomorphic to a real torus of dimension $r=\dim T$. This is clearly true also in our case, since
$$\lvert \cC_{\un q}/H_1(\Gamma,\bbZ)\rvert\cong \lvert \cC_{\un q}\rvert /H_1(\Gamma,\bbZ)\cong
H_1(\Gamma,\bbR)/H_1(\Gamma, \bbZ)\cong \bbR^r/\bbZ^r. $$
\end{proof}
%Proposition \ref{P:H1-nodal}
%Every fine compactified Jacobian of a nodal curve $X$ is a stable quasiabelian variety of dimension $p_a(X)$ in the sense of \cite{AN}, as proved by Alexeev in \cite[Sec. 5.5]{Ale2}. Therefore, the conclusion now follows from \cite[Thm. 4.3]{AN}, which says that for a stable quasiabelian variety of dimension $g$ the structure sheaf has the same cohomology groups of the structure sheaf of an abelian variety of dimension $g$.

We will now prove Theorem C from the introduction for fine compactified Jacobians $\ov{J}_X(\un q)$ that satisfy  condition \eqref{E:H1-g}, and in particular for all fine
compactified Jacobians of nodal curves by the above Proposition \ref{P:H1-nodal}.
Note that, a posteriori,  it will follow from Corollary B that every fine compactified Jacobian $\ov{J}_X(\un q)$ of any curve $X$ with locally planar singularities satisfies condition \eqref{E:H1-g} (and even the stronger condition \eqref{E:Hi}). However, we do not know a direct proof of this fact avoiding the use of the Fourier-Mukai transform.

\vspace{0,2cm}

The special case of Theorem C that we are going to prove will follow from a more general result involving  the semiuniversal deformation family of $X$.
Let us fix the set-up.  Consider  the semiuniversal deformation family $\pi:\X\to \Spec R_X$  for $X$ as in \S\ref{S:DefX}.
The generalized Jacobian $J(X)$ and the fine compactified Jacobian $\ov{J}_X(\un q)$  deform over $\Spec R_X$ to, respectively,
the universal generalized Jacobian $v:J(\X)\to \Spec R_X$ (see Fact \ref{F:ungenJac}) and  the universal fine compactified Jacobian $u:\ov{J}_{\X}(\un q)\to \Spec R_X$
with respect to the polarization $\un q$ (see Theorem \ref{T:univ-fine}).

We are now going to define a  universal Poincar\'e line bundle $\Pun$ on the fiber product  $\ov{J}_{\X}(\un q)\times_{\Spec R_X}  J(\X)$, similarly to the definition \eqref{E:Poin-sheaf}. With that in mind,
consider  the triple product $\X\times_{\Spec R_X} \ov J_{\X}(\un q)\times_{\Spec R_X} J(\X)$ and, for any $1\leq i<j\leq 3$, denote by $p_{ij}$ the projection onto the product of the $i$-th and $j$-th factors.
Choose a universal sheaf $\wh{\I}$ on $\X\times_{\Spec R_X} \bJbar_{\X}$ (see Fact \ref{F:univ-Jac}\eqref{F:univ-Jac1}); denote by $\wh{\I}^0$ its restriction to  $\X\times_{\Spec R_X} J(\X)$ and, by an abuse of notation,
by $\wh{\I}$ its restriction to $\X\times_{\Spec R_X} \ov J_{\X}(\un q)$. Then the universal Poincar\'e line bundle $\Pun$ on the fiber product  $\ov{J}_{\X}(\un q)\times_{\Spec R_X}  J(\X)$ is defined by
\begin{equation}\label{E:Psheaf-un}
\Pun:={\mathcal D}_{p_{23}}(p_{12}^*\wh{\I}\otimes p_{13}^*\wh{\I}^0)^{-1}\otimes \mathcal D_{p_{23}}(p_{13}^*\wh{\I}^0)
\otimes {\mathcal D}_{p_{23}}(p_{12}^*\wh{\I})
\end{equation}
where ${\mathcal D}_{p_{23}}$ denotes the determinant of cohomology with respect to the morphism $p_{23}$.
%By using a universal sheaf $\wh{\I}$ on $\X\times_{\Spec R_X} \bJbar_{\X}$ (see Fact \ref{F:univ-Jac}\eqref{F:univ-Jac1}), it is possible to define a universal Poincar\'e line bundle $\Pun$ on the fiber product  $\ov{J}_{\X}(\un q)\times_{\Spec R_X}  J(\X)$, similarly to the definition \eqref{E:Poin-sheaf} of the Poincar\'e line bundle $\P$ on $\ov{J}_X(\un q)\times J(X)$ via the determinant of cohomology.
Since the determinant of cohomology commutes with base change, it follows
that $\Pun$ restricts to $\P$ on the central fiber of $\ov{J}_{\X}(\un q)\times_{\Spec R_X}  J(\X)\to \Spec R_X$.

Assume now that $X$ has locally planar singularities and that $\ov{J}_X(\un q)$ satisfies condition \eqref{E:H1-g}.
In analogy with the definition \eqref{E:beta-pola} of the map $\beta_{\un q}$,
the existence of a universal Poincar\'e line bundle $\Pun$ on the fiber product  $\ov{J}_{\X}(\un q)\times_{\Spec R_X}  J(\X)$ defines a morphism
\begin{equation}\label{E:univ-beta}
\beun: J(\X)\to \Pic^o(\ov{J}_{\X}(\un q)),
\end{equation}
between two group schemes which are of finite type, smooth and separated over $\Spec R_X$ (see Fact \ref{F:ungenJac}
and Theorem \ref{T:Pic-univ}\eqref{T:Pic-univ4}). Moreover, by construction, the fibers of $J(\X)\to \Spec R_X$ and of
$\Pic^o(\ov{J}_{\X}(\un q))\to \Spec R_X$ are non empty and geometrically connected. Therefore, using that $\Spec R_X$ is regular by Lemma
\ref{L:morpi}\eqref{L:morpi2}, we get that the schemes $J(\X)$ and $\Pic^o(\ov{J}_{\X}(\un q))$ are regular and connected, hence irreducible.
Since $\Pun$ restricts to $\P$ on the central fiber of
$\ov{J}_{\X}(\un q)\times_{\Spec R_X}  J(\X)\to \Spec R_X$, the morphism $\beun$ restricts to the morphism
$\beta_{\un q}:J(X)\to \Pic^o(\ov{J}_{X}(\un q))$ on the central fiber. Using Proposition \ref{P:beta-hom}, we can easily show that $\beun$ is a homomorphism of group schemes.

\begin{prop}\label{P:hom-betaun}
The morphism $\beun$ is a homomorphism of group schemes.
\end{prop}
\begin{proof}
Observe that, since the determinant of cohomology commutes with base change, the pull-back of $\Pun$ to the geometric fiber over any point $s\in \Spec R_X$ is equal to the Poincar\'e line bundle
$\P_{\ov s}$ over $\J_{\X_{\ov s}}(\un q^s)\times_{\ov{k(s)}} J(\X_{\ov s})$. This implies that the pull-back $(\beun)_{\ov s}$ of the morphism $\beun$ to the geometric fiber over $s$ coincides with the morphism  $\beta_{\un q^s}: J(\X_{\ov s})\to \Pic^o(\J_{\X_{\ov s}}(\un q^s))$ of \eqref{E:beta-pola} for the curve $\X_{\ov s}$.  Therefore, Proposition \ref{P:beta-hom} gives that $(\beun)_{\ov s}$ is a homomorphism of group schemes. We conclude that $\beun$ is a homomorphism of group schemes using the lemma below.
\end{proof}

\begin{lemma}\label{L:hom-grp}
Let $S$ be an integral scheme and let $f:G_1\to G_2$ be an $S$-morphism between  two $S$-group schemes.
Assume that $G_1\to S$ is smooth and $G_2\to S$ is separated. If the base change  $f_{\ov s}:(G_1)_{\ov s}\to (G_2)_{\ov s}$ of $f$ to the geometric fiber over the generic point $s\in \Spec S$ is a homomorphism of $\ov{k(s)}$-group schemes, then $f$ is a homomorphism of $S$-group schemes.
\end{lemma}
\begin{proof}
The fact that $f:G_1\to G_2$ is a homomorphism of $S$-group schemes amounts to checking the following three equalities of $S$-morphisms
\begin{enumerate}[(i)]
\item \label{E:ugua1} $f\circ 0_1=0_2:S\to  G_2$,
\item \label{E:ugua2} $f\circ m_1=m_2\circ (f\times f):G_1\times_S G_1\to G_2$,
\item \label{E:ugua3} $i_2\circ f=f\circ i_1:G_1\to G_2$,
\end{enumerate}
where $0_j:S\to G_j$ is the identity, $m_j:G_j\times_S G_j\to G_j$ is the multiplication  and $i_j:G_j\to G_j$ is the inverse of the $S$-group scheme $G_j$ (for $j=1,2$).

Using the diagonal $\Delta\subseteq G_2\times_S G_2$, we can reformulate the above equalities of morphisms in terms of equalities of $S$-schemes as follows:

% which is a closed subscheme since $G_2\to S$ is separated.
%The above three equalities \eqref{E:ugua1},  \eqref{E:ugua2} and \eqref{E:ugua3} are equivalent to the following properties:
\begin{enumerate}[(a)]
\item \label{E:facto1} $(f\circ 0_1,0_2)^{-1}(\Delta)=S$,
  %\to  G_2\times_S G_2$ factors through $\Delta$,
\item \label{E:facto2} $(f\circ m_1,m_2\circ (f\times f))^{-1}(\Delta)=G_1\times_S G_1$,
    %\times_S G_2$ factors through $\Delta$,
\item \label{E:facto3} $(i_2\circ f,f\circ i_1)^{-1}(\Delta)=G_1$,
 %\to G_2\times_S G_2$ factors through $\Delta$.
\end{enumerate}
where in each case we take the scheme-theoretic inverse image. Now observe that $\Delta\subseteq G_2\times_S G_2$ is closed since $G_2\to S$ is separated, hence its scheme-theoretic inverse image in  \eqref{E:facto1},  \eqref{E:facto2} and  \eqref{E:facto3} is also a closed subscheme.
Moreover, using that $G_1\to S$ is smooth and $S$ is integral by assumption, the three schemes $S$, $G_1\times_S G_1$ and $G_1$, appearing in  \eqref{E:facto1},  \eqref{E:facto2} and  \eqref{E:facto3}, are integral and smooth over $S$. Therefore, in order to check that we have equality of $S$-schemes in  \eqref{E:facto1},  \eqref{E:facto2} and  \eqref{E:facto3}, it is enough to prove that we have equalities when we restrict to the fibers over the generic point of $S$.  Furthermore, since the fact that a morphism is an isomorphism can be checked after a faithfully flat base change  (by \cite[(2.7.1)]{EGAIV2}), it is enough to prove that we have equalities when we  restrict to the geometric generic fiber over $S$.  But this is equivalent to saying that $f$ induces a group scheme homomorphism on the geometric generic fibers, which holds true by assumption, q.e.d.

\end{proof}

The main result of this Section is the following

\begin{thm}\label{T:isobetauniv}
Let $X$ be a reduced curve with locally planar singularities and let $\ov{J}_X(\un q)$ be a fine compactified
Jacobian of $X$ having the property that
$h^1(\ov{J}_X(\un q),\O_{\ov{J}_X(\un q)})=p_a(X)$.
Then the group homomorphism
$$\beun: J(\X)\to \Pic^o(\ov{J}_{\X}(\un q))$$
is an isomorphism.
\end{thm}
\begin{proof}
Consider the open subset $U\subseteq \Spec R_X$ consisting of all the points $s$ such that
the geometric fiber $\X_{\ov s}$ of the universal family $\pi:\X\to \Spec R_X$ over $s$ is smooth or has a unique singular point
which is a node. By Lemma \ref{L:morpi}\eqref{L:morpi3}, the complement of $U$ inside $\Spec R_X$ has codimension at least two.

\un{Claim  1:} The restriction of $\beun$ to $U$
$$(\beun)_{|U}:  J(\X)_{|U}\to \Pic^o(\ov{J}_{\X}(\un q))_{|U}$$
is an isomorphism. In particular, $\beun$ is an isomorphism in codimension one.

Indeed, since the map $\Pic^o(\ov{J}_{\X}(\un q))_{|U}\to U$ is flat, using \cite[(17.9.5)]{EGAIV4} it is enough to prove that the restrictions of $\beun$ to the fibers over $U$ are isomorphisms. Moreover, since the property  of being an isomorphism is invariant under faithfully flat base change  (see \cite[(2.7.1)]{EGAIV2}), it is enough to prove that the restriction of $\beun$ to the geometric fibers
\begin{equation}\label{E:fib-beta}
(\beun)_{\ov s}:  J(\X)_{\ov s}=J(\X_{\ov s})\to \Pic^o(\ov{J}_{\X}(\un q))_{\ov s}=\Pic^o(\ov{J}_{\X_{\ov s}}(\un q^s))
\end{equation}
is an isomorphism for every $s\in U$. By the definition of $U$, the geometric fibers $\X_{\ov s}$ can be of three types:
\begin{enumerate}[(i)]
\item \label{E:fib1} $\X_{\ov s}$ is smooth;
\item \label{E:fib2} $\X_{\ov s}$ is an irreducible curve having a unique singular point that  is a node;
\item \label{E:fib3} $\X_{\ov s}$ has two smooth irreducible components $\X_{\ov s}^1$ and $\X_{\ov s}^2$ which meet in a separating node.
\end{enumerate}
In cases \eqref{E:fib1} and  \eqref{E:fib2}, the fact that the morphism $(\beun)_{\ov s}$ is an isomorphism is a particular case of
the main result of Esteves-Gagn\'e-Kleiman in \cite[(2.1)]{egk} (the case of a smooth curve is classical), who proved that the same conclusion is true for any integral curve with double points singularities (in which case all the fine compactified Jacobian are isomorphic to $\bJbar_X^{0}$).
In case \eqref{E:fib3},
using Theorem \ref{T:Abel}\eqref{T:Abel1}, we get that $\ov{J}_{\X_{\ov s}}(\un q^s)\cong \ov{J}_{\X_{\ov s}^1}(\un q^1)\times
\ov{J}_{\X_{\ov s}^2}(\un q^2)$ for some general polarizations $\un q^i$ on $\X_{\ov s}^i$ (for $i=1,2$). The diagram \eqref{E:diagr1} in \S  \ref{S:Poincare} translates
into the following commutative diagram
\begin{equation}\label{E:diagr1bis}
\xymatrix{
& \Pic^o(\ov{J}_{\X_{\ov s}^1}(\un q^1))\times \Pic^o(\ov{J}_{\X_{\ov s}^2}(\un q^2)) \ar[dr]^{p_1^*(-)\otimes p_2^*(-)} & \\
 J(\X_{\ov s}^1)\times J(\X_{\ov s}^2) \ar[ur]^{\beta_{\un q^1}\times \beta_{\un q^2}}& & \Pic^o(\ov{J}_{\X_{\ov s}^1}(\un q^1)\times \ov{J}_{\X_{\ov s}^2}(\un q^2)) \ar[d]_{\cong}  \\
J(\X_{\ov s}) \ar[u]_{\cong} \ar[rr]_{\beta_{\un q^s}} & & \Pic^o(\ov{J}_{\X_{\ov s}}(\un q^s))}
\end{equation}
The maps $\beta_{\un q^1}$ and $\beta_{\un q^2}$ are isomorphism since $\X_{\ov s}^1$ and $\X_{\ov s}^2$ are smooth curves (as in case \eqref{E:fib1}); hence
the fact that $\beta_{\un q^s}$ is an isomorphism follows from the previous diagram together with the fact that $p_1^*(-)\otimes p_2^*(-)$ is an isomorphism by \cite[Cor. 4.7]{Lan} (using the fact that
$\Pic^o(\ov{J}_{\X_{\ov s}^i}(\un q^i))$ is smooth for $i=1,2$).

\vspace{0,2cm}

\un{Claim  2:} $\beun$ is an open embedding.

Indeed, since $\beun$ is a birational map between two integral schemes which is an isomorphism in
codimension one (by Claim 1) and the codomain is normal and locally factorial (being regular),
we deduce that $\beun$ is a local isomorphism by Van der Waerden's
theorem on the purity of the ramification locus (see \cite[(21.12.12)]{EGAIV4}). In particular, $\beun$ is
quasi-finite. Moreover, since $\beun$ is birational (by Claim 1) and separated (which follows from the fact that  $J(\X)\to \Spec R_X$ is separated, see \cite[Chap. II, Cor. 4.6(e)]{Har}) and the codomain is normal (being regular),
we deduce that $\beun$ is an open embedding by Zariski's main theorem (see \cite[(4.4.9)]{EGAIII1}).

\vspace{0,2cm}

We can now easily conclude the proof of the Theorem. Indeed, for any $s\in \Spec R_X$,   the restriction $(\beun)_{\ov s}$ of  \eqref{E:fib-beta} is a group homomorphism between two connected and smooth algebraic groups over $\ov{k(s)}$ of the same dimension, $p_a(X)$, which is moreover an open embedding by Claim 2. This forces $(\beun)_{\ov s}$ to be surjective (see e.g. \cite[Sec. 7.3, Lemma 1]{BLR}), hence an isomorphism.
Since the map $\Pic^o(\ov{J}_{\X}(\un q))\to \Spec R_X$ is flat, using again \cite[(17.9.5)]{EGAIV4}, it follows that $\beun$ is an isomorphism, q.e.d.

\end{proof}

\begin{cor}\label{C:isobeta}
Let $X$ be a reduced curve with locally planar singularities and let $\ov{J}_X(\un q)$ be a fine compactified
Jacobian of $X$ having the property that  $h^1(\ov{J}_X(\un q),\O_{\ov{J}_X(\un q)})=p_a(X)$. Then the group homomorphism
$$\beta_{\un q}: J(X)\to \Pic^o(\ov{J}_{X}(\un q))$$
is an isomorphism.
\end{cor}
\begin{proof}
This follows directly from  Theorem \ref{T:isobetauniv} by restricting to the central fiber.
\end{proof}

\section{Proof of Theorem A, Corollary B and Theorem C}\label{S:proof1}

The aim of this section is to prove the first three results that were stated in the introduction, namely Theorem A, Corollary B and Theorem C.

% \begin{assumption}\label{A:ass-Abel}
%Throughout this section, we fix a general polarization $\un q$ on $X$ such that the corresponding fine compactified Jacobian $\ov{J}_X(\un q)$ of $X$ admits an Abel map,
%i.e.  there exists $L\in \Pic^{|\un q|+1}(X)$ with the property that $\Im A_L\subseteq \ov{J}_X(\un q)$.
 %\end{assumption}

A key role will be played by the semiuniversal deformation family $\pi:\X\to \Spec R_X$  for $X$ as in \S\ref{S:DefX}.
More precisely,  we will be looking at the following Cartesian diagram
\begin{equation}\label{E:diag-ThmA}
\xymatrix{
&  \ov{J}_{\X}(\un q)\times_{\Spec R_X}  J(\X) \ar[dl]^{\wt{u}} \ar[dr]_{\wt{v}} & \\
J(\X) \ar[dr]^{v}\ar@{}[rr]|{\square} & & \ov{J}_{\X}(\un q) \ar[dl]_{u}\ar@(u,ul)[ul]_(0.3){\wt{\zeta}}\\
& \Spec R_X \ar@(dl,dl)[ul]^{\zeta} &
}
\end{equation}
where $v:J(\X)\to \Spec R_X$ is the universal generalized Jacobian (see Fact \ref{F:ungenJac}), $u:\ov{J}_{\X}(\un q)\to \Spec R_X$ is the universal fine compactified Jacobian
with respect to the polarization $\un q$ (see Theorem \ref{T:univ-fine}), $\zeta$ is the zero section of $v$ and $\wt{\zeta}:=\id\times \zeta$.
Let $\Pun$ be the universal Poincar\'e line bundle on  $\ov{J}_{\X}(\un q)\times_{\Spec R_X}  J(\X)$ as defined in \eqref{E:Psheaf-un}, which restricts to $\P$ on the central fiber $\ov{J}_{X}(\un q)\times  J(X)$.

%As observed already in Section \ref{S:nodalC}, there exists a universal Poincar\'e line bundle $\Pun$ on the fiber product  $\ov{J}_{\X}(\un q)\times_{\Spec R_X}  J(\X)$ which restricts to $\P$ on the central fiber $\ov{J}_{X}(\un q)\times  J(X)$.

Assuming that $X$ has locally planar singularities and setting $g:=p_a(X)$, the morphisms appearing in the above diagram satisfy the following properties:
the morphism $v$ (hence also $\wt{v}$) is smooth of relative dimension $g$ (see Fact \ref{F:univ-Jac}\eqref{F:univ-Jac1}); the morphism $u$ (and hence also $\wt{u}$)
is projective, flat of relative dimension $g$ with trivial relative dualizing sheaf and geometrically connected fibers (see Theorems \ref{T:compJac} and \ref{T:univ-Jac}).
Moreover all the schemes appearing in diagram \eqref{E:diag-ThmA} are regular: $\Spec R_X$ is regular by Lemma \ref{L:morpi}\eqref{L:morpi2};
$\ov{J}_{\X}(\un q)$ is regular by Theorem  \ref{T:univ-Jac}; $J(\X)$ (resp.  $\ov{J}_{\X}(\un q)\times_{\Spec R_X}  J(\X)$) is regular because the morphism $v$ (resp.
$\wt{v}$) is smooth over a regular codomain (see \cite[(17.5.8)]{EGAIV4}).

The following result is a generalization of a well-known result of Mumford for abelian varieties (see \cite[Sec. III.13]{Mum}) and it is the key for the proof of our main theorems.

\begin{thm}\label{T:push-for}
Let $X$ be a reduced curve with locally planar singularities of arithmetic genus $g:=p_a(X)$ and let $\un q$ be a general polarization.
There is a natural quasi-isomorphism of complexes of coherent sheaves on $J(\X)$:
\begin{equation}\label{E:mor-Phi}
\Phi: R\wt{u}_* (\Pun)\stackrel{\cong}{\longrightarrow} \zeta_*(\O_{\Spec R_X})[-g].
\end{equation}
In particular, we get that
\begin{equation}\label{directim}
R p_{2*} \P= {\bf k}(0)[-g]
\end{equation}
where ${\bf k}(0)$ denotes the skyscraper sheaf supported at the origin $0=[\O_X]\in J(X)$, and $p_2:\ov{J}_X(\un q)\times J(X)\to J(X)$  is the projection onto the second factor.
\end{thm}
\begin{proof}
Clearly, the last assertion follows from the first one by base change  to the central fiber of $v:J(\X)\to \Spec R_X$; hence it is enough to prove \eqref{E:mor-Phi}.

We will first explain how the morphism $\Phi$ is defined.
By applying base change  (see e. g. \cite[Prop. A.85]{BBH}) to the diagram \eqref{E:diag-ThmA} and using that $u$ is flat,  we get a natural isomorphism
\begin{equation}\label{E:basechan}
L\zeta^*(R\wt{u}_* (\Pun))\stackrel{\cong}{\longrightarrow} R u_* (L \wt{\zeta}^*(\Pun)).
\end{equation}
Consider now the right hand side of \eqref{E:basechan}. Since $\Pun$ is a line bundle, we have that $L \wt{\zeta}^*(\Pun)=\wt{\zeta}^*(\Pun)$.
From the definition \eqref{E:Psheaf-un}, using the functoriality of the determinant of cohomology and the fact that $(\id\times \wt \zeta)^*(p_{13}^*\wh{\I}^0)=\O_{\X\times \J_{\X}(\un q)}$ by the definition of $\wt \zeta$,
% we can compute
%$$\zeta^* \Pun={\mathcal D}_{p_{2}}((\id\times \wt \zeta)^*(p_{12}^*\wh{\I}\otimes p_{13}^*\wh{\I}^0))^{-1}\otimes \mathcal D_{p_{2}}(\wt (\id \times \zeta)^*(p_{13}^*\wh{\I}^0)) \otimes {\mathcal D}_{p_{2}}((\id\times \wt \zeta)^*((p_{12}^*\wh{\I}))= $$
% and the fact that $(\id,\wt \zeta)^* p_{13}^* \wt \I^0$
% Applying the morphism $\beta_{\un q}$ to the zero section $\zeta\in J(\X)(\Spec R_X)$ of $v$, we get an element $\beta_{\un q}(\zeta)\in \Pic^o(\J_{\X}(\un q))(\Spec R_X)$, which, by the definition of $\beta_{\un q}$, is represented by the line bundle $\wt{\zeta}^*(\Pun)\in {\mathcal Pic}(\J_{\X}(\un q))$. It follows from Proposition \ref{P:hom-betaun} that the element $\beta_{\un q}(\zeta)\in \Pic^o(\J_{\X}(\un q))(\Spec R_X)$ is zero, which implies that  $\wt{\zeta}^*(\Pun)$ is the pull-back of a line bundle from $\Spec R_X$ (see \cite[Sec. 8.1, Prop. 4]{BLR}).
%However, since $R_X$ is a power series ring  by Lemma \ref{L:morpi}\eqref{L:morpi2}, the only line bundle on $\Spec R_X$ is the trivial one; hence
we deduce that $\wt{\zeta}^*(\Pun)= \O_{\ov{J}_{\X}(\un q)}$.
Using this, we get an identification
\begin{equation}\label{E:Lzeta}
R u_* (L \wt{\zeta}^*(\Pun))=R u_* (\wt{\zeta}^*(\Pun))=R u_*(\O_{\ov{J}_{\X}(\un q)}).
\end{equation}
Since the complex of sheaves $R u_*(\O_{\ov{J}_{\X}(\un q)})$ is concentrated in cohomological degrees
from $0$ to $g$, we get a morphism of complexes of sheaves
\begin{equation}\label{E:mor-to-g}
R u_*(\O_{\ov{J}_{\X}(\un q)})\to R^g u_*(\O_{\ov{J}_{\X}(\un q)})[-g].
\end{equation}
Moreover, since the morphism $u$ is projective of relative dimension $g$, with trivial relative dualizing sheaf and geometrically connected fibers,
then the relative duality  applied to $u$ gives that (see \cite[Chap. III, Cor. 11.2(g)]{Har-bis}) :
\begin{equation}\label{E:iso-deg-g}
R^g u_*(\O_{\ov{J}_{\X}(\un q)})\cong  \O_{\Spec R_X}.
%\cong \Hom(u_*(\O_{\ov{J}_{\X}(\un q)}), \O_{\Spec R_X})\cong \O_{\Spec R_X},
\end{equation}
%Finally we have a morphism
%\begin{equation}\label{E:Hom-mor}
%\Hom(R u_*(\O_{\ov{J}_{\X}(\un q)}), \O_{\Spec R_X})\to \Hom(\O_{\Spec R_X}, \O_{\Spec R_X})\cong \O_{\Spec R_X},
%\end{equation}
%which is induced by applying the functor $\Hom(-, \O_{\Spec R_X})$ to the composition of the two natural
%morphisms $\O_{\Spec R_X}\to u_*(\O_{\ov{J}_{\X}(\un q)})\to Ru_*(\O_{\ov{J}_{\X}(\un q)})$.
Putting together  \eqref{E:basechan},  \eqref{E:Lzeta}, \eqref{E:mor-to-g} and  \eqref{E:iso-deg-g},
we get a morphism
\begin{equation}\label{E:basech2}
\Psi:L\zeta^*(R\wt{u}_* (\Pun))\to \O_{\Spec R_X}[-g].
\end{equation}
Since $L\zeta^*$ is left adjoint to $R\zeta_*$ (see \cite[p. 83]{Huy}) and $R\zeta_*\O_{\Spec R_X}\cong \zeta_*\O_{\Spec R_X}$ because $\zeta$ is a closed embedding (hence $\zeta_*$ is an exact functor),
 the morphism $\Psi$ gives rise to the morphism $\Phi$ by adjunction.

The remaining part of the proof will be devoted to showing that the morphism $\Phi$ is a quasi-isomorphism
of complexes of sheaves.  We need some preliminary results that we collect under the name of Claims.
The first result says that $\Phi$
is generically a quasi-isomorphism. More precisely, let $(\Spec R_X)_{\rm sm}$ be the open subset of $\Spec R_X$
consisting of the points $s\in \Spec R_X$ such that $\X_{\ov s}$ is smooth. Then we have:

\vspace{0,1cm}

\un{Claim  1:} The morphism $\Phi$ is a quasi-isomorphism over the open subset $v^{-1}((\Spec R_X)_{\rm sm})$.

Indeed, $A:=v^{-1}((\Spec R_X)_{\rm sm})$ is an abelian group scheme over $B:=(\Spec R_X)_{\rm sm}$ via the map $v$. Therefore, it follows from \cite[Proof of Theorem 1.1]{mukai2} (which generalizes the classical result of
Mumford \cite[Sec. III.13]{Mum} for abelian varieties over a field), that we have quasi-isomorphisms
\begin{equation}\label{E:Muk-iso}
R\wt{u}_*(\Pun)_{|A}\cong  R^g \wt{u}_*( \Pun)[-g]_{|A}\cong \zeta_*(\zeta^*\omega_{A/B})[-g].
\end{equation}
However, since $R_X$ is a power series ring, the line bundle $\zeta^*(\omega_{A/B})$  is trivial on $B=(\Spec R_X)_{\rm sm}$.
By comparing the construction of the quasi-isomorphism \eqref{E:Muk-iso} in loc. cit. and our definition of the morphism $\Phi$, one can easily check that the quasi-isomorphism \eqref{E:Muk-iso} coincides with the restriction $\Phi_{|A}$ of $\Phi$ to the open subset $A=v^{-1}((\Spec R_X)_{\rm sm})$, up to possibly multiplying by the pullback of an invertible function on $B$ (depending on the choice of a trivialization of $\zeta^*(\omega_{A/B})$ on $B$). Therefore, we conclude that $\Phi_{|A}$ is a quasi-isomorphism.

\vspace{0,2cm}

\un{Claim  2:} We have that $\codim( \supp (R\wt{u}_* \Pun))=g$.

%Let us first prove that
%\begin{equation}\label{E:codim-push}
%\codim( \supp (R\wt{u}_* \Pun))=g.
%\end{equation}
First of all, Claim 1 gives that $\codim(\supp  (R\wt{u}_* \Pun))\leq g$. In order to prove the reverse inequality, we stratify
the scheme $\Spec R_X$  into locally closed subsets (see Lemma \ref{L:lower-semcont}) according to the geometric genus of the fibers of the universal family $\X\to \Spec R_X$:
$$(\Spec R_X)^{g^{\nu}=l}:=\{s\in \Spec R_X\: :\: g^{\nu}(\X_{\ov s})=l\}, $$
for any $g^{\nu}(X)\leq l\leq p_a(X)=g$. Corollary \ref{C:Diaz-Har} gives that $\codim (\Spec R_X)^{g^{\nu}=l}\geq g-l$. On the other hand, on the fibers of $v$ over
$(\Spec R_X)^{g^{\nu}=l}$, the sheaf $R\wt{u}_* \Pun$ has support of codimension at least $l$ by Corollary \ref{codim}. Therefore, we get
\begin{equation}\label{E:supp-strata}
\codim (\supp (R\wt{u}_* \Pun) \cap v^{-1}((\Spec R_X)^{g^{\nu}=l}))\geq g \text{ for any }Êg^{\nu}(X)\leq l\leq g.
\end{equation}
Since the locally closed subsets $(\Spec R_X)^{g^{\nu}=l}$ form a stratification of $\Spec R_X$, we deduce that
$g\leq \codim( \supp (R\wt{u}_* \Pun))$, which concludes
the proof of Claim 2.

\un{Claim 3:} The complex $R\wt{u}_*(\Pun)$ is supported in cohomological degree $g$, i.e.
$R\wt{u}_*(\Pun)[g]\cong R^g \wt{u}_*( \Pun)$.

We apply the relative duality (see e.g.  \cite[Chap. VII.3]{Har-bis}) to the projective morphism $\wt{u}$. Since $\wt{u}$ is flat of relative dimension $g$
and it has  trivial relative dualizing sheaf, we get a quasi-isomorphism
\begin{equation}\label{E:rel-duality}
R\HHom (R\wt{u}_*\,(\Pun)^{-1}, \O_{J(\X)})\cong R\wt{u}_* \Pun[g],
\end{equation}
where $(\Pun)^{-1}$ is the inverse of $\Pun$, i.e. $(\Pun)^{-1}:=\HHom(\Pun,\O_{\ov{J}_{\X}(\un q)\times_{\Spec R_X}  J(\X)})$.
The left hand side of \eqref{E:rel-duality} can be computed using the following spectral sequence (see \cite[Chap. 3, Formula (3.8)]{Huy}):
\begin{equation}\label{E:spec-seq}
E_2^{p,q}=\EExt^p(R^{-q}\wt{u}_*(\Pun)^{-1},\O_{J(\X)}) \Rightarrow \EExt^{p+q}(R\wt{u}_*(\Pun)^{-1},\O_{J(\X)}),
\end{equation}
where clearly $E_2^{p,q}=0$ unless $0\leq -q\leq g$ and $p\geq 0$.
If we denote by $i$ the involution of the group scheme $v:J(\X)\to \Spec R_X$ that sends $\calM\in J(\X)$ into $\calM^{-1}\in J(\X)$, then
Proposition \ref{P:hom-betaun} gives that $(\Pun)^{-1}=(\id\times i)^*(\Pun)$; hence
\begin{equation}\label{E:inv-Poi}
R\wt{u}_*\,(\Pun)^{-1}=i^*(R\wt{u}_*\Pun).
\end{equation}
In particular, the complex $R\wt{u}_*\,(\Pun)^{-1}$ has  codimension
$g$ by  Claim 2. This implies that for any $0\leq -q\leq g$, the sheaf $R^{-q}\wt{u}_*(\Pun)^{-1}$ has codimension at least $g$; hence, since
the dualizing sheaf of $J(\X)$ is trivial, we get that (see \cite[Prop. 1.1.6]{HL}):
\begin{equation}\label{E:van-Ext}
E_2^{p,q}=\EExt^p(R^{-q}\wt{u}_*(\Pun)^{-1},\O_{J(\X)})=0 \: \: \text{ for every } p<g \text{ and every } q.
\end{equation}
Using the spectral sequence \eqref{E:spec-seq} and the vanishing \eqref{E:van-Ext}, it is easily seen that the complex in the left hand side of \eqref{E:rel-duality}
can have non-vanishing cohomology only in non-negative degrees. On the other hand, since $\wt{u}$ has fibers of dimension $g$, the complex in the right hand side of
\eqref{E:rel-duality} can have non-vanishing cohomology only in degrees belonging to the interval $[-g,0]$. Putting together these two results, we deduce that the two complexes in \eqref{E:rel-duality} must be supported in cohomological degree $0$, which concludes the proof of Claim 3.

%$R\wt{u}_*(\Pun)$ is supported in degree $g$, i.e. that
%\begin{equation}\label{E:degree-g}
%R\wt{u}_*(\Pun)\cong R^g \wt{u}_*( \Pun)[-g].
%\end{equation}

\un{Claim 4:} $R^g\wt{u}_*(\Pun)$ is a Cohen-Macaulay sheaf of codimension $g$.

 Indeed, consider the complex $R\wt{u}_*(\Pun)^{-1}$, which  is also supported in cohomological degree $g$ by \eqref{E:inv-Poi} and Claim 3, i.e.  $R\wt{u}_*(\Pun)^{-1}\cong R^g \wt{u}_*( \Pun)^{-1}[-g]$.
 Substituting into \eqref{E:rel-duality}  and using again Claim 3, we get
\begin{equation}\label{E:1Ext}
\EExt^p(R^{g}\wt{u}_*(\Pun)^{-1},\O_{J(\X)})=
\begin{sis}
R^g\wt{u}_*(\Pun)&\:\: \text{ if } p=g,\\
0  & \:\: \text{ if } p\neq g. \\
\end{sis}
\end{equation}
This implies that $R^{g}\wt{u}_*(\Pun)^{-1}$ is Cohen-Macaulay of codimension $g$ by \cite[Cor. 3.5.11]{BH}. Using \eqref{E:inv-Poi}, we get that
also $R^{g}\wt{u}_*(\Pun)$ is Cohen-Macaulay of codimension $g$, q.e.d.

\vspace{0,2cm}

\un{Claim  5:} We have a set-theoretic equality $\supp(R^g \wt{u}_*( \Pun))=\Im (\zeta)$.

Observe that the pull-back $\supp(R^g \wt{u}_*( \Pun))_{|J(\X_{\ov s})}$ of $\supp(R^g \wt{u}_*( \Pun))$ to the geometric fiber $J(\X_{\ov s})$ of $v$ over $s\in \Spec R_X$
 is equal to the locus of all elements $M\in J(\X_{\ov s})$ such that $H^g(\ov{J}_{\X_{\ov s}}(\un q^s), \P^s_M)\neq 0$, where we have set $\P_M^s:=(\Pun)_{|\ov{J}_{\X_{\ov s}}(\un q^s)\times \{M\}}$.
Since $\ov{J}_{\X_{\ov s}}(\un q^s)$ has trivial dualizing sheaf by Theorem \ref{T:compJac}\eqref{T:compJac4}, Serre's duality gives that
$$H^g(\ov{J}_{\X_{\ov s}}(\un q^s), \P^s_M)\cong H^0(\ov{J}_{\X_{\ov s}}(\un q^s), (\P^s_M)^{-1})^{\vee}= H^0(\ov{J}_{\X_{\ov s}}(\un q^s), \P^s_{M^{-1}})^{\vee}.$$
Applying now Proposition \ref{P:nonvaH0}, whose hypothesis are satisfied by  Lemma \ref{L:morpi}\eqref{L:morpinew}, we get the  set-theoretic equality
\begin{equation}\label{E:supp-fibers}
\supp(R^g \wt{u}_*( \Pun))_{|J(\X_{\ov s})}=\{M\in J(\X_{\ov s}) \: : \:  \P^s_{M^{-1}}\cong \O_{\ov{J}_{\X_{\ov s}}(\un q^s)}\} \: \: \text{ for every } s\in \Spec R_X.
\end{equation}
Moreover, combining \eqref{E:supp-fibers} with Corollary \ref{C:isobeta} and  Proposition \ref{P:H1-nodal}, we get that
\begin{equation}\label{E:supp-nodal}
\supp(R^g \wt{u}_*( \Pun))_{|J(\X_{\ov s})}=\{\O_{\X_{\ov s}}\} \: \: \text{ for every } s\in \Spec R_X \text{ such that } \X_{\ov s} \: \text{ is nodal. }
\end{equation}
The above formula \eqref{E:supp-nodal} allows us to improve the estimate \eqref{E:supp-strata} on the codimension of the intersection of $\supp(R^g \wt{u}_*( \Pun))$ with
the locally closed subset $v^{-1}((\Spec R_X)^{g^{\nu}=l}))$.  Indeed, by looking at the proof of  \eqref{E:supp-strata}, we see that we can have an equality in \eqref{E:supp-strata}
for some $l$ such that $g^{\nu}(X)\leq l\leq g$  only if:
\begin{itemize}
\item the image of  $\supp(R^g \wt{u}_*( \Pun))$ via the morphism $v$ contains a generic point $\eta$ of $(\Spec R_X)^{g^{\nu}=l}$ of codimension $g-l$ in $\Spec R_X$,
\item the codimension of $\supp(R^g \wt{u}_*( \Pun))\cap v^{-1}(\eta)$ in $J(\X_{\eta})$ is equal to $l$.
\end{itemize}
However, since a generic point $\eta$ of the stratum $(\Spec R_X)^{g^{\nu}=l}$ of $\Spec R_X$
is such that $\X_{\ov \eta}$ is nodal by Theorem \ref{T:strati}\eqref{T:strati2}, formula \eqref{E:supp-nodal} tells us that equality in  \eqref{E:supp-strata}
 is only possible for $l=g$; in other words we have that
\begin{equation}\label{E:supp-lowerstra}
\codim (\supp (R\wt{u}_* \Pun) \cap v^{-1}((\Spec R_X)^{g^{\nu}=l}))\geq g+1 \text{ for any } g^{\nu}(X)\leq l< g.
\end{equation}
After these preliminaries, we can now finish the proof of Claim 5.
Since $R^g \wt{u}_*( \Pun)$ is a Cohen-Macaulay sheaf of codimension $g$ by Claim 4, then all the irreducible components of  $\supp(R^g \wt{u}_*( \Pun))$
have codimension $g$ by \cite[Thm. 6.5(iii), Thm. 17.3(i)]{Mat}. Let $Z$ be an irreducible component of $\supp(R^g \wt{u}_*( \Pun))$. Using \eqref{E:supp-lowerstra} and the fact that $Z$ has codimension $g$,
we get that $v(Z)$ must contain the generic point $\eta$  of $\Spec R_X$. Then Claim 1 implies that necessarily we must have $Z=\Im(\zeta)$, q.e.d.

%Clearly $\Im (\zeta)$ is one of the irreducible components of  $\supp(R^g \wt{u}_*( \Pun))$. If $Z$ is another irreducible component of $\supp(R^g \wt{u}_*( \Pun))$ then, using \eqref{E:supp-lowerstra} and the fact that $Z$ has codimension $g$, we get that $v(Z)$ must contain the generic point $\eta$  of $\Spec R_X$. Then Claim 1 implies that necessarily we must have $Z=\Im(\zeta)$, q.e.d.

\vspace{0,2cm}

\un{Claim  6:} We have a scheme-theoretic equality $\supp(R^g \wt{u}_*( \Pun))=\Im (\zeta)$.

Since the subscheme $\Im (\zeta)$ is reduced, the inclusion of subschemes $\Im (\zeta)\subseteq  \supp(R^g \wt{u}_*( \Pun))$ follows from Claim 5.
Moreover, Claim 1 says that this inclusion is generically an equality; in particular $\supp(R^g \wt{u}_*( \Pun))$ is generically reduced.
Furthermore, since $R^g \wt{u}_*( \Pun)$ is a Cohen-Macaulay sheaf by Claim 2, Lemma \ref{L:CM-supp} below implies that $\supp(R^g \wt{u}_*( \Pun))$ is reduced.
Therefore, we must have the equality of subschemes $\supp(R^g \wt{u}_*( \Pun))=\Im (\zeta)$.

\vspace{0,2cm}

We can now finish the proof of the fact that $\Phi$ is a quasi-isomorphism.  Observe that by Claim 1, the shifted morphism $\Phi[g]$ can be regarded as a morphism of sheaves
$$\Phi[g]: R^g \wt{u}_*(\Pun)\to \zeta_*(\O_{\Spec R_X}).$$
Moreover, using Claim 6, we get that $\Phi$ is a quasi-isomorphism if and only if
$\zeta^*\Phi[g]$  is an isomorphism. By definition of $\Phi$ and using Claim 3, the shifted pull-back
\begin{equation}\label{E:basech2bis}
\zeta^*\Phi[g]: \zeta^*R^g \wt{u}_*(\Pun)\to \zeta^*\zeta_*(\O_{\Spec R_X})=\O_{\Spec R_X}
\end{equation}
coincides, up to the shift, with the morphism $H^g(\Psi)$ induced by the morphism $\Psi$ of \eqref{E:basech2}.
By tracing back the definition of the morphism $\Psi$, we get that $\zeta^*\Phi[g]$ is the composition of the top degree base change
morphism
\begin{equation}\label{E:base-bis}
\zeta^*R^g \wt{u}_*(\Pun)\to R^g u_*(\wt{\zeta}^*(\Pun))=R^g u_* (\O_{\ov{J}_{\X}(\un q)})
\end{equation}
with the isomorphism \eqref{E:iso-deg-g}. However, since $u$ has fibers of dimension $g$, the top degree base change  \eqref{E:base-bis} is a quasi-isomorphism, hence we are done.

\end{proof}

\begin{lemma}\label{L:CM-supp}
Let $Y$ be a Noetherian scheme and $\F$ a coherent sheaf on $Y$. Assume that $\F$ is Cohen-Macaulay and that the scheme-theoretic support $\supp(\F)$ of $\F$ is generically reduced.
Then $\supp(\F)$ is reduced.
\end{lemma}
\begin{proof}
The statement is clearly local; hence we may assume that $Y=\Spec R$ with $R$ a Noetherian ring and that $\F$ is equal to the sheafification of a finitely generated module $M$ over $R$.
Therefore, $\supp(\F)$ is the closed subscheme $V(\ann(M))$ of $\Spec R$ defined by the annihilator ideal $\ann(M)$ of $M$.
%$$\ann(M):=\{r\in R\: : \: r\cdot m=0 \: \:\text{ for every } m\in M\}.$$
Consider the set $\Ass(M):=\{P_1,\cdots, P_r\}$ of associated primes of $M$. Since $M$ is a Cohen-Macaulay module,
all its associated primes are minimal by \cite[Thm. 17.3(i)]{Mat}; therefore $\{P_1,\ldots,P_r\}$ are exactly
the associated minimal primes of $R/\ann(M)$ by \cite[Thm. 6.5(iii)]{Mat}.

Consider now a finite set of generators $\{m_1,\cdots,m_s\}$ of the $R$-module $M$. Clearly, we have that
\begin{equation}\label{E:annM}
\ann(M)=\bigcap_{i=1}^s \ann(m_i),
\end{equation}
where $\ann(m_i)$ is the annihilator ideal of the element $m_i\in M$.
%$$\ann(m_i):=\{r\in R\: : \: r\cdot m_i=0 \}.$$

Since we have an inclusion $R/\ann(m_i)\hookrightarrow M$ of $R$-modules obtained by sending the class of $1$ to $m_i$,
the set of associated primes of $R/\ann(m_i)$ is a subset of $\Ass(M)$; say $\Ass(R/\ann(m_i))=\{P_j \: : \: j\in A_i\}$
for some $A_i\subseteq \{1,\cdots, r\}$. In particular, $R/\ann(m_i)$ does not have embedded primes.
Moreover, since $V(\ann(m_i))\subseteq V(\ann(M))$ and $V(\ann(M))$ is generically reduced along the subvarieties
$V(P_i)$ by hypothesis, it follows that $V(\ann(m_i))$ is also generically reduced. This implies that $V(\ann(m_i))$
is reduced, or in other words that
\begin{equation}\label{E:annmi}
\ann(m_i)=\bigcap_{j\in A_i} P_j.
\end{equation}
Combining \eqref{E:annM} and \eqref{E:annmi}, together with the fact that $\{P_1,\cdots, P_r\}\subseteq \Ass(R/\ann(M))$, we
get that
\begin{equation}\label{E:ann-decompo}
\ann(M)=\bigcap_{i=1}^r P_i,
\end{equation}
which shows that $V(\ann(M))$ is a reduced subscheme of $\Spec R$.
\end{proof}

The formula \eqref{directim} established in Theorem \ref{T:push-for} allows us to prove Theorem A from the introduction, following the original approach of Mukai \cite[Thm. 2.2]{mukai}.

\begin{proof}[Proof of Theorem A]
We have to show that the integral transform, with kernel the Poincar\'e line bundle $\P$ on $\ov{J}_X(\un q)\times J(X)$,
\begin{eqnarray*}
\Phi^{\P}_{J(X)\to \ov{J}_X(\un q)}=\Phi^{\P}:D^b(J(X))& \longrightarrow & D^b(\ov{J}_X(\un q))\\
\cplx{E} &\longmapsto & R  p_{1*}(p_2^*(\cplx{E})\otimes \P),
\end{eqnarray*}
is fully faithful.

Since $\ov J_X(\un q)$ is a projective variety of dimension $g:=p_a(X)$ with trivial dualizing sheaf,  the functor $\Phi^{\P}$ admits as a left adjoint the following integral transform
(see \cite[Prop. 1.17]{HLS08})
\begin{eqnarray*}
\Phi^{\P^{-1}[g]}_{\ov{J}_X(\un q)\to J(X)}=\Phi^{\P^{-1}[g]}:D^b(\ov{J}_X(\un q)) & \longrightarrow & D^b(J(X)) \\
\cplx{E} &\longmapsto & R  p_{2*}(p_1^*(\cplx{E})\otimes \P^{-1}[g]).
\end{eqnarray*}
In order to show that $\Phi^{\P}$ is fully faithful, it is sufficient (and necessary, see e.g. \cite[\href{http://stacks.math.columbia.edu/tag/07RB}{Tag07RB}]{ST}) to show that the composition $\Phi^{\P^{-1}[g]}\circ \Phi^{\P}$ is an isomorphism.

By the standard convolution formula (see e.g. \cite[Prop. 1.3]{mukai}), the composition $\Phi^{\P^{-1}[g]}\circ \Phi^{\P}$ is the integral functor $\Phi^{\M}_{J(X)\to J(X)}:=\Phi^{\M}$ with kernel given by
\begin{equation}\label{E:ker-comp}
\M:=R  p_{13*}(p_{12}^*(\P)\otimes p_{23}^*(\P^{-1}[g])),
\end{equation}
where $p_{ij}$ are the obvious projections from $J(X)\times \J_X(\un q)\times J(X)$. Consider now the Cartesian diagram
\begin{equation}\label{E:divmap}
\xymatrix{
J(X)\times \J_X(\un q) \times J(X)  \ar[d]_{p_{13}}\ar[r]^{\wt n}\ar@{}[dr]|{\square}&  \J_X(\un q)\times J(X)\ar[d]^{p_2}\\
J(X)\times J(X)\ar[r]^{n} &  J(X),
}
\end{equation}
where $n$ is the morphism sending $(M,N)\in J(X)\times J(X)$ into $M\otimes N^{-1}\in J(X)$ and $\wt n$ sends $(M,I, N)\in J(X)\times \J_X(\un q)\times J(X)$ into $(I, M\otimes N^{-1})\in \J_X(\un q)\times J(X)$.
By Proposition \ref{P:beta-hom}, it follows that
$$\left(p_{12}^*(\P)\otimes p_{23}^*(\P^{-1})\right)_{\{M\}\times \J_X(\un q)\times \{N\}}=\P_M\otimes (\P_{N})^{-1}=\P_{M\otimes N^{-1}} \text{ for any } M, N\in J(X).$$
Therefore, by the seesaw principle, we get that
\begin{equation}\label{E:pullP}
p_{12}^*(\P)\otimes p_{23}^*(\P^{-1})=\wt n^*(\P)\otimes p_{13}^*(L),
\end{equation}
for some line bundle $L$ on $J(X)\times J(X)$. Now, applying the base change  formula  to the diagram \eqref{E:divmap} (using that $n$ is flat morphism), formula \eqref{directim} and the projection formula, we obtain that
\begin{equation}\label{E:forM}
\M=R  p_{13*}(p_{12}^*(\P)\otimes p_{23}^*(\P^{-1}[g]))\cong n^*(R p_{2*}(\P[g])) \otimes L\cong n^{*}({\bf k}(0)) \otimes L=\O_{\Delta}\otimes L=L_{|\Delta},
\end{equation}
where $\Delta$ is the diagonal of $J(X)\times J(X)$. This show that the integral functor $\Phi^{\M}$ is equal to the tensor product with the line bundle $L_{|\Delta}$ on $\Delta\cong J(X)$, hence an isomorphism, q.e.d.
\end{proof}

Corollary B follows now quite easily from Theorem \ref{T:push-for} and Theorem A.

\begin{proof}[Proof of Corollary B]
Let $M\in J(X)$.
If $M\neq [\O_X]$ then  the vanishing of  $H^i(\ov{J}_{X}(\un q), \P_M)$ for any $i$ follows from \eqref{directim}.

If $M=[\O_X]=0$ is the identity element of $J(X)$ then obviously $\P_M=\O_{\J_X(\un q)}$. Observe that, from the definition of $\Phi^{\P}$, it follows that
\begin{equation*}\label{E:skyscraper}
\Phi^{\P}({\bf k}(0))= R  p_{1*}(p_2^*({\bf k}(0))\otimes \P)=  R  p_{1*}( \P_{|\ov{J}_X(\un q)\times \{0\}})=\P_{\O_{X}}=\O_{\J_X(\un q)},
\end{equation*}
where ${\bf k}(0)$ denotes the structure sheaf of the point $0\in J(X)$.
Using  the fully faithfulness of the
integral transform $\Phi^{\P}$ (see Theorem A), we get
$$H^i(\J_X(\un q), \O_{\J_X(\un q)})={\rm Ext}^i_{\J_X(\un q)}(\O_{\J_X(\un q)},\O_{\J_X(\un q)})={\rm Ext}^i_{\J_X(\un q)}(\Phi^{\P}({\bf k}(0)),
\Phi^{\P}({\bf k}(0)) )={\rm Ext}^i_{J(X)}({\bf k}(0), {\bf k}(0)).$$
Now we conclude using the well-known facts that ${\rm Ext}^i_{J(X)}({\bf k}(0), {\bf k}(0))=\Lambda^i
{\rm Ext}^1_{J(X)}({\bf k}(0), {\bf k}(0))$ (using the Koszul resolution, see e.g. the proof of \cite[Cor. 2, p. 129]{Mum}) and that ${\rm Ext}^1_{J(X)}({\bf k}(0),{\bf k}(0))$
is canonically isomorphic to the tangent space of $J(X)$ at $0$, which is isomorphic to $H^1(X,\O_X)$
(see e.g. \cite[Sec. 8.4]{BLR}).
\end{proof}

We can now prove that autoduality holds for fine compactified Jacobians.

\begin{proof}[Proof of Theorem C]
Theorem C follows from Corollary \ref{C:isobeta}, whose hypothesis is satisfied by  Corollary B.

%Fix a general polarization $\un q$ on $X$ and set $\ov{J}:=\ov{J}_X(\un q)$. The map
%$$\begin{aligned}
%\beta_{\un q}:\Pic^{\un 0}(X)=J(X) & \to \Pic^o(\ov{J}), \\
%M& \mapsto \P_M,
%\end{aligned}$$
%is an injective homomorphism of connected group schemes by Corollary \ref{C:inj-beta}.
%Consider the differential of the map $\beta_{\un q}$ at the origin $\zeta=[\O_X]$ of $J(X)$:
%$$d\beta_{\un q}: T_{\zeta} J(X)=H^1(X, \O_X) \to T_{\O_{\ov{J}}} \Pic^o(\ov{J})=H^1(\ov{J}, \O_{\ov{J}}).$$ÊÊ
%The linear map $d \beta_{\un q}$ is naturally  identified with the linear map induced by the Fourier-Mukai functor $\FF_{\P}$ with kernel
%equal to the Poincar\'e bundle $\P$:
%$$H^1(X,\O_X)=\Ext_{J(X)}^1(\O_{\zeta}, \O_{\zeta})\stackrel{\FF_{\P}}{\longrightarrow} \Ext^1_{\ov{J}}(\FF_{\P}(\O_{\zeta}), \FF_{\P}(\O_{\zeta}))=Ê\Ext^1_{\ov{J}}(\O_{\ov{J}}, \O_{\ov{J}})=
%H^1(\ov{J}, \O_{\ov{J}})$$
%where we have used the fact that $\FF_{\P}(\O_{\zeta})=\O_{\ov{J}}$ (see the proof of Corollary B). Since $\FF_{\P}$ is fully-faithful according to Theorem A, we deduce that the above
%map, and hence $d\beta_{\un q}$, is an isomorphism. This implies that the image of $\beta_{\un q}$ is open; since it is also closed (being the image of an homomorphism of group schemes)
%and $\Pic^o(\ov{J})$ is connected,  we deduce that $\beta_{\un q}$ is surjective. This implies that $\beta_{\un q}$ is an isomorphism of group schemes.
\end{proof}

%Finally, we prove that algebraic equivalence and numerical equivalence coincide for fine compactified Jacobians.

%\begin{proof}[Proof of Theorem D]
%This follows from Theorem \ref{T:Pico=tau} using Corollary B.
%\end{proof}

\section{Proof of Theorem D}\label{S:ThmD}

The aim of this section is to prove Theorem D from the introduction. We will first prove the result for fine compactified Jacobians of curves  that admit an Abel map,
in the sense of \S \ref{S:Abel}, and under the assumption that the underlying curve does not have separating nodes.

\begin{thm}\label{T:thmD-Abel}
Let $X$ be a reduced curve with locally planar singularities and without separating nodes. Let $\un q$ be a general polarization on $X$ such that the associated fine compactified Jacobian
$\ov{J}_X(\un q)$ admits an Abel map, i.e. there exists $L\in \Pic^{|\un q|+p_a(X)}(X)$ such that $\Im A_L\subseteq \ov{J}_X(\un q)$. Then the universal fine compactified Jacobian $\J_{\X}(\un q)$ with respect to the polarization $\un q$
(as in \S \ref{S:unJac}) satisfies
\begin{equation}\label{E:D-un}
\Pic^o(\ov{J}_{\X}(\un q))=\Pic^{\tau}(\ov{J}_{\X}(\un q)).
\end{equation}
In particular, by restricting to the central fiber, we get $\Pic^o(\ov{J}_{X}(\un q))=\Pic^{\tau}(\ov{J}_{X}(\un q))$.
\end{thm}

\begin{proof}
Since $X$ does not have separating nodes by assumption, the Abel map $A_L:X\to \bJbar_X$ sends a point $p$ into $\m_p\otimes L$, where $\m_p$ is the ideal sheaf of $p\in X$ (see Theorem \ref{T:Abel}\eqref{T:Abel2}).
In other words, if  $\I_{\Delta}$ denotes the ideal sheaf of the diagonal $\Delta\subset X\times X$ and $p_i$ is the projection of $X\times X$ onto the $i$-th factor,  then the map $A_L$ is induced by the sheaf
$\I_{\Delta}\otimes p_1^*(L)$ on $X\times X$, seen as a flat family of simple torsion-free rank-$1$ sheaves on $X$ via the projection $p_2$ (see e.g. \cite[Lemma (8.7)]{AK}).

We will now extend the Abel map $A_L$ to a relative  Abel map over $\Spec R_X$.
First of all, the line bundle $L$ on $X$ can be extended to a line bundle $\L$ on the family $\X$.
Indeed, since an obstruction space for the functor of deformations of $L$ is $H^2(X,L\otimes L^{\vee})=H^2(X, \O_X)$
(see e.g. \cite[Thm. 8.5.3(b)]{FGA}) and since this latter group is zero because $X$ is a curve, we get that $L$ can be
extended to a line bundle $\ov{\L}$ on the formal semiuniversal deformation $\ov{\X}\to \Spf R_X$ of $X$.
However, by Grothendieck's algebraization theorem for coherent sheaves (see \cite[Thm. 8.4.2]{FGA}), the line bundle $\ov{\L}$ is the completion of a line bundle $\L$ on the
effective semiuniversal deformation family $\pi: \X\to \Spec R_X$ of $X$.
By construction, the restriction of $\L$ to the central fiber of $\pi$ is isomorphic to the line bundle $L$ on $X$.

Consider now the sheaf $\I_{\Delta^{\rm un}}\otimes p_1^*(\L)$ on $\X\times_{\Spec R_X} \X$, where  $\I_{\Delta^{\rm un}}$ denotes the ideal sheaf of the diagonal
$\Delta^{\rm un} \subset \X\times_{\Spec R_X} \X$ and $p_i$ is the projection of $\X\times_{\Spec R_X} \X$ onto the $i$-th factor. Via the projection $p_2$, we can regard $\I_{\Delta^{\rm un}}\otimes p_1^*(\L)$ as a flat family of torsion-free rank-$1$ sheaves (see e.g. \cite[Lemma (8.7)]{AK}). Moreover, since the geometric fibers of $p_2$ do not have separating nodes by Corollary  \ref{C:sep-point} above,  the pull-back of $\I_{\Delta^{\rm un}}\otimes p_1^*(\L)$ to the geometric fibers of $p_2$ is simple
(see e.g. \cite[Example 38]{est1}).
It follows that $\I_{\Delta^{\rm un}}\otimes p_1^*(\L)$ is a flat family of simple torsion-free rank-$1$ sheaves via the projection $p_2$; hence it defines a relative Abel map
over $\Spec R_X$
\begin{equation}\label{E:globAb}
A_{\L}: \X  \longrightarrow \bJbar_{\X},
\end{equation}
which, by construction,  restricts on each geometric fiber over $s\in \Spec R_X$ to the Abel map $A_{\L_s}:\X_{\ov s}\to \bJbar_{\X_{\ov s}}$
associated to $\L_s:=\L_{|\X_{\ov s}}$. In particular, the restriction of $A_{\L}$ to the closed point $[\m_X]\in \Spec R_X$ is equal to the Abel map $A_L$.
Since $A_L$ takes values in $\ov{J}_X(\un q)$ by hypothesis and $\J_{\X}(\un q)$ is open in $\bJbar_{\X}$, the map $A_{\L}$ takes values in $\ov{J}_{\X}(\un q)$, or in other words we get a relative Abel map
\begin{equation}\label{E:global-Abel}
A_{\L}: \X\to \ov{J}_{\X}(\un q).
\end{equation}

The pull-back morphism $A_{\L}^*: \Pic(\ov{J}_{\X}(\un q))\to \Pic(\X)=\bJ_{\X}$ between the two relative Picard schemes (whose existence is guaranteed by
Fact \ref{F:univ-Jac}\eqref{F:univ-Jac1} and Theorem \ref{T:Pic-univ}\eqref{T:Pic-univ1}) clearly sends $\Pic^o(\ov{J}_{\X}(\un q))$, which exists by Theorem
\ref{T:Pic-univ}\eqref{T:Pic-univ4} and Corollary B, into the universal generalized Jacobian $J(\X)=\Pic^o(\X)$ of $\X$, which exists by Fact \ref{F:ungenJac}. We denote by
$A_{\L}^{*,o}:\Pic^o(\ov{J}_{\X}(\un q))\to J(\X)$ the induced homomorphism of group schemes.
Consider now the   composition
\begin{equation}\label{E:compmor}
A_{\L}^{*,o}\circ \beun: J(\X)\to J(\X),
\end{equation}
where $\beun:J(\X)\to \Pic^o(\ov{J}_{\X}(\un q))$ is defined in  \eqref{E:univ-beta}.
Proposition \ref{P:prop-AL} implies that $A_{\L}^{*,o}\circ \beun$ is an isomorphism on each geometric fiber of $J(\X)\to \Spec R_X$; hence, the same is true on each fiber since the property  of being an isomorphism is invariant under faithfully flat base change  (see \cite[(2.7.1)]{EGAIV2}).
% the morphism $A_{\L}^{*,o}\circ \beun$ is an isomorphism on each fiber of $J(\X)\to \Spec R_X$.
Since the morphism $J(\X)\to \Spec R_X$ is flat (see Fact \ref{F:ungenJac}), we deduce that $A_{\L}^{*,o}\circ \beun$ is an isomorphism by \cite[(17.9.5)]{EGAIV4}, or in other words that
$A_{\L}^{*,o}$ is a left inverse of  $\beun$. Since $\beun$ is an isomorphism by Theorem \ref{T:isobetauniv} and Corollary B,
we get that $A_{\L}^{*,o}:\Pic^o(\ov{J}_{\X}(\un q))\to J(\X)$ is an isomorphism of group schemes.

The pull-back morphism $A_{\L}^*$ also sends $\Pic^{\tau}(\ov{J}_{\X}(\un q))$, which exists by Theorem \ref{T:Pic-univ}\eqref{T:Pic-univ2}, into the generalized Jacobian
$J(\X)$ of $\X$ since it is well known  that for the family of curves $\pi:\X\to \Spec R_X$ we have that $\Pic^{\tau}(\X)=\Pic^o(\X)=J(\X)$ (see e.g. \cite[Ex. 9.6.21]{FGA}).
Therefore, we get that the induced homomorphism $A_{\L}^{*,\tau}:\Pic^{\tau}(\ov{J}_{\X}(\un q))\to J(\X)$ is a surjective homomorphism of $\Spec R_X$-group schemes.

Summing up this discussion, we get the following diagram of homomorphisms of group schemes over $\Spec R_X$:
\begin{equation}\label{E:Abel-dia}
\xymatrix{
\Pic^{\tau}(\ov{J}_{\X}(\un q)) \ar@{->>}[rd]^{A_{\L}^{*,\tau}} & \\
& J(\X)=\Pic^{\tau}(\X)=\Pic^o(\X)\\
\Pic^o(\ov{J}_{\X}(\un q)) \ar[ru]_{A_{\L}^{*,o}}^{\cong} \ar@{^{(}->}[uu]^i& \\
}
\end{equation}
where $i$ is an open embedding between two smooth group schemes over $\Spec R_X$ (as it follows from Theorem \ref{T:Pic-univ} and Corollary B).

Consider now  the open subset $U\subseteq \Spec R_X$ (introduced in Lemma \ref{L:morpi}\eqref{L:morpi3}) consisting of all the points $s$ such that
the geometric fiber $\X_{\ov s}$ of the universal family $\pi:\X\to \Spec R_X$ over $s$ is smooth or has a unique singular point that is a node. By Lemma \ref{L:morpi}\eqref{L:morpi3}, the complement of $U$ inside $\Spec R_X$ has codimension at least two.

\un{Claim  1:} The restriction of $A_{\L}^{*,\tau}$ to $U$
$$(A_{\L}^{*,\tau})_{|U}:  \Pic^{\tau}(\ov{J}_{\X}(\un q))_{|U}\to J(\X)_{|U}$$
is an isomorphism. In particular, $A_{\L}^{*,\tau}$ is an isomorphism in codimension one.

Indeed, using the above diagram \eqref{E:Abel-dia}, it is enough to prove that the open embedding $i$ is an isomorphism over $U$ or, in other words,
that:
\begin{equation}\label{E:equa-2Pic}
\Pic^o(\ov{J}_{\X_{\ov s}}(\un q^s))=\Pic^{\tau}(\ov{J}_{\X_{\ov s}}(\un q^s)) \text{  for any } s\in U.
\end{equation}
By the definition of $U$ and Corollary \ref{C:sep-point}, the fiber $\X_{\ov s}$ over a point $s\in U$ can be of two types:
\begin{enumerate}[(i)]
\item \label{E:fibra1} $\X_{\ov s}$ is smooth;
%\item \label{E:fibra2} $\X_{\ov s}$ has two smooth irreducible components $\X_{\ov s}^1$ and $\X_{\ov s}^2$ which meet in a separating node.
\item \label{E:fibra3} $\X_{\ov s}$ is an irreducible curve having a unique singular point that  is a node;
\end{enumerate}
In case \eqref{E:fib1}, $\ov{J}_{\X_{\ov s}}(\un q^s)$ is an abelian variety and the equality \eqref{E:equa-2Pic}  is proved for abelian varieties by Mumford in  \cite[Cor. 2, p. 178]{Mum}.
%Also in case \eqref{E:fib2}, $\ov{J}_{\X_{\ov s}}(\un q^s)$ is an abelian variety by Fact \ref{F:for-smooth}  and the result follows again from the above mentioned result of Mumford.
In case \eqref{E:fibra3}, the equality \eqref{E:equa-2Pic} is due to Esteves-Gagn\'e-Kleiman \cite[Thm. 2.1]{egk}, where the same result is proved
for integral curves with at worst double points.

\vspace{0,2cm}

\un{Claim  2:} $A_{\L}^{*,\tau}$ is an isomorphism.

Indeed,  observe that $J(\X)$ is integral and regular by Fact \ref{F:ungenJac} while  $\Pic^{\tau}(\ov{J}_{\X}(\un q))$ is integral and separated over $\Spec R_X$ by Theorem \ref{T:Pic-univ}.
Therefore,  the same argument used in Claim 2 of the proof of Theorem \ref{T:isobetauniv} gives that $A_{\L}^{*,\tau}$ is an open embedding. Since we know that $A_{\L}^{*,\tau}$ is surjective,
we deduce that $A_{\L}^{*,\tau}$ is an isomorphism, q.e.d.

%VECCHIA DIMOSTRAZIONE
%Indeed, since $A_{\L}^{*,\tau}$ is a birational map between two integral schemes which is an isomorphism in codimension one (by Claim 1) and the codomain is normal and locally factorial (being regular by  Fact \ref{F:ungenJac}), we deduce that $A_{\L}^{*,\tau}$ is a local isomorphism by Van der Waerden's theorem on the purity of the ramification locus (see \cite[(21.12.12)]{EGAIV4}). In particular, $A_{\L}^{*,\tau}$ is quasi-finite. Moreover, since $A_{\L}^{*,\tau}$ is birational (by Claim 1) and separated (which follows from the fact that  $\Pic^{\tau}(\ov{J}_{\X}(\un q))\to \Spec R_X$ is separated by Theorem \ref{T:Pic-univ}\eqref{T:Pic-univ2}, see \cite[Chap. II, Cor. 4.6(e)]{Har}) and the codomain is normal (being regular), we deduce that $A_{\L}^{*,\tau}$ is an open embedding by Zariski's main theorem (see \cite[(4.4.9)]{EGAIII1}). Since we know that $A_{\L}^{*,\tau}$ is surjective, we deduce that $A_{\L}^{*,\tau}$ is an isomorphism, q.e.d.

\vspace{0,2cm}

From diagram \eqref{E:Abel-dia} and Claim 2, we deduce that the open embedding $i$ must be an equality, q.e.d.
%By considering the central fiber of $i$, the theorem now follows.

\end{proof}

In order to prove the general case of Theorem D, we will use the
following result, which allows us to compare two different
universal fine compactified Jacobians of $X$ over the open subset $U\subset \Spec R_X$ considered in Lemma \ref{L:morpi}\eqref{L:morpi3}.
We state and prove it only under the additional assumption that $X$ does not have separating nodes, because this is enough for our purposes and
this assumption simplifies the proof. However, the result still holds for curves with separating nodes.

\begin{lemma}\label{L:comp-unJac}
Let $\un q$ and $\un q'$ be two general polarizations on a curve $X$ with locally planar singularities and without separating nodes.
Let $U$ be the open subset of $\Spec R_X$ consisting of those points $s\in \Spec R_X$ such that $\X_{\ov s}$ has at most a unique
singular point that  is a node (as in Lemma \ref{L:morpi}\eqref{L:morpi3}).
Consider the induced universal fine compactified Jacobians $u:\ov{J}_{\X}(\un q)\to \Spec R_X$ and
$u':\ov{J}_{\X}(\un q')\to \Spec R_X$ (as in Theorem \ref{T:univ-fine}), and set $\ov{J}_{\X}(\un q)_{|U}:=u^{-1}(U)$
and $\ov{J}_{\X}(\un q')_{|U}:=(u')^{-1}(U)$.
Then there exists a line bundle $\L$ on $\X$ such that
the multiplication by $\L$ induces an isomorphism of schemes over $U$:
$$-\otimes \L: \ov{J}_{\X}(\un q)_{|U}\stackrel{\cong}{\longrightarrow} \ov{J}_{\X}(\un q')_{|U}.
$$
\end{lemma}
\begin{proof}
Choose a line bundle $L$ on $X$ of degree  $\deg L=|\un q'|-|\un q|$.
As in the proof of Theorem \ref{T:thmD-Abel}, we can find a line bundle $\L$ on the total space $\X$ of the effective semiuniversal deformation
$\pi:\X\to \Spec R_X$ such that the restriction of $\L$ to the central fiber $X$ of $\pi$ is isomorphic to $L$.
Clearly, the multiplication by $\L$ induces an isomorphism
$$-\otimes \L: \bJbar_{\X}\stackrel{\cong}{\longrightarrow}\bJbar_{\X},$$
the inverse being given by the multiplication by $\L^{-1}$. Since any universal fine compactified Jacobian
is an open subscheme of $\bJbar_X$, in order to conclude the proof it is enough to prove that
for any $s\in U$
\begin{equation}\label{E:inc-fibers}
(-\otimes \L_s)(\ov{J}_{\X_{\ov s}}(\un q^s))= \ov{J}_{\X_{\ov s}}(\un q'^s)
\end{equation}
where $\L_s$ denotes the restriction of $\L$ to the geometric fiber $\X_{\ov s}$ and $\ov{J}_{\X_{\ov s}}(\un q^s)$ (resp.
$\ov{J}_{\X_{\ov s}}(\un q'^s)$) is the geometric fiber of $\ov{J}_{\X}(\un q)$ (resp. $\ov{J}_{\X}(\un q')$)
over $s$ (see Theorem \ref{T:univ-fine}).

By the definition of $U$ and Corollary \ref{C:sep-point}, if $s\in U$ then $\X_{\ov s}$ is  irreducible (either smooth or with a unique node).
Therefore, $\J_{\X_{\ov s}}(\un q)$ (resp. $\J_{\X_{\ov s}}(\un q')$) parametrizes all torsion-free, rank-$1$ sheaves on $X$ of Euler characteristic $|\un q|$ (resp. $|\un q'|$).
Hence  \eqref{E:inc-fibers} follows from the fact that $\deg \L_s=\deg L=|\un q'|-|\un q|$.

\end{proof}

We now prove the general case of Theorem D.

\begin{proof}[Proof of Theorem D]

First of all, we make the following

\vspace{0.1cm}

\un{Reduction:} It is enough to prove Theorem D for a curve $X$ without separating nodes.

\vspace{0.1cm}

Indeed, let $X$ be an arbitrary curve with locally planar singularities and let $Y_1,\ldots, Y_r$ be the separating blocks of $X$ as in \S\ref{S:Abel}.
By Theorem \ref{T:Abel}\eqref{T:Abel1}, every fine compactified Jacobian of $X$ is isomorphic to
$$\J_X(\un q)\cong \prod_{i=1}^r \J_{Y_i}(\un q^i),$$
for some fine compactified Jacobians $\J_{Y_i}(\un q^i)$ of $Y_i$. Observe that $\Pic^{\tau}(\J_X(\un q))$ and $\Pic^{\tau}(\J_{Y_i}(\un q^i))$ are smooth by Theorem \ref{T:Pic-univ}\eqref{T:Pic-univ4} and Corollary B; and similarly for $\Pic^o$. Therefore, we can apply a result of Langer (\cite[Cor. 4.7]{Lan}) which says that
$$\begin{sis}
& \Pic^{\tau}(\J_X(\un q))\cong \prod_{i=1}^r \Pic^{\tau}(\J_{Y_i}(\un q^i)),\\
& \Pic^{o}(\J_X(\un q))\cong \prod_{i=1}^r \Pic^{o}(\J_{Y_i}(\un q^i)).\\
\end{sis}$$
Moreover, since the above isomorphisms are obtained in loc. cit. via the natural box product maps,
the inclusion $\Pic^{o}(\J_X(\un q))\hookrightarrow \Pic^{\tau}(\J_X(\un q))$ is given by the product of the inclusions
$\Pic^{o}(\J_{Y_i}(\un q^i))\hookrightarrow \Pic^{\tau}(\J_{Y_i}(\un q^i))$ on each single factor.
Therefore, if we prove Theorem D for the curves $Y_i$ (that do not have separating nodes), Theorem D will follow also for $X$, q.e.d.

\vspace{0.1cm}

From now on, we assume that $X$ does not have separating nodes.
Consider an arbitrary fine compactified Jacobian $\ov{J}_X(\un q)$ of  $X$. Since $\Pic^o(\ov{J}_X(\un q))$ is an open subscheme of $\Pic^{\tau}(\ov{J}_X(\un q))$ and they are both
of finite type over an algebraically closed field $k$, in order to prove that they are equal, it is sufficient (and necessary) to prove that they have the same $k$-points, i.e. that
\begin{equation}\label{E:equa-points}
\Pic^o(\ov{J}_X(\un q))(k)=\Pic^{\tau}(\ov{J}_X(\un q))(k).
\end{equation}
Consider now the schemes $\Pic^o(\ov{J}_{\X}(\un q))$ and $\Pic^{\tau}(\ov{J}_{\X}(\un q))$, which are smooth over $\Spec R_X$ by Theorem \ref{T:Pic-univ} and Corollary B.
Since $S:=\Spec R_X$ is henselian (because $R_X$ is a complete ring), the canonical reduction maps $\Pic^o(\ov{J}_{\X}(\un q))(S)\to \Pic^o(\ov{J}_X(\un q))(k)$ and
$\Pic^{\tau}(\ov{J}_{\X}(\un q))(S) \to \Pic^{\tau}(\ov{J}_X(\un q))(k)$ from the set of $S$-valued points to the set of $k$-valued points are surjective by \cite[Sec. 2.3, Prop. 5]{BLR}.
Therefore, in order to show the equality \eqref{E:equa-points}, it is enough to show that
\begin{equation}\label{E:equa-lifts}
\Pic^o(\ov{J}_{\X}(\un q))(S)=\Pic^{\tau}(\ov{J}_{\X}(\un q))(S).
\end{equation}
Observe that we have the following natural inclusions
\begin{equation}\label{E:incl-Pic}
\Pic^o(\ov{J}_{\X}(\un q))(S)\subseteq \Pic^{\tau}(\ov{J}_{\X}(\un q))(S) \subseteq \Pic(\ov{J}_{\X}(\un q))(S).
\end{equation}
Note that  ${\mathcal Pic}(S)={\mathcal Pic}(\Spec R_X)=0$ because $R_X$ is a power series ring. Also  the morphism $u:\ov{J}_{\X}(\un q)\to S$ admits a section passing through its smooth
locus (see Theorem \ref{T:univ-Jac}) $J_{\X}(\un q)\to S$  by \cite[Sec. 2.3, Prop. 5]{BLR}. Thus, by  \cite[Sec. 8.1, Prop. 4]{BLR}, we
have a natural identification
\begin{equation}\label{E:sec-lb}
{\mathcal Pic}(\ov{J}_{\X}(\un q))\stackrel{\cong}{\longrightarrow} \Pic(\ov{J}_{\X}(\un q))(S),
\end{equation}
where, as usual (see \S\ref{N:Pic-field}), we denote by ${\mathcal Pic}(\ov{J}_{\X}(\un q))$ the group of line bundles on $\ov{J}_{\X}(\un q)$.

Consider now the open subscheme $\ov{J}_{\X}(\un q)_{|U}:=u^{-1}(U)\subseteq \ov{J}_{\X}(\un q)$, where $U$ is the open subset of $\Spec R_X$ consisting of those points $s\in \Spec R_X$ such that $\X_{\ov s}$ has at most a unique singular point that  is a node (as in Lemma \ref{L:morpi}\eqref{L:morpi3}). The complement of $\ov{J}_{\X}(\un q)_{|U}$ inside $\ov{J}_{\X}(\un q)$
has codimension at least two by Lemma \ref{L:morpi}\eqref{L:morpi3} . Since $\ov{J}_{\X}(\un q)$ is a regular, irreducible and separated scheme by Theorem \ref{T:univ-Jac}, we can apply \cite[Chap. II, Prop. 6.5, Cor. 6.16]{Har} in order to conclude that the natural restriction map induces an isomorphism
 \begin{equation}\label{E:iso-restr}
 {\mathcal Pic}(\ov{J}_{\X}(\un q))\stackrel{\cong}{\longrightarrow} {\mathcal Pic}(\ov{J}_{\X}(\un q)_{|U}).
 \end{equation}
 Moreover, the same argument used to prove \eqref{E:sec-lb} (noticing that ${\mathcal Pic}(U)={\mathcal Pic}(S)=0$) gives that
 \begin{equation}\label{E:sec-lbU}
{\mathcal Pic}(\ov{J}_{\X}(\un q)_{|U})\stackrel{\cong}{\longrightarrow} \Pic(\ov{J}_{\X}(\un q)_{|U})(U).
\end{equation}
 By combining \eqref{E:sec-lb}, \eqref{E:iso-restr} and \eqref{E:sec-lbU}, we get that the following natural restriction map of sections is an isomorphism
\begin{equation}\label{E:restr-U}
{\rm res}: \Pic(\ov{J}_{\X}(\un q))(S) \stackrel{\cong}{\longrightarrow} \Pic(\ov{J}_{\X}(\un q))(U)=\Pic(\ov{J}_{\X}(\un q)_{|U})(U).
\end{equation}
It is clear that ${\rm res}(\Pic^o(\ov{J}_{\X}(\un q))(S))\subseteq\Pic^o(\ov{J}_{\X}(\un q)_{|U})(U)$ and similarly that
${\rm res}(\Pic^{\tau}(\ov{J}_{\X}(\un q))(S))\subseteq\Pic^{\tau}(\ov{J}_{\X}(\un q)_{|U})(U)$.

Consider any general polarization $\un q'$ on $X$ such that the associated fine compactified Jacobian $\ov{J}_X(\un q')$ admits an Abel map in the sense of \S \ref{S:Abel}.   Note that there are plenty of such general polarizations due to Theorem
\ref{T:Abel}\eqref{T:Abel3}. Then the inclusions \eqref{E:incl-Pic} and the isomorphism \eqref{E:restr-U}
hold true also for the polarization $\un q'$. Moreover, Theorem \ref{T:thmD-Abel} (which we can apply since $X$ does not have separating nodes by assumption) implies that
\begin{equation}\label{E:eq-Picotau}
\Pic^o(\ov{J}_{\X}(\un q'))=\Pic^{\tau}(\ov{J}_{\X}(\un q')).
\end{equation}
Lemma \ref{L:comp-unJac} below implies that there exists a line bundle $\L$ on $\X$ inducing an isomorphism
\begin{equation}\label{E:eq-U}
-\otimes \L: \ov{J}_{\X}(\un q)_{|U}\stackrel{\cong}{\longrightarrow} \ov{J}_{\X}(\un q')_{|U}.
\end{equation}

Combining \eqref{E:incl-Pic} and \eqref{E:restr-U} (and their analogues for $\un q'$) together with \eqref{E:eq-Picotau} and \eqref{E:eq-U}, we get the following commutative diagram of (abstract) abelian groups
\begin{equation}\label{E:diag-restr}
\xymatrix{
\Pic^o(\ov{J}_{\X}(\un q))(S) \ar@{^{(}->}[r] \ar[d]& \Pic^{\tau}(\ov{J}_{\X}(\un q))(S)  \ar@{^{(}->}[r] \ar[d]& \Pic(\ov{J}_{\X}(\un q))(S) \ar[d]^{{\rm res}}_{\cong} \\
\Pic^o(\ov{J}_{\X}(\un q)_{|U})(U) \ar@{^{(}->}[r] & \Pic^{\tau}(\ov{J}_{\X}(\un q)_{|U})(U)  \ar@{^{(}->}[r]  & \Pic(\ov{J}_{\X}(\un q)_{|U})(U)  \\
\Pic^o(\ov{J}_{\X}(\un q')_{|U})(U) \ar@{=}[r] \ar[u]^{\cong} & \Pic^{\tau}(\ov{J}_{\X}(\un q')_{|U})(U)  \ar@{^{(}->}[r] \ar[u]^{\cong} & \Pic(\ov{J}_{\X}(\un q')_{|U})(U) \ar[u]_{(-\otimes \L)^*}^{\cong}  \\
\Pic^o(\ov{J}_{\X}(\un q'))(S) \ar@{=}[r] \ar[u]& \Pic^{\tau}(\ov{J}_{\X}(\un q'))(S)  \ar@{^{(}->}[r] \ar[u]& \Pic(\ov{J}_{\X}(\un q'))(S) \ar[u]_{{\rm res'}}^{\cong} \\
}
\end{equation}

Now we need the following crucial

\begin{lemma}\label{L:claim}
Notation as above. For any section $\sigma\in  \Pic(\ov{J}_{\X}(\un q))(S)$,  set $\sigma':=(((-\otimes \L)^*\circ {\rm res'})^{-1} \circ{\rm res})(\sigma)\in  \Pic(\ov{J}_{\X}(\un q'))(S)$. Then it holds that
\begin{equation}\label{E:res-Pico}
\sigma \in \Pic^o(\ov{J}_{\X}(\un q))(S) \Longleftrightarrow \sigma' \in \Pic^o(\ov{J}_{\X}(\un q'))(S).
\end{equation}
\end{lemma}

%\un{Claim :} For any section $\sigma\in  \Pic(\ov{J}_{\X}(\un q))(S)$, if we set $\sigma':=(((-\otimes \L)^*\circ {\rm res'})^{-1} \circ{\rm res})(\sigma)\in  \Pic(\ov{J}_{\X}(\un q'))(S)$, then  we have that
%\begin{equation}\label{E:res-Pico}
%\sigma \in \Pic^o(\ov{J}_{\X}(\un q))(S) \Longleftrightarrow \sigma' \in \Pic^o(\ov{J}_{\X}(\un q'))(S).
%\end{equation}

Using the above Lemma (which will be proved below), we can conclude the proof of Theorem D.
From \cite[Thm. 5.1]{Klei}, it follows that there exists a number $N$ such that any section $\sigma\in \Pic^{\tau}(\ov{J}_{\X}(\un q))(S)$ is such that $\sigma^N\in \Pic^{o}(\ov{J}_{\X}(\un q))(S)$ and similarly for the sections in $\Pic^{\tau}(\ov{J}_{\X}(\un q'))(S)$.
Therefore, from the commutative diagram \eqref{E:diag-restr} and \eqref{E:res-Pico},
we deduce that for a section $\sigma\in  \Pic(\ov{J}_{\X}(\un q))(S)$, if we set $\sigma':=(((-\otimes \L)^*\circ {\rm res'})^{-1} \circ{\rm res})(\sigma) \in  \Pic(\ov{J}_{\X}(\un q'))(S)$ as before, then
 we have that
\begin{equation}\label{E:res-Pictau}
\sigma \in \Pic^{\tau}(\ov{J}_{\X}(\un q))(S) \Longleftrightarrow \sigma' \in \Pic^{\tau}(\ov{J}_{\X}(\un q'))(S).
\end{equation}
An easy chasing in diagram \eqref{E:diag-restr} together with \eqref{E:res-Pico} and \eqref{E:res-Pictau} shows that the required equality \eqref{E:equa-lifts} does hold true and this concludes the proof of Theorem D.

%An easy chasing in diagram \eqref{E:diag-restr} together with \eqref{E:res-Pico} shows that the required equality \eqref{E:equa-lifts} follows if we prove the following equality
%\begin{equation}\label{E:equa-rest}
%\Pic^o(\ov{J}_{\X}(\un q)_{|U})(U)= \Pic^{\tau}(\ov{J}_{\X}(\un q)_{|U})(U).
%\end{equation}
%In order to prove \eqref{E:equa-rest}, consider any general polarization $\un q'$ on $X$ such that the associated fine compactified Jacobian $\ov{J}_X(\un q')$ admits an Abel map, in the sense of Definition \ref{D:ex-Abel-sep}. Note that there are plenty of such general polarizations due to Corollary  \ref{C:exist-Abel}\eqref{C:exist-Abel2}. Theorem \ref{T:thmD-Abel} (see also its proof) implies that $\Pic^o(\ov{J}_{\X}(\un q'))=\Pic^{\tau}(\ov{J}_{\X}(\un q'))$ which, restricting to the open subset $U\subseteq\Spec R_X$, gives that
%\begin{equation}\label{E:eq-Pic-Ab}
%\Pic^o(\ov{J}_{\X}(\un q')_{|U})= \Pic^{\tau}(\ov{J}_{\X}(\un q')_{|U}).
%\end{equation}
%This equality, together with Lemma \ref{L:comp-unJac} below, implies \eqref{E:equa-rest}, which concludes our proof.

\end{proof}

\begin{proof}[Proof of Lemma \ref{L:claim}]

Let us prove the implication $\Rightarrow$ (the other being analogous).
Let $f^{o}:\Pic^{o}(\ov{J}_{\X}(\un q))\rightarrow S$ be the morphism representing the functor  $\Pic^{o}_{u}$.
This morphism is smooth by Theorem \ref{T:Pic-univ}\eqref{T:Pic-univ4} and it has geometrically connected fibers by construction.
We are going to apply the following Lemma \ref{L:def-sect}  (whose proof will be given below)
to  the morphism $f^{o}$ denoting by  $\tau_{1}$ its section $\sigma$
and by $\tau_{2}$ the zero section of the $S$-group scheme $\Pic^{o}(\ov{J}_{\X}(\un q))$.

%Notice that the sections $\tau_1$ and $\tau_2$ are contained in an affine open subset $\mathcal{Z}$ of $\Pic^{o}(\ov{J}_{\X}(\un q))$ since the central fiber $\Pic^o(\ov{J}_X(\un q))$ of $f^{o}$ is a $k$-group scheme, hence quasi-projective (see e.g. \cite[Sec. 6.4, Thm. 1]{BLR}). Let  $Z$ be the  intersection $\mathcal{Z}\cap \Pic^o(\ov{J}_{X}(\un q))$ and set $q_{1}=\tau_{1}(\overline{0})\in Z$ and $q_{2}=\tau_{2}(\overline{0})\in Z$, where $\overline{0}\subset S$ is the closed point.

\begin{lemma}\label{L:def-sect}
%Let $S=\Spec k[[x_1,\ldots,x_r]]$, let $\overline{0}$ be  the unique closed point of $S$  and
Let $f^{o}:\Y\to S=\Spec k[[x_1,\ldots,x_r]]$ be a surjective
smooth morphism  with geometrically connected smooth fibers and
%such that $Y:=(f^{o})^{-1}(\overline{0})$ is connected.
let $\tau_1$ and $\tau_2$ be two sections of $f^{o}$.
%and assume that there exists an open affine subset $\mathcal{Z}=\Spec(A) \subseteq \Y$ containing $q_1:=\tau_1(\overline{0})$ and $q_2:=\tau_2(\overline{0})$.

For any  affine $k$-variety  $V=\Spec(B)$,  set
$S_V:=\Spec{B[[x_1,\ldots,x_r]]}$ and let $\varphi_{V}:S_V\rightarrow V$ be the map induced by the natural inclusion $B\subset B[[x_1,\ldots,x_r]]$.
Let $f^{o}_{V}: S_V\times_S \Y\rightarrow S_V$ be the base change  map of $f^{o}$ via the natural map $\pi_V: S_V\to S$ induced by the inclusion $k[[x_1,\ldots, x_r]]\subset B[[x_1,\ldots,x_r]]$.
%[Notice also that the fiber of the morphism $S_V\to S$ over the unique closed point $\ov 0$ of $S$ is isomorphic to $V$].
%\subset S_V$ induced by the quotient $B[[x_1,\ldots,x_r]]\twoheadrightarrow B$.

Then $\tau_1$ and $\tau_2$ are homotopic in the following sense: there are a sequence of connected smooth affine $k$-varieties $V_1, \ldots ,V_m$, sections $\sigma_i$ of $f^o_{V_i}$ and closed points $q_i,q'_i \in V_i$ such that
\begin{enumerate}
\item $\tau_1=(\sigma_1)_{|\varphi_{V_1}^{-1}(q_1)}$,
\item $(\sigma_i)_{|\varphi_{V_i}^{-1}(q_i')}=(\sigma_{i+1})_{|\varphi_{V_{i+1}}^{-1}(q_{i+1})}$ \: \text{ for any } $i=1,\ldots, m-1$,
\item $(\sigma_m)_{|\varphi_{V_m}^{-1}(q_m')}=\tau_2$.
\end{enumerate}

%Denote by $Z=\Spec(R)$ the regular connected affine scheme $\mathcal{Z}\cap Y$, by
%$\mathbb{A}^{1}$ the affine line $\Spec(k[t])$ over $k$ and let $0$ (resp. $1$) be the point of $\mathbb{A}^1$ defined by the ideal $(t)$ (resp. $(t-1)$).
%Then there exist two sections $\tau_{i,\mathbb{A}^{1}}$ (for $i=1, 2$) of $f^{o}_{\mathbb{A}^{1}}$ and a section $\tau_{Z}$ of $f^{o}_{Z}$ such that
%\begin{enumerate}
%\item{$\tau_1=\tau_{1,\mathbb{A}^{1}|\varphi_{\mathbb{A}^{1}}^{-1}(1)}$,}
%\item{$\tau_{1,\mathbb{A}^{1}|\varphi_{\mathbb{A}^{1}}^{-1}(0)}= \tau_{Z|\varphi_{Z}^{-1}(q_1)}$,}
%\item{$\tau_{Z|\varphi_{Z}^{-1}(q_2)}= \tau_{2,\mathbb{A}^{1}|\varphi_{\mathbb{A}^{1}}^{-1}(0)}$,}
%\item{$\tau_{2,\mathbb{A}^{1}|\varphi_{\mathbb{A}^{1}}^{-1}(1)}= \tau_2$.}
%\end{enumerate}
\end{lemma}

%As in Lemma \ref{L:def-sect}, for any affine $k$-variety $V=\Spec B$ set $S(B):=\Spec B[[x_1,\ldots,x_r]]$ and let $f^{o}_V: S(B)\times_S \Pic^o(\ov{J}_{\X}(\un q)) \to S(B)$ be the map induced by $f^{o}$ by base change, finally let $\varphi_{V}:S(B)\rightarrow V$ be the map induced by the natural inclusion $B\subset B[[x_1,\ldots,x_r]]$.

Since (as observed before) the morphism $u:\ov{J}_{\X}(\un q)\to S$ admits a section passing through its smooth
locus, the same holds for the induced map  $u_{V}: S_V\times_S \ov{J}_{\X}(\un q) \to S_V$ for an affine $k$-variety $V$, hence  by \cite[Sec. 8.1, Prop. 4]{BLR}
every section of $f^{o}_V: \Pic^o(S_V\times_S \ov J_{\X}(\un q))\cong S_V\times_S \Pic^o(\ov{J}_{\X}(\un q)) \to S_V$
is represented by a  line bundle on $S_V\times_S \ov{J}_{\X}(\un q)$  which is unique, up to tensor product with a line bundle from $S_V$.
Therefore, using that the Picard group of $S$ is trivial, Lemma \ref{L:def-sect} implies
the existence of connected smooth affine $k$-varieties $V_i$, points $q_i, q_i'\in V_i$ and line bundles $L_i$ on $S_{V_i}\times_S \J_{\X}(\un q)$ such that
\begin{enumerate}[(1)]
\item $L\cong(L_1)_{|\varphi_{V_1}^{-1}(q_1)\times_{S} \ov{J}_{\X}(\un q)}$,
\item $(L_i)_{|\varphi_{V_i}^{-1}(q_i')\times_{S} \ov{J}_{\X}(\un q)}\cong (L_{i+1})_{|\varphi_{V_{i+1}}^{-1}(q_{i+1})\times_{S} \ov{J}_{\X}(\un q)}$ \: \text{ for any } $i=1,\ldots, m-1$,
\item $(L_m)_{|\varphi_{V_m}^{-1}(q_m')\times_{S} \ov{J}_{\X}(\un q)}\cong \O_{\J_{\X}(\un q)}$.
\end{enumerate}
where $L$ is the line bundle on $\ov{J}_{\X}(\un q)$ corresponding to the section $\sigma$ under the bijection \eqref{E:sec-lb}.

% two line bundles $L_{1}$ and $L_{2}$ on $S(k[t])\times_{S} \ov{J}_{\X}(\un q)$ and of a line bundle  $L_{Z}$ on $S(R)\times_{S} \ov{J}_{\X}(\un q)$, where $R:=\Gamma(Z,\cO_Z)$,  such that
%\begin{enumerate}\label{antonio}
%\item{$L\cong L_{1|\varphi_{\mathbb{A}^{1}}^{-1}(1)\times_{S} \ov{J}_{\X}(\un q)}$,}
%\item{$L_{1|\varphi_{\mathbb{A}^{1}}^{-1}(0)\times_{S} \ov{J}_{\X}(\un q)}\cong L_{Z|\varphi_{Z}^{-1}(q_1)\times_{S} \ov{J}_{\X}(\un q)}$,}
%\item{$L_{Z|\varphi_{Z}^{-1}(q_2)\times_{S} \ov{J}_{\X}(\un q)}\cong L_{2|\varphi_{\mathbb{A}^{1}}^{-1}(0)\times_{S} \ov{J}_{\X}(\un q)}$,}
%\item{$L_{2|\varphi_{\mathbb{A}^{1}}^{-1}(1)\times_{S} \ov{J}_{\X}(\un q)}\cong \mathcal{O}_{\ov{J}_{\X}(\un q)}$,}
%\end{enumerate}
%where $L$ is the line bundle on $\ov{J}_{\X}(\un q)$ corresponding to the section $\sigma$ under the bijection \eqref{E:sec-lb}.

By Lemma \ref{L:comp-unJac},
$\ov{J}_{\X}(\un q)$ and $\ov{J}_{\X}(\un q')$ are isomorphic in codimension $1$ and the same holds for
$S_V\times_{S} \ov{J}_{\X}(\un q)$ and $S_V\times_{S} \ov{J}_{\X}(\un q')$ for any affine $k$-variety $V$, since the natural morphism $S_V\to S$ is flat (by \cite[Thm. 22.3(v)]{Mat}).
Moreover, if $V$ is a smooth affine $k$-variety,  the schemes  $S_V\times_{S} \ov{J}_{\X}(\un q)$ and $S_V\times_{S} \ov{J}_{\X}(\un q')$
are regular. To see this for $S_V\times_{S} \ov{J}_{\X}(\un q)$ (the case of $S_V\times_{S} \ov{J}_{\X}(\un q')$ is analogous),
notice that, since $\ov{J}_{\X}(\un q)$ is proper over $S$, the closure of any point of $S_V\times_{S} \ov{J}_{\X}(\un q)$
contains a closed point $p$ whose residue field is $k$ and such that its projection $\pi_1(p)$ onto $S_V$ belongs to $V\cong (\pi_V)^{-1}(\ov 0)\subset S_V$.
Since regularity is stable under generalization, we only have to check that any such $p$ is a regular point of  $S_V\times_{S} \ov{J}_{\X}(\un q)$.
Since $S_V$ is flat over $S$, the projection $\pi_{2}:S_V\times_{S} \ov{J}_{\X}(\un q)\rightarrow \ov{J}_{\X}(\un q)$ is flat too.
Hence the regularity of $S_V\times_{S} \ov{J}_{\X}(\un q)$ at $p$ follows from the regularity of $\ov{J}_{\X}(\un q)$ and the regularity
of the fiber  $\pi_{2}^{-1}(\pi_{2}(p))\simeq V$ containing the point $p$ (see \cite[Theorem 23.7]{Mat}).
It follows that, if $V$ is a smooth affine $k$-variety,  the schemes $S_V\times_{S} \ov{J}_{\X}(\un q)$ and $S_V\times_{S} \ov{J}_{\X}(\un q')$
are locally factorial and isomorphic in codimension $1$, hence their Picard groups are  isomorphic  functorially with respect to   $V$ (i.e. via isomorphisms which are compatible with
the natural morphism $S_V\to S_{V'}$ induced by a morphism $V\to V'$).

In particular, the line bundles $L_i$ on  $S_{V_i}\times_S \J_{\X}(\un q)$ induce line bundles $L_i'$ on  $S_{V_i}\times_S \J_{\X}(\un q')$
such that
\begin{enumerate}[(1')]
\item $L'\cong(L_1')_{|\varphi_{V_1}^{-1}(q_1)\times_{S} \ov{J}_{\X}(\un q')}$,
\item $(L_i')_{|\varphi_{V_i}^{-1}(q_i')\times_{S} \ov{J}_{\X}(\un q')}\cong (L_{i+1}')_{|\varphi_{V_{i+1}}^{-1}(q_{i+1})\times_{S} \ov{J}_{\X}(\un q')}$ \: \text{ for any } $i=1,\ldots, m-1$,
\item $(L_m')_{|\varphi_{V_m}^{-1}(q_m')\times_{S} \ov{J}_{\X}(\un q')}\cong \O_{\J_{\X}(\un q')}$,
\end{enumerate}
where $L'$ is the line bundle on $\ov{J}_{\X}(\un q')$ corresponding to the section $\sigma'$ under the analogue  of the bijection \eqref{E:sec-lb} for $\ov{J}_{\X}(\un q')$.

%In particular the line bundles  $L_{1}$ and $L_{2}$ on $S(k[t])\times_{S} \ov{J}_{\X}(\un q)$ and the line bundle  $L_{Z}$ on $S(R)\times_{S} \ov{J}_{\X}(\un q)$ induce line bundles $L'_{1}$ and $L'_{2}$ on  $S(k[t])\times_{S} \ov{J}_{\X}(\un q')$ and a line bundle  $L'_{Z}$ on $S(R)\times_{S} \ov{J}_{\X}(\un q')$ such that
%\begin{enumerate}[(1$'$)]
%\item{$L'\cong L'_{1|\varphi_{\mathbb{A}^{1}}^{-1}(1)\times_{S} \ov{J}_{\X}(\un q')}$,}
%\item{$L'_{1|\varphi_{\mathbb{A}^{1}}^{-1}(0)\times_{S} \ov{J}_{\X}(\un q')}\cong L'_{Z|\varphi_{Z}^{-1}(q_{1})\times_{S} \ov{J}_{\X}(\un q')}$,}
%\item{$L'_{Z|\varphi_{Z}^{-1}(q_2)\times_{S} \ov{J}_{\X}(\un q')}\cong L'_{2|\varphi_{\mathbb{A}^{1}}^{-1}(0)\times_{S} \ov{J}_{\X}(\un q')}$,}
%\item{$L'_{2|\varphi_{\mathbb{A}^{1}}^{-1}(1)\times_{S} \ov{J}_{\X}(\un q')}\cong \mathcal{O}_{\ov{J}_{\X}(\un q')}$,}
%\end{enumerate}
%where $L'$ is the line bundle on $\ov{J}_{\X}(\un q')$ corresponding to the section $\sigma'$ under the analogue  of the bijection \eqref{E:sec-lb} for $\ov{J}_{\X}(\un q')$.

Restricting the line bundles $L_i'$ to $V_i\times \ov{J}_X(\un q') \subset S_{V_i}\times_S \ov{J}_X(\un q')$ and considering the isomorphisms (1')-(2')-(3') on the central fiber $\ov{J}_X(\un q')$,
we get that the restriction $L'_{|\ov{J}_X(\un q')}$ of $L'$ to $\ov{J}_X(\un q')$
is algebraically equivalent to the trivial line bundle $\O_{\ov{J}_X(\un q')}$ (because the varieties $V_i$ are connected).
%Restricting the line bundles $L'_{1}$ and $L'_{2}$ to $\Spec k[t]\times \ov{J}_X(\un q')\subset S(k[t])\times_{S} \ov{J}_{\X}(\un q')$ and the line bundle $L'_{Z}$ to $Z\times \ov{J}_X(\un q')\subset S(R)\times_{S} \ov{J}_{\X}(\un q')$ and considering the isomorphisms induced by (1$'$)--(4$'$) on the central fibre $\ov{J}_X(\un q')$ we get that the restriction $L'_{|\ov{J}_X(\un q')}$ of $L'$ to $\ov{J}_X(\un q')$ is algebraically equivalent to the trivial line bundle $\O_{\ov{J}_X(\un q')}$ (because $\mathbb{A}^1$ and $Z$ are connected varieties).
This means that $\sigma'$ sends the unique closed point of $S=\Spec R_X$ to  $\Pic^o(\ov{J}_{\X}(\un q))$.
 Since $\Pic^o(\ov{J}_{\X}(\un q))$ is open in  $\Pic(\ov{J}_{\X}(\un q))$ and the only open subset of $S$ containing
the closed point is the whole $S$, we conclude that
 $\sigma'(s)\in \Pic^o(\ov{J}_{\X_{\ov s}}(\un q'^s))$ for every $s\in S$, or in other words that $\sigma'\in \Pic^o(\ov{J}_{\X}(\un q'))(S)$, which
concludes the proof.

%Restricting  $L'_{i}$ to $\mathbb{A}^{1}\times \ov{J}_{X}(\un q')\subset S(k[t])\times_{S} \ov{J}_{\X}(\un q')$ we get that $L'_{i|1\times \ov{J}_{X}(\un q')}$ and  $L'_{i|0\times \ov{J}_{X}(\un q')}$ belong to the same connected component of the Picard scheme of $\ov{J}_{X}(\un q')$. Analogously, restricting $L'_{Z}$ to  $Z\times \ov{J}_{X}(\un q')\subset S(R)\times_{S} \ov{J}_{\X}(\un q')$ we obtain that $L'_{Z|q_{1}\times \ov{J}_{X}(\un q')}$ and  $L'_{Z|q_{2}\times \ov{J}_{X}(\un q')}$ belong  to the same connected component of the Picard scheme of $\ov{J}_{X}(\un q')$. Finally, by (1$'$), (2$'$), (3$'$) and (4$'$) we have
%\begin{enumerate}
%\item{$L'_{1|1\times \ov{J}_{X}(\un q')}=L'_{|\ov{J}_{X}(\un q')}$,}
%\item{$L'_{1|0\times \ov{J}_{X}(\un q')}=L'_{Z|q_{1}\times \ov{J}_{X}(\un q')}$,}
%\item{$L'_{Z|q_{2}\times \ov{J}_{X}(\un q')}=L'_{2|0\times \ov{J}_{X}(\un q')}$,}
%\item{$L'_{2|0\times \ov{J}_{X}(\un q')}=\mathcal{O}_{\ov{J}_{X}(\un q')}$,}
%\end{enumerate}
%hence $L'_{|\ov{J}_{X}(\un q')}\in \Pic^o(\ov{J}_X(\un q'))(k)$.

\end{proof}

%It remains to prove Lemma \ref{L:def-sect}.  Before proving this Lemma let us explain the statement. The Lemma roughly says that the sections $\tau_{1}$ and $\tau_{2}$ may be connected in a generalized sense by a sequence of three smooth varieties, two copies of ${\mathbb A}^{1}=\Spec (k[t])$ and a copy of $Z=\Spec R$. This connection is only performed  in a generalized sense because $S(k[t])\ne  {\mathbb A}^{1} \times S$ and $S(R)\ne Z\times S$. More precisely, the  composition  of $\tau_{1,{\mathbb A}^{1}}: S(k[t])\rightarrow S(k[t])\times_{S}\mathcal{Y}$ with the projection to $\mathcal{Y}$ gives a morphism $h_1: S(k[t])\rightarrow \mathcal{Y}$ connecting, along ${\mathbb A}^{1}$, the section  $\tau_{1}$ with a preferred section $\hat{\tau}_{1}:={\tau}_{1,{\mathbb A}^{1}|\varphi^{-1}_{{\mathbb A}^{1}}(0)}$ still passing through $q_{1}$. In the same way  the  composition  of $\tau_{Z}: S(R)\rightarrow S(R)\times_{S}\mathcal{Y}$ with the projection to $\mathcal{Y}$ gives a morphism $h_Z: S(R)\rightarrow \mathcal{Y}$ connecting, along $Z$, the section $\hat{\tau}_{1}$ with a preferred section $\hat{\tau}_{2}:={\tau}_{2,{\mathbb A}^{1}|\varphi^{-1}_{{\mathbb A}^{1}}(0)}$ passing through $q_{2}$. Finally the composition $h_2: S(k[t])\rightarrow \mathcal{Y}$ of $\tau_{2,{\mathbb A}^{1}}: S(k[t])\rightarrow S(k[t])\times_{S}\mathcal{Y}$ with the projection to $\mathcal{Y}$ connects $\hat{\tau}_{2}$ and $\tau_{2}$ along  ${\mathbb A}^{1}$.

\begin{proof}[Proof of Lemma \ref{L:def-sect}]

We can simplify the proof by performing two progressive reductions.

\vspace{0.1cm}

\un{Reduction 1:} We can assume that the images of $\tau_1$ and $\tau_2$ are contained in an open affine subset $\cU\subseteq \cY$.

\vspace{0.1cm}

Indeed, for $i=1,2$, let $\cU_i$ be an affine open subset of $\cY$ containing $p_i:=\tau_i(\ov 0)$, where $\ov{0}$ is  the unique closed point of $S$. From the hypothesis it follows that $\cY$ is regular and connected, hence irreducible. Therefore $\cU_1$ and $\cU_2$ must intersect. Pick any point $p_3\in \cU_1\cap \cU_2$ and choose a section $\tau_3$ of $f^o$ such that $\tau_3(\ov 0)=p_3$ (which exists since $f^o$ is a smooth morphism and $R_X$ is strictly Henselian, see \cite[Sec. 2.3, Prop. 5]{BLR}). Clearly, it is enough to prove that $\tau_1$ and $\tau_3$ are homotopic and that $\tau_3$ and $\tau_2$ are homotopic (because being homotopic is an equivalence relation).
The reduction is proved once we observe that the images of $\tau_1$ and $\tau_3$ (resp. of $\tau_3$ and $\tau_2$) are contained in $\cU_1$ (resp. in $\cU_2$), because $\ov 0$ is the unique closed point of $S$.

\vspace{0.1cm}

\un{Reduction 2:} We can assume that $\cY=S_U$ for a  connected smooth $k$-variety $U=\Spec R$ and that $f^{o}:S_U=\Spec R[[x_{1},\ldots, x_{r}]]\rightarrow S=\Spec k[[x_1,\ldots,x_r]]$
 is the map induced by the natural embedding $k[[x_1,\ldots,x_r]]\subset R[[x_{1},\ldots, x_{r}]]$ \footnote{Note that the morphism $f^o$ in the set-up of the Reduction 2 is  not smooth, being not locally of finite type. However, the statement of the Lemma still makes sense.}.

\vspace{0.1cm}

Indeed, using the hypothesis of  Reduction 1,  consider the open affine subset $U:=\cU\cap Y$ of the central fiber $Y:=\cY_{\ov 0}=(f^o)^{-1}(\ov 0)$ of $f^o$. Notice that $U$ is smooth and irreducible since $Y$ is so. Consider the coordinate rings $A:=\Gamma(\cU,\O_{\cU})$ and
$R:=\Gamma(U,\O_U)$ and let $I$ be the ideal of $A$ such that $R=A/I$.
As $S=\Spec k[[x_1,\ldots,x_r]]$, any section $\tau$ of $f^{o}$ factors through $\Spec \hat{A}$ where $\hat{A}$ is the $I$-adic completion of $A$.
Moreover, as $\Spec R$ is a smooth affine scheme over $k$,
 it is  rigid and its infinitesimal deformations are trivial (see \cite[Theorem 1.2.4]{Ser}).
As a consequence, $A/I^{n}\cong \frac{A/I[x_{1},\ldots, x_{r}]}{(x_{1},\ldots,
x_{r})^{n}}$ for any $n\in \mathbb{N}$ and $\hat{A}\cong A/I[[x_{1},\ldots, x_{r}]]\cong
R[[x_{1},\ldots, x_{r}]]$ and the second reduction is proved.

%So it will suffice to prove the thesis of the Lemma with $\Y$ and $\cU$ replaced by $\Spec(R[[x_{1},\ldots, x_{r}]])=S_U$ and assuming that $f^{o}:S_U=\Spec(R[[x_{1},\ldots, x_{r}]])\rightarrow S=\Spec k[[x_1,\ldots,x_r]]$ is the map induced by the natural embedding $k[[x_1,\ldots,x_r]]\subset R[[x_{1},\ldots, x_{r}]]$.

\vspace{0.1cm}

Under the assumptions of Reduction 2, set $V_1=V_3:=\bbA^1=\Spec k[t]$ and $V_2:=U=\Spec R$ and consider the points $q_1:=1, q_1':=0\in \bbA^1$, $q_2:=p_1=\tau_1(\ov 0), q_2':=p_2=\tau_2(\ov 0)\in U$,
$q_3:=0, q_3'=1\in \bbA^1$.  In order to conclude the proof of the Lemma with the above choices, it remains to construct sections
$\tau_{i,\mathbb{A}^{1}}$ and $\tau_{U}$ of, respectively $f_{\bbA^1}^o$ and $f^o_{U}$, such that
\begin{enumerate}
\item{$\tau_1=\tau_{1,\mathbb{A}^{1}|\varphi_{\mathbb{A}^{1}}^{-1}(1)}$,}
\item{$\tau_{1,\mathbb{A}^{1}|\varphi_{\mathbb{A}^{1}}^{-1}(0)}= \tau_{U|\varphi_{U}^{-1}(p_1)}=\text{constant section } p_1$,}
\item{$\tau_{U|\varphi_{U}^{-1}(p_2)}= \tau_{2,\mathbb{A}^{1}|\varphi_{\mathbb{A}^{1}}^{-1}(0)}=\text{constant section } p_2$,}
\item{$\tau_{2,\mathbb{A}^{1}|\varphi_{\mathbb{A}^{1}}^{-1}(1)}= \tau_2$.}
\end{enumerate}

We define $\tau_{U}: S_U \rightarrow S_U\times_{S} S_U$ as the diagonal embedding.
For every closed point $q\in U$, the fiber $\varphi_{U}^{-1}(q)$ is naturally identified with $S$. Using this identification, the definition of
$\tau_{U}$ implies that the section
$\tau_{U|\varphi_{U}^{-1}(q)}:\varphi_{U}^{-1}(q)\rightarrow
\varphi_{U}^{-1}(q)\times_{S} S_U$ is the constant section whose value is $q$.
In other words $\tau_{U|\varphi_{U}^{-1}(q)}$ is the map induced on spectra by the $k$-algebra morphism $g_{q}: R[[x_{1},\ldots,
x_{r}]]\rightarrow k[[x_{1},\ldots, x_{r}]]$
defined by reduction of coefficients of the power series  modulo the ideal of $q$ in $\Spec(R)$.

To define $\tau_{i,\mathbb{A}^{1}}:S_{\bbA^1}\rightarrow S_{\bbA^1}\times_{S} S_U$ (for $i=1,2$),
notice that $\tau_{i}: S\rightarrow  S_U$ is induced by a $k[[x_1,\ldots, x_r]]$-morphism $g_{i}: R[[x_1,\ldots,x_r]] \rightarrow k[[x_1,\ldots,x_r]]$ sending the ideal $J_i$ of $p_{i}$ in $\Spec R[[x_{1},\ldots, x_{r}]]$
to the maximal ideal of $k[[x_1,\ldots,x_r]]$.
Therefore $g_{i}$ factors through the $J_i-$adic completion
$\widehat{R[[x_1,\ldots,x_r]]}$  of $R[[x_1,\ldots,x_r]]$. Since $U$  is smooth over $k$,
there exists a  $k[[x_1,\ldots,x_r]]-$algebra isomorphism
$\widehat{R[[x_1,\ldots,x_r]]}$ $\simeq k[[x_1,\ldots,x_s]]$
 for some $s\ge r$. More precisely, for any regular sequence  $m_{1},\ldots, m_{s-r}$  generating the maximal ideal $J_i\cap R$ defining $p_{i}\in \Spec(R)$, the set $\{x_{1},\ldots,x_{r}, m_{1},\ldots, m_{s-r}\}$ generates both the ideals
$J_i\subset R[[x_1,\ldots,x_r]]$ and $J_i\cap R[x_1,\ldots,x_r] \subset R[x_1,\ldots,x_r]$. As a consequence we have  functorial isomorphisms
$$R[[x_1,\ldots,x_r]]/J_{i}^{n}\simeq R[x_1,\ldots,x_r]/(J_{i}\cap R[x_1,\ldots,x_r])^{n},$$
which induce a $k[[x_1,\ldots,x_r]]-$algebra isomorphism and homeomorphism
between $\widehat{R[[x_1,\ldots,x_r]]}$ and the completion
$\widehat{R[x_1,\ldots,x_r]}_{J_i\cap R[x_1,\ldots,x_r]}$ of
$R[x_1,\ldots,x_r]$ at the maximal ideal $J_i\cap R[x_1,\ldots,x_r]$. As $U=\Spec(R)$ is smooth of dimension $s-r$ over $k$, we also have a    $k[[x_1,\ldots,x_r]]-$algebra isomorphism and homeomorphism
$$\widehat{R[x_1,\ldots,x_r]}_{J_i\cap R[x_1,\ldots,x_r]}\simeq k[[x_1,\ldots,x_s]].$$
Summing up, there exist $k[[x_1,\ldots,x_r]]-$algebra morphisms
$\phi_i: R[[x_1,\ldots,x_r]]\rightarrow k[[x_1,\ldots,x_s]]$
and $\overline{g}_{i}: k[[x_1,\ldots,x_s]]\rightarrow k[[x_1,\ldots,x_r]]$
 such that $g_{i}= \overline{g}_{i}\circ \phi_i$.

Moreover, by construction, $\overline{g}_{i}$ is continuous with respect to the topologies induced by the maximal ideals and
$\overline{g}_{i}(x_{t})=x_{t}$ for $t\le r$, hence it is completely determined by the values $\overline{g}_{i}(x_{t})$ for $t>r$; explicitly, for any
 $b(x_1,\ldots, x_{r},x_{r+1},\ldots, x_{s})\in k[[x_1,\ldots,x_s]]$ we have
$$\overline{g}_{i}(b(x_1,\ldots, x_{r},x_{r+1},\ldots, x_{s}))=
b(x_1,\ldots, x_{r},\overline{g}_{i}(x_{r+1}),\ldots,
\overline{g}_{i}(x_{s})).$$
Using these morphisms we define
$$\phi_{i,\mathbb{A}^{1}}:k[t][[x_1,\ldots,x_r]]\otimes_{k[[x_1,\ldots,x_r]]} R[[x_1,\ldots,x_r]]
\rightarrow k[t][[x_1,\ldots,x_r]]\otimes_{k[[x_1,\ldots,x_r]]}
k[[x_1,\ldots,x_s]]$$
as the unique $k[t][[x_1,\ldots,x_r]]-$morphism induced by $\phi_i$. Next, we define
$$\begin{aligned}
&\overline{g}_{i,\mathbb{A}^{1}}:
k[t][[x_1,\ldots,x_r]]\otimes_{k[[x_1,\ldots,x_r]]} k[[x_1,\ldots,x_s]]  \longrightarrow k[t][[x_1,\ldots,x_r]], \\
&a(t,x_1,\ldots, x_{r})\otimes b(x_1,\ldots, x_{r},x_{r+1},\ldots, x_{s})  \mapsto
a(t,x_1,\ldots, x_{r})\;b(x_1,\ldots,
x_{r},t\overline{g}_{i}(x_{r+1}),\ldots, t\overline{g}_{i}(x_{s})). \end{aligned}$$
Finally, the evaluation of the composition
$\overline{g}_{i,\mathbb{A}^{1}}\circ\phi_{i, \mathbb{A}^{1}}$ modulo $(t-\alpha)$ gets $g_{i}$ for $\alpha=1$ and the reduction of the coefficients of the power series modulo the ideal of $p_i$ in $\Spec(R)$ for $\alpha=0$.
Hence we may choose  $\tau_{i,\mathbb{A}^{1}}$ as the map induced on spectra by $\overline{g}_{i,\mathbb{A}^{1}}\circ\phi_{i, \mathbb{A}^{1}}$.

\end{proof}

\section{Appendix: Hitchin fibration vs compactified Jacobians of spectral curves}
%\label{S:Hitchin}

\noindent

Let $C$ be a fixed connected smooth and projective curve of genus $g$ over an algebraically closed field $k$ and let $L$ be a line bundle on $C$ (often it is convenient to assume that
$L$ has high degree, e.g. $\deg L\geq 2g-2$). Fix a natural number $r\geq 1$ and an integral number $d\in \Z$.

An {\em $L$-twisted Higgs pair}  (or simply a Higgs pair when $L$ is clear from the context) on $C$ is a pair $(E,\phi)$ consisting of a vector bundle $E$ on $C$ and a homomorphism $\phi:E\to E\otimes L$ (called the Higgs field).  The degree (resp. the rank) of a Higgs pair $(E,\phi)$ is the degree $\deg E$ (resp. the rank $\rk E$) of the underlying vector bundle $E$. In the important special case when $L=\omega_C$, an $\omega_C$-twisted Higgs pair is simply called a {\em Higgs bundle}.

The algebraic stack $\M=\M(r,d, L)$ of all $L$-twisted Higgs pairs $(E,\phi)$ on $C$ of rank $r$ and degree $d$ is endowed with a morphism  (called the {\em Hitchin morphism})
\begin{equation}\label{E:fibr-Hitchin}
\begin{aligned}
\H: \M(r,d, L)=\M & \longrightarrow \A=\A(r,L):=\oplus_{i=1}^r H^0(C, L^i) \\
(E,\phi) & \mapsto \H(E,\phi):=(a_1(E,\phi), \ldots, a_r(E,\phi)),
\end{aligned}
\end{equation}
where $L^i=L^{\otimes i}$ is the $i$-th tensor product of $L$ and $a_i(E,\phi):=(-1)^i \Tr(\Lambda^i \phi)\in H^0(C,L^i)$.

The algebraic (Artin) stack $\M$ is not of finite type. In order to obtain a space of finite type (and indeed a variety), one introduces a semistability condition as follows.
A Higgs pair $(E,\phi)$ is called {\em semistable} (resp. {\em stable}) if for all non-trivial proper subsheaves $F\subsetneq E$ that are stable with respect to $\phi$ (i.e. such that $\phi(F)\subseteq F\otimes L$) we have
$$\frac{\deg F}{\rk F}\leq \frac{\deg E}{\rk E} \:\:\: \text{(resp. }<). $$
Observe that, given a Higgs pair $(E,\phi)$,  if $E$ is a semistable (resp. stable) vector bundle then  $(E,\phi)$ is semistable (resp. stable) but the converse is in general false.

The coarse moduli space $M=M(r,d,L)$ of S-equivalence classes of semistable $L$-twisted Higgs pairs $(E,\phi)$ of rank $r$ and degree $d$  has been constructed by N. Hitchin \cite{Hit} for $L=\omega_C$ using analytic methods (namely gauge theory) and later by C. Simpson \cite{Sim} for $L=\omega_C$ and by N. Nitsure \cite{Nit} for an arbitrary $L$, using algebro-geometric methods (namely geometric invariant theory). As proved in \cite{Sim} and \cite{Nit}, the Hitchin fibration \eqref{E:fibr-Hitchin} induces  a flat projective morphism (called the {\em Hitchin fibration}):
\begin{equation}\label{E:fibr-Hitchin2}
\begin{aligned}
H: M(r,d,L)=M & \longrightarrow \A=\A(r,L):=\oplus_{i=1}^r H^0(C, L^i) \\
(E,\phi) & \mapsto (a_1(E,\phi), \ldots, a_r(E,\phi)).
\end{aligned}
\end{equation}
%where $a_i(E,\phi):=(-1)^i \Tr(\Lambda^i \phi)\in H^0(C,L^i)$.

\begin{remark}\label{R:hyperkahler}
In the special case where $L=\omega_C$, $r$ and $d$ are coprime (so that there are no strictly semistable Higgs pairs) and $k=\mathbb{C}$, Hitchin \cite{Hit}  proved  that:
\begin{itemize}
\item  $M=M(r,d,\omega_C)$ is an hyperk\"ahler (non compact) manifold containing, as an open subset, the cotangent bundle of the moduli space of stable (= semistable) vector bundles on $C$ of degree $d$ and rank $r$;
\item  $H$  is an algebraically completely integrable system.
\end{itemize}
This result has been generalized to the case where $L\otimes \omega_C^{-1}$ is effective by F. Bottacin \cite{Bot} and E. Markman \cite{Mar}: in this case, it is shown in loc. cit. that $M=M(r,d,L)$ is endowed with a  Poisson structure  (depending upon the choice of a section of $L\otimes \omega_C^{-1}$) with respect to which $H$ becomes an  algebraically completely integrable system.
\end{remark}

The fibers of the Hitchin morphism $\H$ and of the Hitchin fibration $H$ can be described in terms of compactified Jacobians of spectral curves, as we are now going to explain, following Beauville-Narasimhan-Ramanan \cite{BNR}
and Schaub \cite{Sch}.
Consider the $\PP^1$-fibration $p:P=\PP(\O_C\oplus L^{-1})\to C$  and  let $\O(1)$ be the relatively ample line bundle on $P$. We will denote by $y$ the section of $\O(1)$ whose pushforward via $p$ corresponds to the constant section $(1,0)$ of  the vector bundle $p_*\O(1)=\O_C\oplus L^{-1}$. Similarly, we will denote by $x$ the section of $\O(1)\otimes p^*(L)$ whose pushforward via $p$ corresponds to the constant section $(0,1)$ of the vector bundle $p_*(\O(1)\otimes p^*(L))=(\O_C\oplus L^{-1})\otimes L=L\oplus \O_C$. In other words, $\{y=0\}$ is the section of $p$ (that we call $\infty$-section) corresponding to the surjection $\O_C\oplus L^{-1}\twoheadrightarrow L^{-1}$ and $\{x=0\}$ is the section of $p$ (that we call $0$-section) corresponding to the surjection $\O_C\oplus L^{-1}\twoheadrightarrow \O_C$.
Given $\un a=(a_1,\ldots, a_r)\in \oplus_{i=1}^r H^0(C, L^i)=\A$, the {\em spectral curve} $C_{\un a}$ associated to $\un a$ is the projective (but possibly singular) curve \footnote{In this Appendix, we violate \S\ref{N:curves}: by curve, we mean any projective scheme over algebraically closed field $k$ of pure dimension one, not necessarily reduced nor connected.} inside $P$ given by
%as the zero locus of the following section of $[\O(1)\otimes p^*(L)]^r$:
$$C_{\un a}:=\{x^r+p^*(a_1)\cdot x^{r-1}\cdot y+\ldots+ p^*(a_r)\cdot y^r=0\}\subset P.$$
Via this construction, the affine space $\A$ is identified with the open subset of the complete linear system $|\O(r)\otimes p^*(L^r)|=\PP(p_*(\O(r))\otimes L^r)=\PP(\oplus_{i=0}^r H^0(C, L^i))$ consisting of all curves that do not meet the $\infty$-section $\{y=0\}$.

The arithmetic genus of the spectral curves can be computed as follows. First note that the canonical sheaf of $P=\PP(\O_C\oplus L^{-1})$ is equal to
$\omega_P=\O(-2)\otimes p^*(\omega_C\otimes L)$ (see \cite[Chap. V, Lemma 2.10]{Har}).
Therefore, if we set $\xi:=c_1(\O(1))$ and we denote by $f$  the class of the fiber of $p$ in the N\'eron-Severi group of $P$, the  adjunction formula gives
$$
p_a(C_{\un a})=\frac{\left[c_1(\omega_P)+C_{\un a}\right]\cdot C_{\un a}}{2}+1= \frac{\left[-2\xi+(2g-2-\deg L)f+r\xi+r\deg L f\right]\cdot (r\xi+r\deg L f)}{2}+1=
$$
\begin{equation}\label{E:genus-spec}
= r(g-1)+\binom{r}{2}\deg L+1,
\end{equation}
where we used that  $\xi\cdot f=1$, $\xi^2=-\deg L$ and $f^2=0$.

\vspace{0.1cm}

The spectral curve $C_{\un a}$ can be very singular (although it has locally planar singularities because it is embedded in the smooth surface $P$), and in particular it is not necessarily reduced nor irreducible.  The base $\A$ of the Hitchin morphism admits two notable  open (by \cite[(12.2.4)]{EGAIV3}) subsets $\A^{\rm ell}\subseteq \A^{\rm reg}\subseteq \A $, called respectively the elliptic locus and the regular locus, defined as follows:
$$\begin{aligned}
&\A^{\rm ell}:=\{\un a\in \A\: :\: C_{\un a} \: \text{is integral}\}, \\
& \A^{\rm reg}:=\{\un a\in \A\: :\: C_{\un a} \: \text{is reduced and connected}\}. \\
\end{aligned}
$$
The study of the topology of the Hitchin morphism restricted to the elliptic locus $\A^{\rm ell}$ has played a crucial role in B. C. Ng\^o's proof of the fundamental lemma (see \cite{Ngo1} and \cite{Ngo2}) and, more generally, the study of  the Hitchin morphism over the regular locus $\A^{\rm reg}$ was a crucial ingredient in Chaudouard-Laumon's proof of the weighted fundamental lemma (see \cite{CL1} and \cite{CL2}).

\begin{remark}\label{E:strata-A}
If $L$ is globally generated and non trivial on $C$, then the complete linear system $|\O(1)\otimes p^*(L)|$ on $P$ is globally generated and it defines a morphism from $P$ onto the cone over the image of $C$ via $|L|$ (see \cite[Chap. V, Example 2.11.4]{Har}). Therefore, the complete linear system $|\O(r)\otimes p^*(L^r)|$ on $P$ is globally generated and it is not composed with a pencil (i.e. the image of the associated morphism has dimension greater than one). From this, we deduce that (under the above assumption on $L$):
\begin{itemize}
\item all spectral curves are connected by Bertini's second theorem (see \cite[Chap. III, Exercise 11.3]{Har});
\item the generic spectral curve is smooth if ${\rm char}(k)=0$ by Bertini's first theorem (see \cite[Chap. III, Cor. 10.9]{Har}).
\end{itemize}
In particular, we deduce that under the above assumptions on $L$ and in characteristic zero, the above loci $\A^{\rm ell}$ and $\A^{\rm reg}$ are non empty.
See also \cite[Prop. 2.1]{Mar} where the above two properties are stated (without a proof) under the assumption that $L^{\otimes r}$ is very ample.
\end{remark}

The restriction of the morphism $p:P\to C$ to $C_{\un a}$ is a degree-$r$ finite morphism $\pi_{\un a}:C_{\un a}\to C$, called the {\em spectral cover } associated to $\un a\in \A$. Note that, since the zero sets of $x$ and $y$ are disjoint in $P$, the restriction of the section $y$ to $C_{\un a}$ is everywhere non-zero, which implies that $\O(1)_{|C_{\un a}}=\O_{C_{\un a}}$. Therefore, the restriction $x_{|C_{\un a}}$ of the section $x$ to $C_{\un a}$ can be considered as a section of $\left[p^*(L)\otimes \O(1)\right]_{|C_{\un a}}=p^*(L)_{|C_{\un a}}=\pi_{\un a}^*(L).$

%Note that the curve $C_{\un a}$ need not to be reduced nor irreducible.

\begin{fact}[Spectral correspondence \cite{BNR}--\cite{Sch}]\label{F:corr-spec}
Let $\un a\in \A= \oplus_{i=1}^r H^0(C, L^i)$ and consider the associated degree-$r$ spectral cover $\pi_{\un a}:C_{\un a}\to C$.
\noindent
\begin{enumerate}[(i)]
\item \label{F:corr-spec1} There is a bijective correspondence
$$\Pi:\left\{\begin{matrix}  \text{Torsion-free rank-1 sheaves $I$ on } C_{\un a} \\ \end{matrix} \right\} \rightarrow
\left\{\begin{matrix} \text{$L$-twisted Higgs pairs } (E,\phi) \text{ on } C \: \text{of rank } r \\
 \text{ such that }  \H(E,\phi)=\un a \end{matrix} \right\}=\H^{-1}(\un a)
$$
obtained by associating to a torsion-free rank-1 sheaf $I$ on $C_{\un a}$ the $L$-twisted Higgs pair $\Pi(I)=(E,\phi)$ on $C$ consisting of the vector bundle $E:=(\pi_{\un a})_*(I)$ on $C$ together with the Higgs field $\phi:(\pi_{\un a})_*(I)\to  (\pi_{\un a})_*(I)\otimes L=(\pi_{\un a})_*(I\otimes \pi_{\un a}^*(L))$ given by the multiplication with $x_{|C_{\un a}}\in H^0(C_{\un a},\pi_{\un a}^*(L))$.

Moreover, $\chi(I)=\deg \Pi(I)+r(1-g)$.
%Moreover, the degree of $\Pi(I)$ is equal to $\displaystyle \deg \Pi(I)= \deg I-\binom{r}{2}\deg L$.
\item \label{F:corr-spec2} Given a torsion-free rank-1 sheaf $I$ on $C_{\un a}$, then $\Pi(I)$ is a semistable (resp. stable)  $L$-twisted Higgs pair on $C$ if and only if for any subscheme $Z\subseteq C_{\un a}$ of pure dimension one we have that
\begin{equation}\label{E:stabHig}
\chi(I_Z)\geq \chi(I) \frac{\deg ({\pi_{\un a}}_{|Z})}{r}  \: \: \: (\text{resp. }> ),
\end{equation}
where $\deg ({\pi_{\un a}}_{|Z})$ is the degree of the finite morphism ${\pi_{\un a}}_{|Z}: Z\to C$.

In particular, if $C_{\un a}$ is reduced and connected then $\Pi(I)$ is a semistable (resp. stable)  $L$-twisted Higgs pair of degree $d$ on $C$ if and only if $I$ is semistable (resp. stable) with respect to the polarization $\un q^{\un a}$ on $C_{\un a}$ given by
\begin{equation}\label{E:pola-cover}
\un q^{\un a}_{ Z}:=[d+r(1-g)] \frac{\deg ({\pi_{\un a}}_{|Z})}{r}.
\end{equation}
\item \label{F:corr-spec3} Assume that $\un a\in \A^{\rm reg}$ (i.e. $C_{\un a}$ is a reduced connected curve) and that $\un q^{\un a}$ is a general polarization in the sense of Definition \ref{def-int} (i.e. every torsion-free rank-1 sheaf on $C_{\un a}$ which is  $\un q^{\un a}$-semistable is also $\un q^{\un a}$-stable).
%in the sense of Definition \ref{def-int} (which is always the case if $r$ and $d$ are coprime).
Then the bijective correspondence $\Pi$ gives rise to an isomorphism
\begin{equation}\label{E:fibersH-Jac}
\ov{J}_{C_{\un a}}(\un q^{\un a})\xrightarrow[\cong]{\Pi}H^{-1}(\un a),
\end{equation}
where $\ov{J}_{C_{\un a}}(\un q^{\un a})$ is the fine compactified Jacobian of the reduced curve $C_{\un a}$ with respect to the general polarization $\un q^{\un a}$
(see \S  \ref{S:fine-Jac}).
\end{enumerate}
\end{fact}
\begin{proof}
Part \eqref{F:corr-spec1} is proved in \cite[Prop. 3.6]{BNR} under the hypothesis that the spectral curve $C_{\un a}$ is integral and in \cite[Prop. 2.1]{Sch} (see also \cite[\S 6]{HP}) for an arbitrary spectral curve. The last assertion follows from Riemann-Roch
$$\chi(I)=\chi((\pi_{\un a})_*(I))=\deg (\pi_{\un a})_*(I)+r(1-g).$$
%degree of $I$ is related to the degree $\deg E=\deg\left( (\pi_{\un a})_*(I)\right)=d$ by mean of the following
%$$\deg I=\chi(I)-\chi(\O_{C_{\un a}})= \chi(E)-(1-p_a(C_{\un a}))= d+r(1-g)+\left( r(g-1)+\binom{r}{2}\deg L\right)= d+\binom{r}{2}\deg L,$$
%where we used the formula \eqref{E:genus-spec} for $p_a(C_{\un a})$ together with the fact that $\chi(I)=\chi((\pi_{\un a})_*(I))$.

Part \eqref{F:corr-spec2}: from the proof of \cite[Thm. 3.1]{Sch} it follows that $\Pi(I)$ is a semistable (resp. stable) $L$-twisted Higgs pair if and only if for any subscheme $Z\subseteq C_{\un a}$ of pure dimension one we have that
\begin{equation}\label{E:Schaub1}
\frac{\deg [({\pi_{\un a}}_{|Z})_*(I_{Z})]}{\rk [({\pi_{\un a}}_{|Z})_*(I_{Z})]}\geq \frac{\deg[ ({\pi_{\un a}})_*(I)]}{\rk[ ({\pi_{\un a}})_*(I)]}=\frac{d}{r} \: \: \: (\text{resp. }> ),
\end{equation}
where $I_{Z}$ is the biggest torsion-free quotient of the restriction $I_{|Z}$ of $I$ to $Z$. The sheaf ${(\pi_{\un a}}_{|Z})_*(I_{Z})$ is locally free of rank equal to the degree $\deg ({\pi_{\un a}}_{|Z})$ of the finite morphism ${\pi_{\un a}}_{|Z}:Z \to C$ and
its degree can be computed using Riemann-Roch:
$$\deg [ ({\pi_{\un a}}_{|Z})_*(I_{Z})]=\chi [({\pi_{\un a}}_{|Z})_*(I_{Z})] - \rk [ ({\pi_{\un a}}_{|Z})_*(I_{Z}) ](1-g)=\chi(I_{Z}) -\deg ({\pi_{\un a}}_{|Z}) (1-g).
$$
Therefore, the inequality \eqref{E:Schaub1} is equivalent to
\begin{equation}\label{E:Schaub2}
\chi(I_Z)\geq \left[\frac{d}{r}+1-g \right] \deg ({\pi_{\un a}}_{|Z})=\chi(I) \frac{\deg ({\pi_{\un a}}_{|Z})}{r}  \: \: \: (\text{resp. }> ),
\end{equation}
which concludes the proof of  \eqref{F:corr-spec2}.

Part \eqref{F:corr-spec3} follows from \eqref{F:corr-spec2} and the fact that the bijective correspondence $\Pi$ of \eqref{F:corr-spec1} does hold in  families as well (see \cite[Prop. 5.1]{Sch}).

\end{proof}

\begin{remark}\label{R:gen-pola}
Note that if  $d$ and $r$ are coprime, then for every $\un a\in \A^{\rm reg}$ the polarization $\un q^{\un a}$ on $C_{\un a}$ is general since there are no strictly semistable Higgs pairs, hence no strictly semistable torsion-free rank-1 sheaves on $C_{\un a}$ by Fact \ref{F:corr-spec}\eqref{F:corr-spec2}.

In the general case, Chaudouard-Laumon \cite{CL1} have introduced, after taking a suitable cover of $ \A^{\rm reg}$, alternative semistability conditions on the stack of Higgs pairs over $\A^{\rm reg}$, for which there are no strictly semistable objects.
The moduli spaces defined by these new semistability conditions are such that the fibers of the associated Hitchin fibrations  are always isomorphic to  fine compactified Jacobians.

\end{remark}

Donagi-Pantev conjectured in \cite[Conj. 2.5]{DP} that the stack $\M(r,\omega_C)$ of Higgs bundles  satisfies the following autoduality property, which can be viewed as a ``classical limit" of the (conjectural) geometric Langlands correspondence (see \cite[Sec. 2]{DP} for a discussion of the geometric Langlands correspondence and the passage to the classical limit).

\begin{conj}[Langlands duality for Higgs bundles \cite{DP}]\label{C:Lang-duality}
Let $\M(r,\omega_C)$ be the moduli stack  of Higgs bundles of rank $r$ and let $D^b(\M(r,\omega_C))$ be the bounded derived category of quasi-coherent sheaves on $\M(r,\omega_C)$. There exists a canonical equivalence of categories
$${\Phi}:D^b(\M(r,\omega_C))\to D^b(\M(r,\omega_C))$$
which intertwines the action of the classical limit tensorization functors with the action of the classical limit Hecke functors.
%Then there exists a universal Poincar\'e sheaf $\P$ on $\M\times_{\A}\M$ such that the Fourier-Mukai transform with kernel $\P$
%$${\Phi}^{\P}:D^b(M)\to D^b(M)$$
%s an auto-equivalence of the bounded derived category $D^b(M)$ of quasi-coherent sheaves on $M$.
\end{conj}
See \cite[Sec. 2]{DP} for the definition of the tensorization functors and Hecke functors, together with their classical limits.
%More precisely, the above auto-equivalence ${\Phi}^{\P}$ is expected to intertwine the action of the classical limit tensorization functors with the action of the classical limit Hecke functors (see \cite[Conj. 2.5]{DP} for a precise formulation).
More generally, Donagi-Pantev conjectured in loc. cit. a  Langlands duality between the stack of $G$-Higgs bundles (for $G$ any reductive group) and the stack of ${}^{L}G$-Higgs bundles, where ${}^{L}G$ is the Langlands dual of $G$.  Conjecture \ref{C:Lang-duality} is the special case in which $G=\GL_r={}^{L}G$.

The autoequivalence $\Phi$ of Conjecture \ref{C:Lang-duality} is expected to be given by a Fourier-Mukai transform with kernel equal to a universal Poincar\'e sheaf $\P$ on $\M(r,\omega_C)\times_{\A(r,\omega_C)}\M(r,\omega_C)$. Moreover, $\Phi$ is expected to preserve the Hitchin
morphism $\H:\M(r,\omega_C)\to \A(r,\omega_C)$, i.e. for any $\un a\in \A(r,\omega_C)$ the Fourier-Mukai transform with kernel $\P_{\un a}:=\P_{|(\H\times \H)^{-1}(\un a)}$
\begin{equation}\label{E:fibdual}
{\Phi}^{\P_{\un a}}:D^b(\H^{-1}(\un a))\to D^b(\H^{-1}(\un a))
\end{equation}
should be an auto-equivalence of the bounded derived category $D^b(\H^{-1}(\un a))$ of quasi-coherent sheaves on $\H^{-1}(\un a)$.

In \cite[Sec. 5.3, Sec. 5.4]{DP}, Donagi-Pantev proved Conjecture \ref{C:Lang-duality} (and its generalized version for any reductive group $G$) over the open subset $\A^{\rm sm}:=\{\un a\in \A\: :\: C_{\un a} \: \text{ is smooth and } \pi_{\un a}: C_{\un a} \to C \text{ is simply ramified} \}\subset \A$.
  More precisely, if $\un a\in \A^{\rm sm}$ then $\H^{-1}(\un a)$ is isomorphic to the Jacobian of $C_{\un a}$ by Fact \ref{F:corr-spec}, and Donagi-Pantev proved in loc. cit. that the classical Fourier-Mukai transform (introduced by Mukai in \cite{mukai}) intertwines the action of the classical limit tensorization functors with the action of the classical limit Hecke functors.

 If $\un a\in \A^{\rm ell}$, i.e. if the associated spectral curve $C_{\un a}$ is integral, then  $\H^{-1}(\un a)$ is isomorphic to the compactified Jacobian of $C_{\un a}$ by  Fact \ref{F:corr-spec} (no stability conditions are needed in this case to define the compactified Jacobian) and the expected autoequivalence of \eqref{E:fibdual}  is constructed by Arinkin \cite{arin2}.

If $\un a\in  \A^{\rm reg}$, i.e. $C_{\un a}$ is reduced, then the stack $\H^{-1}(\un a)$ of rank-$1$ torsion-free sheaves on $C_{\un a}$ (see Fact \ref{F:corr-spec}\eqref{F:corr-spec1}) contains the fine compactified Jacobians of $C_{\un a}$ as open and proper subsets. Therefore Theorem E of the introduction (whose proof will appear in \cite{MRV2}) can be seen as a first step towards proving the autoequivalence \eqref{E:fibdual} for reduced spectral curves.

\vspace{0,3cm}

\noindent {\bf Acknowledgments.} We are extremely grateful to C. L\'opez-Mart\'in, who shared with us some thoughts at an early stage of this project and provided a faithfulness criterion for integral transforms that was initially used in the proof of Theorem A.
We thank T. Dedieu, F. Esposito, E. Esteves,  L. Migliorini, V. Schende, E. Sernesi  for several useful conversations.
 %on fine compactified Jacobians, T. Dedieu and E. Sernesi for discussions on deformation theory of curves with locally planar singularities and L. Migliorini for discussions on the Hitchin fibration.
We are very grateful to the referee for her/his very careful reading of the paper and for the many suggestions and questions that helped in improving  the presentation and the level of accuracy.  
%In particular, the referee's questions prompted us to give a complete and detailed proof of the properties of the equigeneric stratification stated in Theorem \ref{T:strati}, a result which is well-known to the experts but for which no complete proof (not even over the complex numbers)  seems to be available.  
In particular, the referee pointed out a gap in a previous proof of Theorem \ref{T:Pic-univ} and suggested  a shorter proof of Theorem A, which avoids any faithfulness criterion.

This project started while the first author was visiting the Mathematics Department of the University of Roma ``Tor Vergata'' funded by the ``Michele Cuozzo'' 2009 award. She wishes to express her gratitude to Michele Cuozzo's family and to the Department for this great opportunity.
%providing her the opportunity to .

M. Melo was partially supported by CMUC -- UID/MAT/00324/2013, funded by the Portuguese Government through FCT/MEC and co-funded by the European Regional Development Fund through PT2020 and by a Rita Levi Montalcini Grant from the Italian government.
M. Melo and F. Viviani
were partially supported by  the FCT (Portugal) projects \textit{Geometria de espa\c cos de moduli de curvas e variedades abelianas} (EXPL/MAT-GEO/1168/2013) and \textit{Comunidade Portuguesa de Geometria Alg\'ebrica} (PTDC/MAT-GEO/0675/2012).
A. Rapagnetta and F. Viviani were supported by the MIUR project  \textit{Spazi di moduli e applicazioni} (FIRB 2012).

\bibliographystyle{apsr}

\end{document}